\documentclass[a4paper,10pt]{extarticle}
 
\usepackage{ae,lmodern} 
\usepackage[utf8]{inputenc}
\usepackage[T1]{fontenc}
\usepackage[english]{babel}
\usepackage{geometry}
\usepackage{url,xspace}
\usepackage{amsmath,amssymb,stmaryrd,spalign,bbm,mathtools,mathrsfs,dsfont,thmtools}
\usepackage[thmmarks,amsmath,framed,thref,hyperref]{ntheorem}
\usepackage{comment}
\usepackage[colorlinks=true,citecolor=teal,linkcolor=blue]{hyperref}
\usepackage{array}
\usepackage{enumitem}
\usepackage{changepage}
\usepackage{euscript}
\usepackage{cite}
\usepackage{graphicx}
\usepackage{pgfplots}
\usepackage[toc,page]{appendix}
\usepackage{etoolbox}

\newtheorem{theorem}{Theorem}[section]
    \newtheorem{corollary}[theorem]{Corollary}
    \newtheorem{lemma}[theorem]{Lemma}
    \newtheorem{proposition}[theorem]{Proposition}

    \theorembodyfont{\upshape}

    \newtheorem{remark}[theorem]{Remark}
    
\counterwithin{equation}{section}

\theoremstyle{nonumberplain}
\theoremheaderfont{\itshape}
\theorembodyfont{\normalfont}
\theoremseparator{.\,—}
\theoremsymbol{}
\newtheorem{proof-wo}{Proof}
\theoremsymbol{\ensuremath{\blacksquare}}
\newtheorem{proof}{Proof}

\pgfplotsset{width=10cm,compat=1.9}

\newcommand{\zerodisplayskips}{%
  \setlength{\abovedisplayskip}{5pt}%
  \setlength{\belowdisplayskip}{5pt}%
  \setlength{\abovedisplayshortskip}{5pt}%
  \setlength{\belowdisplayshortskip}{5pt}}
\appto{\normalsize}{\zerodisplayskips}
\appto{\small}{\zerodisplayskips}
\appto{\footnotesize}{\zerodisplayskips}

\setlist[enumerate,1]{label={(\roman*)}}

\newcommand{\bigslant}[2]{{\raisebox{.2em}{$#1$}\left/\raisebox{-.2em}{$#2$}\right.}}

\newcommand{\eus}{\EuScript}

\newcommand{\llb}{\llbracket}
\newcommand{\rrb}{\rrbracket}

\let\div\undefined
\DeclareMathOperator{\div}{div\,}

\DeclareMathOperator{\supp}{supp\,}

\title{On homogeneous Sobolev and Besov spaces on the whole and the half space.}
\author{Anatole \textsc{Gaudin}\thanks{Aix-Marseille Université, CNRS, I2M,  
Marseille, France - \textbf{email:} anatole.gaudin@univ-amu.fr}}

\date{}

\begin{document}

\maketitle

\begin{abstract} In this paper, we propose an elementary construction of homogeneous Sobolev spaces of fractional order on $\mathbb{R}^n$ and $\mathbb{R}^n_+$ in the scope of the treatment of non-linear partial differential equations. This construction extends the construction of homogeneous Besov spaces on $\eus{S}'_h(\mathbb{R}^n)$ started by Bahouri, Chemin and Danchin on $\mathbb{R}^n$. We will also extend the treatment done by Danchin and Mucha on $\mathbb{R}^n_+$, and the construction of homogeneous Sobolev spaces of integer orders started by Danchin, Hieber, Mucha and Tolksdorf on $\mathbb{R}^n$ and $\mathbb{R}^n_+$.
 
Properties of real and complex interpolation, duality, and density are discussed. Trace results are also reviewed. Our approach relies mostly on interpolation theory and yields simpler proofs of some already known results in the case of Besov spaces.

The lack of completeness for our function spaces with high regularities will lead to the consideration of the intersection with a complete space to enforce the behavior of low frequencies. From this point one performs decoupled estimates, in order to obtain results similar to the one obtained in the case of low regularities.

As standard and simple applications, we treat the problems of Dirichlet and Neumann Laplacians in these homogeneous function spaces.
\end{abstract}

\begin{center}
    {\footnotesize \textsc{THIS MANUSCRIPT HAS BEEN ACCEPTED FOR PUBLICATION IN}

    \textit{\textbf{Tunisian Journal of Mathematics}}\textbf{, Vol. 6 (2024), No. 2, 343–404, DOI:} \href{https://doi.org/10.2140/tunis.2024.6.343}{0.2140/tunis.2024.6.343}}
\end{center}

\addtocontents{toc}{\protect\thispagestyle{empty}}
\tableofcontents
\renewcommand{\arraystretch}{1.5}


~


\section{Introduction}

\subsection{Motivations and interests}

We want to give an appropriate construction of homogeneous Sobolev spaces as subspaces of tempered distributions instead of a quotient space of distributions by polynomials. This construction is motivated by the fact that we would like to: give meaning to (para)product laws, be stable under the action of a global diffeomorphism, or even to look at boundary values, and thus verify the precise meaning of traces, when these spaces are restricted to a domain. This could be somewhat difficult if we work with tempered distributions up to polynomials. Indeed, it is not clear that one can perform previous operations in a way that does not depend on a choice of a representative $u+P\in\eus{S}'(\mathbb{R}^n)$ of $[u]\in \bigslant{\eus{S}'(\mathbb{R}^n)}{\mathbb{C}[x]}$. This is inconvenient when it comes to study non-linear partial differential equations, or partial differential equations on a domain with boundary conditions (moreover, applying an extension operator to a polynomial does not give back a polynomial). However, the interested reader could consult, for instance \cite[Chapter~6,~Section~6.3]{BerghLofstrom1976}, \cite[Chapter~5]{bookTriebel1983}, or \cite[Chapter~2,~Section~2.4]{bookSawano2018} for such a construction on the whole space.

Bourdaud and Triebel have highlighted such problems in the choice of such constructions, see their respective work \cite{Bourdaud1988,Bourdaud2013} and \cite[Chapter~2,~Section~2.4]{Triebel2015}.

In the case of the common realization in the quotient structure, for $[u],[v]\in\bigslant{\eus{S}'(\mathbb{R}^n)}{\mathbb{C}[x]}$, and $u+P, v+Q, u+\Tilde{P}, v+\Tilde{Q} \in\eus{S}'(\mathbb{R}^n)$ two representatives of $[u]$ and $[v]$, we have
\begin{align*}
(u+P) (v+Q) - (u+\Tilde{P}) (v+\Tilde{Q}) =& (P- \Tilde{P})v + (Q-\Tilde{Q})u+ P Q - \Tilde{P}\Tilde{Q}.
\end{align*}
Therefore, under the existence of meaningful bilinear (para)product estimates, although $ P Q - \Tilde{P}\Tilde{Q}$ is a polynomial, this is however not the case for $(P- \Tilde{P})v + (Q-\Tilde{Q})u$ so that the product \textbf{depends on the choice of representatives!}

Another possibility to build homogeneous function spaces would be to naively complete the Schwartz class with respect to the homogeneous norm we want to consider. This construction has the disadvantage of producing elements that may no longer be distributions. For example, one can check that
\begin{align*}
\mathrm{C}^\infty_c(\mathbb{R}^n)\not\subset \dot{\mathrm{H}}^{-\frac{n}{2}}(\mathbb{R}^n).
\end{align*}
This prevents us, \textit{a priori}, from identifying elements of $\dot{\mathrm{H}}^{\frac{n}{2}}(\mathbb{R}^n)$ (as a completion) as distributions. This phenomenon is known as \textit{infrared divergence} and relates to convergence problems for the sum of low frequencies in the sense of Fourier.

Thus, for a consistent realization of homogeneous function spaces, we can only choose two out of the following three properties:
\begin{enumerate}[label=\textit{(\roman*)}]
\item functions spaces whose elements are distributions, in a reasonable sense;
\item well-defined product laws;
\item all spaces are complete.
\end{enumerate}

The idea of Bahouri, Chemin and Danchin in \cite[Chapter~2]{bookBahouriCheminDanchin} was to accept the loss of point \textit{(iii)} and to introduce a subspace of $\eus{S}'(\mathbb{R}^n)$ such that we get rid of polynomials, see \cite[Examples,~p.23]{bookBahouriCheminDanchin}. The aforementioned subspace of $\eus{S}'(\mathbb{R}^n)$ is
\begin{align}
    \eus{S}'_h(\mathbb{R}^n) &:= \left\{ u\in \eus{S}'(\mathbb{R}^n)\,\Big{|}\,\forall \Theta \in \mathrm{C}_c^\infty(\mathbb{R}^n),\,  \left\lVert \Theta(\lambda \mathfrak{D}) u \right\rVert_{\mathrm{L}^\infty(\mathbb{R}^n)} \xrightarrow[\lambda\rightarrow+\infty]{} 0\right\}\text{,}
\end{align}
where for $\lambda>0$, $\Theta(\lambda \mathfrak{D})u := \eus{F}^{-1}{\Theta}(\lambda\cdot)\eus{F}u$, $\eus{F}$ being the Fourier transform.

The condition of uniform convergence for low frequencies in the above definition ensures that for $u\in\eus{S}'_h(\mathbb{R}^n)$, the series
\begin{align*}
    \sum_{j\leqslant 0} \dot{\Delta}_j u
\end{align*}
converges in $\mathrm{L}^\infty(\mathbb{R}^n)$, and then, by \cite[Proposition~2.14]{bookBahouriCheminDanchin}, the following equality holds in $\eus{S}'(\mathbb{R}^n)$
\begin{align*}
    u = \sum_{j \in\mathbb{Z}} \dot{\Delta}_j u \text{, }
\end{align*}
where $(\dot{\Delta}_j)_{j\in\mathbb{Z}}$ is the homogeneous Littlewood-Paley decomposition on $\mathbb{R}^n$. With $\eus{S}'_h(\mathbb{R}^n)$ as an ambient space, Bahouri, Chemin and Danchin gave a construction of homogeneous Besov spaces $\dot{\mathrm{B}}^{s}_{p,q}(\mathbb{R}^n)$ which are complete whenever $(s,p,q)\in\mathbb{R}\times(1,+\infty)\times[1,+\infty]$ satisfies
\begin{align*}\tag{$\mathcal{C}_{s,p,q}$}\label{eq:cdtnComplet}
    \left[ s<\frac{n}{p} \right]\text{ or }\left[q=1\text{ and } s\leqslant\frac{n}{p} \right]\text{. }
\end{align*}
Later, this has also led Danchin and Mucha to consider homogeneous Besov spaces on $\mathbb{R}^n_+$ and on exterior domains, see \cite{DanchinMucha2009,DanchinMucha2015}, and Danchin, Hieber, Mucha and Tolksdorf \cite{DanchinHieberMuchaTolk2020} to consider homogeneous Sobolev spaces $\dot{\mathrm{H}}^{m,p}$ on $\mathbb{R}^n$ and $\mathbb{R}^n_+$, for $m\in\mathbb{N}$, $p\in(1,+\infty)$. Each iteration led to various important applications in fluid dynamics, such as Navier-Stokes equations with variable density in \cite{DanchinMucha2009,DanchinMucha2015}, or free boundary problems as in \cite{DanchinHieberMuchaTolk2020}. This highlights the needs of stability under global diffeomorphism, and (para)product laws that do not rely on a choice of a representative up to a polynomial.

The main purpose of this paper is TO NOT use the canonical embedding of $\eus{S}'_h(\mathbb{R}^n)$ in tempered distributions modulo polynomials $\bigslant{\eus{S}'(\mathbb{R}^n)}{\mathbb{C}[x]}$. We want to preserve consistency in the realization of our scales of homogeneous function spaces : hence, to have a unique common representative for each element of our function spaces, independent of the chosen regularity and integrability parameters. This is not possible in the quotient structure involving polynomials, for more details see \cite[Chapter~2,~Section~2.4.3]{bookSawano2018}.

We mention that several attempts to define (homogeneous) Besov spaces have been made over the last years. See for instance the work by Auscher and Amenta \cite{AmentaAuscher2018} for complete operator-adapted realizations over Tent and $Z$-spaces on $\mathbb{R}^n$, and the work of Iwabuchi, Matsuyama and Taniguchi \cite{IwabuchiMatsuyamaTaniguchi2019} for an operator-adapted-construction based on the spectral theory of Schr\"{o}dinger operators on a general openset $\Omega\subset \mathbb{R}^n$. See also the references therein for more historical background about Besov spaces and their operator-adapted counterparts. However, such kind of construction is either constrained to the pure linear theory, or is not compatible with the standard distribution theory.

We want to summarize, complete and extend the given construction of homogeneous Besov spaces in \cite[Chapter~2]{bookBahouriCheminDanchin} and the one of homogeneous Sobolev spaces started in \cite[Chapter~3]{DanchinHieberMuchaTolk2020}.
We are going to discuss in Section \ref{Sect:HomogeneousFunctionspaces} their construction and usual and expected properties, and especially their behavior through complex and real interpolation. The whole space case is treated first, then the case of the half-space will follow.

To be clearer, firstly, in the case of the whole space $\mathbb{R}^n$, we want to check that the real interpolation identity like
\begin{align}
    (\dot{\mathrm{H}}^{s_0,p}(\mathbb{R}^n),\dot{\mathrm{H}}^{s_1,p}(\mathbb{R}^n))_{\theta,q}= \dot{\mathrm{B}}^{(1-\theta)s_0+ \theta s_1}_{p,q}(\mathbb{R}^n)
\end{align}
still makes sense for our $\eus{S}'_h$-realization of homogeneous function spaces, for $s_0,s_1\in\mathbb{R}$, $\theta\in(0,1)$, $p\in(1,+\infty)$, $q\in[1,+\infty]$. This is done in Theorem \ref{thm:InterpHomSpacesRn}. This is the \textbf{first key result} of this paper. We will also check that expected duality results still hold in this framework.

Secondly, we show via some extension-restriction operators, few duality arguments, and interpolation theory, that we still have:
\begin{itemize}
    \item \textbf{Expected density results:} Proposition \ref{prop:densityCcinftyHsp0}, Lemma \ref{lem:densityCcinftyBesov0s+} and Corollaries \ref{cor:densityCcinftyIntersecHsp0} and \ref{cor:Bspq=Bspq0Rn+}.\\
    For $p\in(1,+\infty)$, $q\in[1,+\infty)$, $s>-1+\frac{1}{p}$, when \eqref{eq:cdtnComplet} is satisfied, 
\begin{align}\label{eq:IntroHsp0Bspq0Rn+}
    \dot{\mathrm{H}}^{s,p}_0(\mathbb{R}^n_+)= \overline{\mathrm{C}_c^\infty(\mathbb{R}^n_+)}^{\lVert \cdot \rVert_{\dot{\mathrm{H}}^{s,p}(\mathbb{R}^n)}}\text{,}\quad\text{and}\quad \dot{\mathrm{B}}^{s}_{p,q,0}(\mathbb{R}^n_+)= \overline{\mathrm{C}_c^\infty(\mathbb{R}^n_+)}^{\lVert \cdot \rVert_{\dot{\mathrm{B}}^{s}_{p,q}(\mathbb{R}^n)}}\text{ ;}
\end{align}
We need to make clear now that this is not a definition but a \textbf{consequence} of the definition written at the beginning of Section \ref{sec:FunctionSpacesRn+}.
    \item  \textbf{Expected duality results:} Propositions \ref{prop:DualitySobolevDomain} and \ref{prop:dualityBesovRn+}.\\
    For all $p\in(1,+\infty)$, $q\in(1,+\infty]$, $s>-1+\tfrac{1}{p}$, when \eqref{eq:cdtnComplet} is satisfied, it holds
\begin{align}
    (\dot{\mathrm{H}}^{s,p}(\mathbb{R}^n_+))' = &\dot{\mathrm{H}}^{-s,p'}_0(\mathbb{R}^n_+)\text{, }\, (\dot{\mathrm{B}}^{-s}_{p',q'}(\mathbb{R}^n_+))'=\dot{\mathrm{B}}^{s}_{p,q,0}(\mathbb{R}^n_+) \text{, }\\
    (\dot{\mathrm{H}}^{s,p}_0(\mathbb{R}^n_+))' = &\dot{\mathrm{H}}^{-s,p'}(\mathbb{R}^n_+)\text{, }\,(\dot{\mathrm{B}}^{-s}_{p',q',0}(\mathbb{R}^n_+))'=\dot{\mathrm{B}}^{s}_{p,q}(\mathbb{R}^n_+)\text{. }
\end{align}
    \item \textbf{Expected interpolation results:} Propositions \ref{prop:InterpHomBesSpacesRn+} and \ref{prop:InterpHomSpacesRn+}.\\
    If $(\dot{\mathfrak{h}},\dot{\mathfrak{b}})\in\{(\dot{\mathrm{H}},\dot{\mathrm{B}}), (\dot{\mathrm{H}}_0,\dot{\mathrm{B}}_{\cdot,\cdot,0})\}$, with $(p_0,q_0),(p_1,q_1),(p,q)\in(1,+\infty)\times[1,+\infty]$, $\theta\in(0,1)$, $s_j,s>-1+1/p_j$, $j\in\{0,1\}$, with $s>-1+1/p$, where $s_0,s_1,s$ are three real numbers, so that one can set
\begin{align*}
    \left(s,\frac{1}{p_\theta},\frac{1}{q_\theta}\right):= (1-\theta)\left(s_0,\frac{1}{p_0},\frac{1}{q_0}\right)+ \theta\left(s_1,\frac{1}{p_1},\frac{1}{q_1}\right)\text{, }
\end{align*}
 such that \eqref{eq:cdtnComplet} is satisfied. Then, one has
\begin{align}
    [\dot{\mathfrak{h}}^{s_0,p_0}(\mathbb{R}^n_+),\dot{\mathfrak{h}}^{s_1,p_1}(\mathbb{R}^n_+)]_\theta=\dot{\mathfrak{h}}^{s,p_\theta}(\mathbb{R}^n_+)\text{, }&\qquad (\dot{\mathfrak{b}}^{s_0}_{p,q_0}(\mathbb{R}^n_+),\dot{\mathfrak{b}}^{s_1}_{p,q_1}(\mathbb{R}^n_+))_{\theta,q} = \dot{\mathfrak{b}}^{s}_{p,q}(\mathbb{R}^n_+)\text{, }\\
    (\dot{\mathfrak{h}}^{s_0,p}(\mathbb{R}^n_+),\dot{\mathfrak{h}}^{s_1,p}(\mathbb{R}^n_+))_{\theta,q}= \dot{\mathfrak{b}}^{s}_{p,q}(\mathbb{R}^n_+)\text{, }&\qquad [\dot{\mathfrak{b}}^{s_0}_{p_0,q_0}(\mathbb{R}^n_+),\dot{\mathfrak{b}}^{s_1}_{p_1,q_1}(\mathbb{R}^n_+)]_{\theta} = \dot{\mathfrak{b}}^{s}_{p_\theta,q_\theta}(\mathbb{R}^n_+)\text{. }
\end{align}
\end{itemize}

Note that Proposition \ref{prop:SobolevMultiplier}, telling that $\mathbbm{1}_{\mathbb{R}^n_+}$ yields a bounded multiplication operator on $\dot{\mathrm{H}}^{s,p}(\mathbb{R}^n)$ for all $p\in(1,+\infty)$, $s\in(-1+\frac{1}{p},\frac{1}{p})$, is the \textbf{second key point} of this paper in order to obtain the expected results.
In particular, this will imply the following equalities of homogeneous Sobolev and Besov spaces with equivalent norms
\begin{align}
    \dot{\mathrm{H}}^{s,p}(\mathbb{R}^n_+) = \dot{\mathrm{H}}^{s,p}_0(\mathbb{R}^n_+)\text{, }\, \dot{\mathrm{B}}^{s}_{p,q}(\mathbb{R}^n_+)=\dot{\mathrm{B}}^{s}_{p,q,0}(\mathbb{R}^n_+) \text{. }
\end{align}

Some already existing density and boundedness results in Besov spaces presented here are commonly known, but redone here differently giving some minor improvements with regard to \cite[Chapter~3]{DanchinHieberMuchaTolk2020}, allowing to deal in several cases with $s>-1+\frac{1}{p}$ or $q=+\infty$. Some other results, despite being well known in the construction of usual Sobolev and Besov spaces, are quite new due to the ambient framework. This leads to some new proofs in a different spirit than the ones already available in the literature.

Due to the lack of completeness for homogeneous Sobolev (and Besov) spaces with high regularity exponents, one will need to consider intersection spaces $\dot{\mathrm{H}}^{s_0,p_0}\cap\dot{\mathrm{H}}^{s_1,p_1}$, with $\dot{\mathrm{H}}^{s_0,p_0}$ known to be complete (\textit{i.e.}, $s_0<n/p_0$). Therefore, one will have to check the boundedness of operators with decoupled estimates.

In Section \ref{Sec:TracesofFunctions}, we will review the meaning of traces at the boundary. As an application, in Section~\ref{Sec:DirNeuHalfspace}, we treat the well-posedness of Neumann and Dirichlet Laplacians on the half-space with fine enough behavior of solutions. The “fine enough behavior” have to be understood in the sense that the decay to $0$ at infinity is given in a very precise sense. 

\subsection{Notation, definition, and usual concepts}

Throughout this paper the dimension will be $n\geqslant 1$, and $\mathbb{N}$ will be the set of non-negative integers. For $a,b\in\mathbb{R}$ with $a\leqslant b$, we write $\llb a,b\rrb:=[a,b]\cap\mathbb{Z}$.

For $x\in\mathbb{R}^n$, the (open) ball centered in $x$ of radius $r>0$ is given by
\begin{align*}
    B(x,r):=\{\,y\in\mathbb{R}^n\,|\,\lvert x-y\rvert<r\,\}.
\end{align*}

For two real numbers $A,B\in\mathbb{R}$, $A\lesssim_{a,b,c} B$ means that there exists a constant $C>0$ depending on ${a,b,c}$ such that $A\leqslant C B$. When both $A\lesssim_{a,b,c} B$ and $B \lesssim_{a,b,c} A$ are true, we simply write $A\sim_{a,b,c} B$. When the number of indices is overloaded, we allow ourselves to write $A\lesssim_{a,b,c}^{d,e,f} B$ instead of $A\lesssim_{a,b,c,d,e,f} B$.

\subsubsection*{Spaces of measurable or smooth functions.}

Denote by $\eus{S}(\mathbb{R}^n,\mathbb{C})$ the space of complex valued Schwartz functions, and $\eus{S}'(\mathbb{R}^n,\mathbb{C})$ its dual called the space of tempered distributions. The Fourier transform on $\eus{S}'(\mathbb{R}^n,\mathbb{C})$ is written $\eus{F}$, and is pointwise defined for any $f\in\mathrm{L}^1(\mathbb{R}^n,\mathbb{C})$ by
\begin{align*}
  \eus{F}f(\xi) :=\int_{\mathbb{R}^n} f(x)\,e^{-ix\cdot\xi}\,\mathrm{d}x\text{, } \xi\in\mathbb{R}^n\text{. }
\end{align*}
Additionally, for $p\in[1+\infty]$, we write $p'=\tfrac{p}{p-1}$ its \textit{\textbf{H\"{o}lder conjugate}}.

For any $m\in\mathbb{N}$, the map $\nabla^m\,:\,\eus{S}'(\mathbb{R}^n,\mathbb{C})\longrightarrow \eus{S}'(\mathbb{R}^n,\mathbb{C}^{n^m})$ is defined as $\nabla^m u := (\partial^\alpha u)_{|\alpha|=m}$.
We denote by $(e^{-t(-\Delta)^\frac{1}{2}})_{t\geqslant0}$ the Poisson semigroup on $\mathbb{R}^n$. We also introduce operators $\nabla'$ and $\Delta'$ which are respectively the gradient and the Laplacian on $\mathbb{R}^{n-1}$ identified with the $n-1$ first variables of $\mathbb{R}^n$, \textit{i.e.} $\nabla'=(\partial_{x_1}, \ldots, \partial_{x_{n-1}})$ and $\Delta' = \partial_{x_1}^2 + \ldots + \partial_{x_{n-1}}^2$.

When $\Omega$ is an open set of $\mathbb{R}^n$, $\mathrm{C}_c^\infty(\Omega,\mathbb{C})$ is the set of smooth compactly supported functions in $\Omega$, and $\eus{D}'(\Omega,\mathbb{C})$ is its topological dual. For $p\in[1,+\infty)$, $\mathrm{L}^p(\Omega,\mathbb{C})$ is the normed vector space of complex valued (Lebesgue-) measurable functions whose $p$-th power is integrable with respect to the Lebesgue measure, $\eus{S}(\overline{\Omega},\mathbb{C})$ (\textit{resp.} $\mathrm{C}_c^\infty(\overline{\Omega},\mathbb{C})$) stands for functions which are restrictions on $\Omega$ of elements of $\eus{S}(\mathbb{R}^n,\mathbb{C})$ (\textit{resp.} $\mathrm{C}_c^\infty(\mathbb{R}^n,\mathbb{C})$). Unless the contrary is explicitly stated, we will always identify $\mathrm{L}^p(\Omega,\mathbb{C})$ (resp. $\mathrm{C}_c^\infty(\Omega,\mathbb{C})$) as the subspace of functions in $\mathrm{L}^p(\mathbb{R}^n,\mathbb{C})$ (resp. $\mathrm{C}_c^\infty(\mathbb{R}^n,\mathbb{C})$) supported in $\overline{\Omega}$ through the extension by $0$ outside $\Omega$. $\mathrm{L}^\infty(\Omega,\mathbb{C})$ stands for the space of essentially bounded (Lebesgue-) measurable functions.

For $s\in\mathbb{R}$, $p\in[1,+\infty)$, $\ell^p_s(\mathbb{Z},\mathbb{C})$, stands for the normed vector space of $p$-summable sequences of complex numbers with respect to the counting measure $2^{ksp}\mathrm{d}k$;  $\ell^\infty_s(\mathbb{Z},\mathbb{C})$ stands for sequences $(x_k)_{k\in\mathbb{Z}}$ such that $(2^{ks}x_k)_{k\in\mathbb{Z}}$  is bounded.
More generally, when $X$ is a Banach space, for $p\in[1,+\infty]$, one may also consider $\mathrm{L}^p(\Omega,X)$ which stands for the space of (Bochner-)measurable functions $u\,:\,\Omega\longrightarrow X$, such that $t\mapsto\lVert u(t)\rVert_X \in \mathrm{L}^p(\Omega,\mathbb{R})$, similarly one may consider $\ell^p_s(\mathbb{Z},X)$. Finally, $\mathrm{C}^0(\Omega,X)$ stands for the space of continuous functions on $\Omega\subset \mathbb{R}^n$ with values in $X$. The subspace  $\mathrm{C}^0_b(\mathbb{R},X)$ is made of uniformly bounded continuous functions and  $\mathrm{C}^0_0(\mathbb{R},X)$ is the set of continuous functions that vanish at infinity. For $\mathcal{C}\in\{\mathrm{C}^0,\mathrm{C}^0_b,\mathrm{C}^0_0\}$, we set $\mathcal{C}(\overline{\Omega}, X)$ to be the set of continuous functions on $\overline{\Omega}$ which are restrictions of elements that belongs to $\mathcal{C}(\mathbb{R}^n, X)$.

\subsubsection*{Interpolation of normed vector spaces}

Let $(X,\left\lVert\cdot\right\rVert_X)$ and $(Y,\left\lVert\cdot\right\rVert_Y)$ be two normed vector spaces. We write $X\hookrightarrow Y$ to say that $X$ embeds continuously into $Y$. Now let us recall briefly basics of interpolation theory. If there exists a Hausdorff topological vector space $Z$, such that $X,Y\subset Z$, then $X\cap Y$ and $X+Y$ are normed vector spaces with their canonical norms, and one can define the $K$-functional of $z\in X+Y$, for any $t>0$ by
\begin{align*}
    K(t,z,X,Y) := \underset{\substack{(x,y)\in X\times Y,\\ z=x+y}}{\inf}\left({\left\lVert{x}\right\rVert_{X}+t\left\lVert{y}\right\rVert_{Y}}\right)\text{. }
\end{align*}
This allows us to construct, for any $\theta\in(0,1)$, $q\in[1,+\infty]$, the real interpolation spaces between $X$ and $Y$ with indexes $\theta,q$ as
\begin{align*}
    (X,Y)_{\theta,q} := \left\{\, x\in X+Y\,\Big{|}\,t\longmapsto t^{-\theta}K(t,x,X,Y)\in\mathrm{L}^q_\ast(\mathbb{R}_+)\,\right\}\text{, }
\end{align*}
where $\mathrm{L}^q_\ast(\mathbb{R}_+):=\mathrm{L}^q((0,+\infty),\mathrm{d}t/t)$. The interested reader could check \cite[Chapter~1]{bookLunardiInterpTheory}, \cite[Chapter~3]{BerghLofstrom1976} for more information about real interpolation and its applications.

If moreover we assume that $X$ and $Y$ are complex Banach spaces, one can consider $\mathrm{F}(X,Y)$ the set of all continuous functions $f:\overline{S}\longmapsto X+Y$, $S$ being the strip of complex numbers whose real part is between $0$ and $1$, with $f$ holomorphic in $S$, and such that
\begin{align*}
    t\longmapsto f(it)\in \mathrm{C}^0_b(\mathbb{R},X) \quad\text{ and }\quad t\longmapsto f(1+it)\in \mathrm{C}^0_b(\mathbb{R},Y)\text{. }
\end{align*}
We can endow the space $\mathrm{F}(X,Y)$ with the norm
\begin{align*}
    \left\lVert{f}\right\rVert_{\mathrm{F}(X,Y)}:=\max\left(\underset{t\in\mathbb{R}}{\mathrm{sup}} \left\lVert {f(it)}\right\rVert_{X},\underset{t\in\mathbb{R}}{\mathrm{sup}} \left\lVert {f(1+it)}\right\rVert_{Y}\right)\text{, }
\end{align*}
which makes $\mathrm{F}(X,Y)$ a Banach space since it is a closed subspace of $\mathrm{C}^0(\overline{S},X+Y)$.
Hence, for $\theta\in(0,1)$, the normed vector space given by
\begin{align*}
    [X,Y]_{\theta} &:= \left\{\,f(\theta)\,\big{|}\,f\in \mathrm{F}(X,Y)\,\right\} 
    \text{, }\\
    \left\lVert{x}\right\rVert_{[X,Y]_{\theta}} &:= \underset{\substack{f\in \mathrm{F}(X,Y),\\ f(\theta)=x}}{\inf} \left\lVert{f}\right\rVert_{\mathrm{F}(X,Y)}\text{, }
\end{align*}
is a Banach space called the complex interpolation space between $X$ and $Y$ associated with $\theta$. Again, the interested reader could check \cite[Chapter~2]{bookLunardiInterpTheory}, \cite[Chapter~4]{BerghLofstrom1976} for more information about complex interpolation and its applications.


\section{Homogeneous function spaces on the whole space} \label{Sect:HomogeneousFunctionspaces}

All the function spaces considered here are scalar complex valued. To alleviate the notations during this whole section, we will write $\mathrm{L}^p(\Omega)$ instead of $\mathrm{L}^p(\Omega,\mathbb{C})$, and similarly for any other function spaces: we drop the arrival space $\mathbb{C}$.

\subsection{Definition, usual properties}

To deal with Besov spaces on the whole space, we need to introduce the Littlewood-Paley decomposition given by $\phi\in \mathrm{C}_c^\infty(\mathbb{R}^n)$, radial, real-valued, non-negative,
such that
\begin{itemize}[label={$\bullet$}]
    \item $\supp \phi \subset B(0,4/3)$;
    \item ${\phi}_{|_{B(0,3/4)}}=1$;
\end{itemize}
so we define the following functions for any $j\in\mathbb{Z}$ for all $\xi\in\mathbb{R}^n$,
\begin{align*}
    \phi_j(\xi):=\phi(2^{-j}\xi)\text{, }\qquad \psi_j(\xi):= \phi_{j}(\xi/2)-\phi_{j}(\xi)\text{,}
\end{align*}
and the family $(\psi_j)_{j\in\mathbb{Z}}$ has the following properties
\begin{itemize}[label={$\bullet$}]
    \item $\mathrm{supp}(\psi_j)\subset \{\,\xi\in\mathbb{R}^n\,|\,3\cdot 2^{j-2}\leqslant\left\lvert{\xi}\right\rvert \leqslant 2^{j+3}/3\,\}$;
    \item $\forall\xi\in\mathbb{R}^n\setminus\{0\}$, $\sum\limits_{j=-M}^N{\psi_j}(\xi)\xrightarrow[N,M\rightarrow+\infty]{} 1$.
\end{itemize}
Such a family $(\phi,(\psi_j)_{j\in\mathbb{Z}})$ is called a Littlewood-Paley family. Now, we consider the two following families of operators associated with their Fourier multipliers:
\begin{itemize}[label={$\bullet$}]
    \item The \textit{\textbf{homogeneous}} family of Littlewood-Paley dyadic decomposition operators $(\dot{\Delta}_j)_{j\in\mathbb{Z}}$, where
    \begin{align*}
        \dot{\Delta}_j:= \eus{F}^{-1}\psi_j\eus{F},
    \end{align*}
    \item The \textit{\textbf{inhomogeneous}} family of Littlewood-Paley dyadic decomposition operators $({\Delta}_k)_{k\in\mathbb{Z}}$, where
    \begin{align*}
       {\Delta}_{-1}:= \eus{F}^{-1}\phi\eus{F}\text{, }
    \end{align*}
    $\Delta_k:=\dot{\Delta}_k$ for any $k\geqslant 0$, and $\Delta_k:=0$ for any $k\leqslant-2$.
    \item The $j$-th frequency cut-off operators given for all $j\in\mathbb{Z}$ by
    \begin{align*}
       \dot{S}_j:= \eus{F}^{-1}\phi_j\eus{F} \text{. }
    \end{align*}
\end{itemize}
One may notice, as a direct application of Young's inequality for the convolution, that they are all uniformly bounded families of operators on $\mathrm{L}^p(\mathbb{R}^n)$, $p\in[1,+\infty]$.

Both family of operators lead for $s\in\mathbb{R}$, $p,q\in[1,+\infty]$, $u\in \eus{S}'(\mathbb{R}^n)$ to the following quantities,
\begin{align*}
    \left\lVert u \right\rVert_{\mathrm{B}^{s}_{p,q}(\mathbb{R}^n)}= \left\lVert(2^{ks}\left\lVert {\Delta}_k u \right\rVert_{\mathrm{L}^{p}(\mathbb{R}^n)})_{k\in\mathbb{Z}}\right\rVert_{\ell^{q}(\mathbb{Z})}\text{ and }
    \left\lVert u \right\rVert_{\dot{\mathrm{B}}^{s}_{p,q}(\mathbb{R}^n)} = \left\lVert(2^{js}\left\lVert \dot{\Delta}_j u \right\rVert_{\mathrm{L}^{p}(\mathbb{R}^n)})_{j\in\mathbb{Z}}\right\rVert_{\ell^{q}(\mathbb{Z})}\text{, }
\end{align*}
respectively named the inhomogeneous and homogeneous Besov norms, but the homogeneous norm is not really a norm since $\left\lVert u \right\rVert_{\dot{\mathrm{B}}^{s}_{p,q}(\mathbb{R}^n)}=0$ does not imply that $u=0$. Thus, following \cite[Chapter~2]{bookBahouriCheminDanchin} and \cite[Chapter~3]{DanchinHieberMuchaTolk2020}, we introduce a subspace of tempered distributions such that $\left\lVert \cdot \right\rVert_{\dot{\mathrm{B}}^{s}_{p,q}(\mathbb{R}^n)}$ is point-separating, say
\begin{align*}
   \eus{S}'_h(\mathbb{R}^n) &:= \left\{ u\in \eus{S}'(\mathbb{R}^n)\,\Big{|}\,\forall \Theta \in \mathrm{C}_c^\infty(\mathbb{R}^n),\,  \left\lVert \Theta(\lambda \mathfrak{D}) u \right\rVert_{\mathrm{L}^\infty(\mathbb{R}^n)} \xrightarrow[\lambda\rightarrow+\infty]{} 0\right\}\text{,}
\end{align*}
where for $\lambda>0$, $\Theta(\lambda \mathfrak{D})u = \eus{F}^{-1}{\Theta}(\lambda\cdot)\eus{F}u$. Notice that $\eus{S}'_h(\mathbb{R}^n)$ does not contain any nonzero polynomials, and for any $p\in[1,+\infty)$, $\mathrm{L}^p(\mathbb{R}^n)\subset\eus{S}'_h(\mathbb{R}^n)$.

One can also define the following quantities called the inhomogeneous and homogeneous Sobolev spaces' potential norms
\begin{align*}
    \left\lVert {u} \right\rVert_{\mathrm{\mathrm{H}}^{s,p}(\mathbb{R}^n)}:= \left\lVert {(\mathrm{I}-\Delta)^\frac{s}{2} u} \right\rVert_{\mathrm{L}^{p}(\mathbb{R}^n)}\text{ and } \left\lVert {u} \right\rVert_{\dot{\mathrm{H}}^{s,p}(\mathbb{R}^n)}:= \Big\lVert \sum_{j\in\mathbb{Z}} (-\Delta)^\frac{s}{2}\dot{\Delta}_{j} u  \Big\rVert_{{\mathrm{L}}^{p}(\mathbb{R}^n)}\text{, }
\end{align*}
where $(-\Delta)^\frac{s}{2}$ is understood on $u\in \eus{S}'_h(\mathbb{R}^n)$ by the action on its dyadic decomposition, \textit{i.e.}
\begin{align*}
    (-\Delta)^\frac{s}{2}\dot{\Delta}_j u:= \eus{F}^{-1}(|\cdot|^s\eus{F}\dot{\Delta}_j u)\text{,}
\end{align*}
which gives, \textit{a priori}, a family of $\mathrm{C}^\infty$ functions with at most polynomial growth. Thanks to \cite[Lemma~3.3,~Definition~3.4]{DanchinHieberMuchaTolk2020},
\begin{align*}
    \sum_{j\in\mathbb{Z}} (-\Delta)^\frac{s}{2}\dot{\Delta}_{j} u \in\eus{S}'_h(\mathbb{R}^n)
\end{align*}
holds for all $u\in\eus{S}'_h(\mathbb{R}^n)$, whenever $s\in [0,+\infty)$.

When $u\in\eus{S}'_h(\mathbb{R}^n)$ and $\sum_{j\in\mathbb{Z}} (-\Delta)^\frac{s}{2}\dot{\Delta}_{j} u \in\eus{S}'_h(\mathbb{R}^n)$, for $s\in\mathbb{R}$, one will simply write without distinction,
\begin{align*}
    (-\Delta)^\frac{s}{2}u = \sum_{j\in\mathbb{Z}} (-\Delta)^\frac{s}{2}\dot{\Delta}_{j} u \in\eus{S}'_h(\mathbb{R}^n) \text{, }
\end{align*}
which is somewhat consistent in this case with the fact that $(-\Delta)^\frac{s}{2}\dot{\Delta}_{j} u=\dot{\Delta}_{j} (-\Delta)^\frac{s}{2} u$, $j\in\mathbb{Z}$.

Hence, for any  $p,q\in[1,+\infty]$, $s\in\mathbb{R}$, we define
\begin{itemize}[label={$\bullet$}]
    \item the inhomogeneous and homogeneous Sobolev (Bessel and Riesz potential) spaces,\\
    \resizebox{0.90\textwidth}{!}{$
        \mathrm{\mathrm{H}}^{s,p}(\mathbb{R}^n)=\left\{\, u\in\eus{S}'(\mathbb{R}^n) \,\big{|}\, \left\lVert {u} \right\rVert_{\mathrm{\mathrm{H}}^{s,p}(\mathbb{R}^n)}<+\infty \,\right\}\text{, }\dot{\mathrm{H}}^{s,p}(\mathbb{R}^n)=\left\{\, u\in\eus{S}'_h(\mathbb{R}^n) \,\big{|}\, \left\lVert {u} \right\rVert_{\dot{\mathrm{H}}^{s,p}(\mathbb{R}^n)}<+\infty \,\right\}\text{ ; }$}
    \item and the inhomogeneous and homogeneous Besov spaces,\\
    \resizebox{0.90\textwidth}{!}{$\mathrm{B}^{s}_{p,q}(\mathbb{R}^n)=\left\{\, u\in\eus{S}'(\mathbb{R}^n) \,\big{|}\, \left\lVert {u} \right\rVert_{\mathrm{B}^{s}_{p,q}(\mathbb{R}^n)}<+\infty \,\right\}\text{, }\dot{\mathrm{B}}^{s}_{p,q}(\mathbb{R}^n)=\left\{\, u\in\eus{S}'_h(\mathbb{R}^n) \,\big{|}\, \left\lVert {u} \right\rVert_{\dot{\mathrm{B}}^{s}_{p,q}(\mathbb{R}^n)}<+\infty \,\right\}\text{, }
    $}
\end{itemize}
which are all normed vector spaces. We also introduce the following closures
\begin{align*}
    \mathcal{B}^{s}_{p,\infty}(\mathbb{R}^n) = \overline{\eus{S}(\mathbb{R}^n)}^{\left\lVert {\cdot} \right\rVert_{{\mathrm{B}}^{s}_{p,\infty}(\mathbb{R}^n)}} \text{ and } \dot{\mathcal{B}}^{s}_{p,\infty}(\mathbb{R}^n) = \overline{\eus{S}_0(\mathbb{R}^n)}^{\left\lVert {\cdot} \right\rVert_{\dot{\mathrm{B}}^{s}_{p,\infty}(\mathbb{R}^n)}} \text{. }
\end{align*}

Here $\eus{S}_0(\mathbb{R}^n)$ is defined as
\begin{align*}
    \eus{S}_0(\mathbb{R}^n):=\left\{\, u\in \eus{S}(\mathbb{R}^n)\,\left|\, 0\notin \supp\left(\eus{F}f\right) \,\right.\right\}\text{.}
\end{align*}

The treatment of homogeneous Besov spaces $\dot{\mathrm{B}}^{s}_{p,q}(\mathbb{R}^n)$, $s\in\mathbb{R}$, $p,q\in[1,+\infty]$, defined on $\eus{S}'_h(\mathbb{R}^n)$ has been done in an extensive manner in \cite[Chapter~2]{bookBahouriCheminDanchin}. However, the corresponding construction for homogeneous Sobolev spaces $\dot{\mathrm{H}}^{s,p}(\mathbb{R}^n)$, $s\in\mathbb{R}$, $p\in(1,+\infty)$ has only been done in the case $(p,s)\in(\{2\},\mathbb{R})\cup ((1,+\infty),\mathbb{N})$. See \cite[Chapter~1]{bookBahouriCheminDanchin} for the case $p=2$, \cite[Chapter~3]{DanchinHieberMuchaTolk2020} for the case $s\in\mathbb{N}$.

The inhomogeneous spaces  $\mathrm{L}^p(\mathbb{R}^n)$, ${\mathrm{H}}^{s,p}(\mathbb{R}^n)$, and $\mathrm{B}^{s}_{p,q}(\mathbb{R}^n)$ are all complete for all $p,q\in [1,+\infty]$, $s\in\mathbb{R}$, but in this setting homogeneous function spaces are no longer always complete (see \cite[Proposition~1.34,~Remark~2.26]{bookBahouriCheminDanchin}).

For homogeneous Besov spaces, we have the following properties:
\begin{proposition}\label{prop:PropertiesHomBesovSpacesRn} Let $p,q\in[1,+\infty]$, $s\in\mathbb{R}$. The following assertions hold
\begin{enumerate}
    \item if $(s,p,q)$ satisfies the condition 
    \begin{align*}\tag{$\mathcal{C}_{s,p,q}$}\label{AssumptionCompletenessExponents}
    \left[ s<\frac{n}{p} \right]\text{ or }\left[q=1\text{ and } s\leqslant\frac{n}{p} \right]
\end{align*}
    then $\dot{\mathrm{B}}^{s}_{p,q}(\mathbb{R}^n)$ is a complete normed vector space,
    \item for all $m\in\mathbb{N}$, all $u\in\eus{S}'_h(\mathbb{R}^n)$,
    \begin{align}
    \sum_{j=1}^n \lVert \partial_{x_j}^m u\rVert_{\dot{\mathrm{B}}^{s}_{p,q}(\mathbb{R}^n)} \sim_{s,m,p,n} \lVert \nabla^m u\rVert_{\dot{\mathrm{B}}^{s}_{p,q}(\mathbb{R}^n)} &\sim_{s,m,p,n} \lVert u\rVert_{\dot{\mathrm{B}}^{s+m}_{p,q}(\mathbb{R}^n)}\label{eq:equivNormsBesRn}\text{, }
\end{align}
\item if $p,q<+\infty$, the space $\eus{S}_0(\mathbb{R}^n)$ is dense in $\dot{\mathrm{B}}^{s}_{p,q}(\mathbb{R}^n)$,
\item if $1/q=1/p-s/n\in(0,1)$, we have dense Sobolev embeddings,
\begin{align*}
    \dot{\mathrm{B}}^{s}_{p,r}(\mathbb{R}^n)&\hookrightarrow \mathrm{L}^q(\mathbb{R}^n),\quad (p,q<+\infty,\, r\in[1,q], \,s<n/p)\\
    \dot{\mathrm{B}}^{n/p}_{p,1}(\mathbb{R}^n)&\hookrightarrow \mathrm{C}^0_0(\mathbb{R}^n),\quad ( p<+\infty)
\end{align*}
\item when $s>0$, one has ${\mathrm{B}}^{s}_{p,q}(\mathbb{R}^n)=\mathrm{L}^p(\mathbb{R}^n)\cap\dot{\mathrm{B}}^{s}_{p,q}(\mathbb{R}^n)$ with equivalence of norms and a continuous embedding
\begin{align*}
    \dot{\mathrm{B}}^{-s}_{p,q}(\mathbb{R}^n) \hookrightarrow {\mathrm{B}}^{-s}_{p,q}(\mathbb{R}^n).
\end{align*}
\end{enumerate}
\end{proposition}
see \cite[Theorem~2.25, Lemma~2.1, Propositions~2.27~\&~2.39, Theorems~2.40~\&~2.41]{bookBahouriCheminDanchin} and \cite[Theorem~6.3.2]{BerghLofstrom1976} for more details.

In the case of $\eus{S}'_h$-realizations of homogeneous Sobolev spaces, only few properties are explicitly stated in the literature. We repair this injustice here, for which usual proofs can be adapted (almost) straightforwardly and are well known, and therefore omitted here. Several references and comments are given after the next statement.

\begin{proposition}\label{prop:PropertiesHomSobolevSpacesRn} Let $p\in(1,+\infty)$, $s\in\mathbb{R}$. The following assertions hold:
\begin{enumerate}
    \item if $s\in[0,n/p)$ and $1/q=1/p-s/n\in(0,1)$, we have the standard Sobolev embeddings,
\begin{align*}
    \dot{\mathrm{H}}^{s,p}(\mathbb{R}^n)&\hookrightarrow \mathrm{L}^q(\mathbb{R}^n),\\
    {\mathrm{L}}^{p}(\mathbb{R}^n)&\hookrightarrow \dot{\mathrm{H}}^{-s,q}(\mathbb{R}^n),
\end{align*}
    \item if $(s,p)$ is such that it satisfies
    \begin{align*}\tag{$\mathcal{C}_{s,p}$}\label{AssumptionCompletenessExponentsSobolev}
     s<\frac{n}{p} 
    \end{align*}
    then $\dot{\mathrm{H}}^{s,p}(\mathbb{R}^n)$ is a complete normed vector space and
    \begin{align*}
        (-\Delta)^\frac{s}{2}\,:\, \dot{\mathrm{H}}^{s,p}(\mathbb{R}^n) \longrightarrow {\mathrm{L}}^{p}(\mathbb{R}^n)
    \end{align*}
    is a bijective isometry of Banach spaces,
\item for all $m\in\mathbb{N}$, all $u\in\eus{S}'_h(\mathbb{R}^n)$,
    \begin{align}
    \sum_{j=1}^n \lVert \partial_{x_j}^m u\rVert_{\dot{\mathrm{H}}^{s,p}(\mathbb{R}^n)} \sim_{s,m,p,n} \lVert \nabla^m u\rVert_{\dot{\mathrm{H}}^{s,p}(\mathbb{R}^n)} &\sim_{s,m,p,n} \lVert u\rVert_{\dot{\mathrm{H}}^{s+m,p}(\mathbb{R}^n)}\label{eq:equivNormsSobRn}\text{, }
    \end{align}
\item the space $\eus{S}_0(\mathbb{R}^n)$ is dense in $\dot{\mathrm{H}}^{s,p}(\mathbb{R}^n)$,
\item for all $u\in\eus{S}'_h(\mathbb{R}^n)$, one has the equivalence of norms
\begin{align*}
    \left\lVert{u}\right\rVert_{\dot{\mathrm{H}}^{s,p}(\mathbb{R}^n)}\sim_{s,p,n}\left\lVert{u}\right\rVert_{\dot{\mathrm{F}}^{s}_{p,2}(\mathbb{R}^n)}:= \left\lVert(2^{js} \dot{\Delta}_j u )_{j\in\mathbb{Z}}\right\rVert_{\mathrm{L}^{p}(\mathbb{R}^n,\ell^{2}(\mathbb{Z}))}\text{,}
\end{align*}
\item if $s\geqslant0$, one has ${\mathrm{H}}^{s,p}(\mathbb{R}^n)=\mathrm{L}^p(\mathbb{R}^n)\cap\dot{\mathrm{H}}^{s,p}(\mathbb{R}^n)$ with equivalence of norms and a continuous embedding
\begin{align*}
    \dot{\mathrm{H}}^{-s,p}(\mathbb{R}^n) \hookrightarrow {\mathrm{H}}^{-s,p}(\mathbb{R}^n),
\end{align*}
\item if for some $\theta\in(0,1)$, one has $(s,1/p) = (1-\theta)(s_0,1/p_0)+\theta(s_1,1/p_1)$ for $j\in\{0,1\}$, $s_j\in\mathbb{R}$, $p_j\in(1,+\infty)$, then for all $u\in\eus{S}'_h(\mathbb{R}^n)$, we have
\begin{align*}
    \lVert u\rVert_{\dot{\mathrm{H}}^{s,p}(\mathbb{R}^n)} \lesssim_{p_0,p_1,n}^{s_0,s_1} \lVert u\rVert_{\dot{\mathrm{H}}^{s_0,p_0}(\mathbb{R}^n)}^{1-\theta} \lVert u\rVert_{\dot{\mathrm{H}}^{s_1,p_1}(\mathbb{R}^n)}^{\theta} \text{.}
\end{align*}
\end{enumerate}
\end{proposition}

\begin{remark} One may check first point \textit{(iii)} as a direct consequence of \cite[Lemma~2.6]{DanchinHieberMuchaTolk2020}. One may deduce the point \textit{(i)} from \cite[Theorem~1.2.3]{bookGrafakos2014Modern} and a density argument achieved manually (at this stage homogeneous Sobolev spaces are not known to be complete or not, only Lebesgue spaces are). The point \textit{(ii)} is then a direct consequence of the point \textit{(i)}. The point \textit{(iv)} follows directly from the proof of \cite[Proposition~3.7]{DanchinHieberMuchaTolk2020} for the case $s=0$.

The point \textit{(v)} is a very well known result, based on extensive use of Khintchine's inequality ($\mathrm{L}^p(\mathbb{R}^n)$ square estimates) and the H\"{o}rmander-Mikhlin Fourier multiplier theorem, see for instance \cite[Remark~3,~p.25]{bookTriebel1992} and \cite[Proposition~6.1.2]{bookGrafakos2014Classical} for the case of $\eus{S}'(\mathbb{R}^n)$ when $s=0$. One may adapt the proof, taking care of possible convergence issues (no density argument is \textit{a priori} allowed).

See \cite[Theorem~6.3.2]{BerghLofstrom1976} for the point \textit{(vi)}. The point \textit{(vii)} is just a direct consequence of point \textit{(v)}, applying H\"{o}lder's inequality twice.

All the details can be found in the dissertation of the author \cite[Chapter~2,~Section~2.1]{GaudinThesis2023}.
\end{remark}

\begin{remark} One can make the comparison with homogeneous function spaces defined by tempered distributions quotiented by polynomials:

The point \textit{(i)} of Proposition \ref{prop:PropertiesHomBesovSpacesRn}, the point \textit{(ii)} of Proposition \ref{prop:PropertiesHomSobolevSpacesRn} and \cite[Theorems~2.31~\&~2.32]{bookSawano2018} tell us that all realizations of homogeneous Sobolev and Besov spaces can be isometrically \textbf{identified} whenever $s<n/p$.
\end{remark}

\begin{lemma}\label{lem:IntersecHomHsp}Let $p_j\in(1,+\infty)$, $s_j \in\mathbb{R}$, for $j\in\{0,1\}$. If $(\mathcal{C}_{s_0,p_0})$ is satisfied, then the intersection space $\dot{\mathrm{H}}^{s_0,p_0}(\mathbb{R}^n)\cap \dot{\mathrm{H}}^{s_1,p_1}(\mathbb{R}^n)$ is a Banach space for which $\eus{S}_0(\mathbb{R}^n)$ is dense in it.
\end{lemma}

\begin{proof}The completeness is straightforward : this follows from the completeness of $\dot{\mathrm{H}}^{s_0,p_0}(\mathbb{R}^n)$, so that a Cauchy sequence of $\dot{\mathrm{H}}^{s_0,p_0}(\mathbb{R}^n)\cap \dot{\mathrm{H}}^{s_1,p_1}(\mathbb{R}^n)$ admits a limit in $\dot{\mathrm{H}}^{s_0,p_0}(\mathbb{R}^n)\subset\eus{S}'_h(\mathbb{R}^n)$. From this point, the equivalence of norms given by point \textit{(v)} of Proposition \ref{prop:PropertiesHomSobolevSpacesRn} allows passing to the limit in the norms for free using the Fatou Lemma.

Concerning the claim about density, we follow the proof of \cite[Proposition~2.27]{bookBahouriCheminDanchin} with minor modifications, in order to adapt it to our setting.

For $u\in \dot{\mathrm{H}}^{s_0,p_0}(\mathbb{R}^n)\cap \dot{\mathrm{H}}^{s_1,p_1}(\mathbb{R}^n)$, and fixed $\varepsilon>0$, for $k\in\{ 0,1\}$ there exists $N\in\mathbb{N}$ such that for all $\tilde{N}\geqslant N$
\begin{align*}
    \lVert u - u_{\tilde{N}}\rVert_{\dot{\mathrm{H}}^{s_k,p_k}(\mathbb{R}^n)} < \varepsilon\text{. }
\end{align*}
Here, for any $K\in\mathbb{N}$,
\begin{align*}
    u_K:= \sum_{ \lvert j \rvert \leqslant K} \dot{\Delta}_j u\text{. }
\end{align*}
For $M\in \llb \tilde{N}+1,+\infty\llb$, $R>0$, provided $\Theta\in \mathrm{C}_c^\infty(\mathbb{R}^n)$, real valued, supported in $B(0,2)$, such that $\Theta_{|_{B(0,1)}}=1$, and $\Theta_R:=\Theta(\cdot /R)$, we introduce
\begin{align*}
    u_{\tilde{N},M}^R:= (\mathrm{I}-\dot{S}_{-M})[\Theta_{R}u_{\tilde{N}}]\text{.}
\end{align*}
Since $\dot{\Delta}_{k} u_{\tilde{N}} = 0$, $k\leqslant -M-1$, we have $\dot{S}_{-M}u_{\tilde{N}}=0$, then
\begin{align*}
    u_{\tilde{N},M}^R - u_{\tilde{N}} = (\mathrm{I}-\dot{S}_{-M})[(\Theta_{R}-1)u_{\tilde{N}}]\text{.}
\end{align*}
If one sets $m_k:= \max(0,\lfloor s_k \rfloor +2)$, since $0\notin \supp \eus{F}(u_{\tilde{N},M}^R - u_{\tilde{N}})$ by construction, we apply \cite[Theorem~6.3.2]{BerghLofstrom1976} and decreasing embedding of inhomogeneous Sobolev spaces to deduce
\begin{align*}
    \lVert u_{\tilde{N},M}^R - u_{\tilde{N}}\rVert_{\dot{\mathrm{H}}^{s_k,p_k}(\mathbb{R}^n)} &\lesssim_{M,s_k,p_k} \lVert u_{\tilde{N},M}^R - u_{\tilde{N}}\rVert_{{\mathrm{H}}^{s_k,p_k}(\mathbb{R}^n)}\\
    &\lesssim_{M,s_k,p_k} \lVert (\mathrm{I}-\dot{S}_{-M})[(\Theta_{R}-1)u_{\tilde{N}}]\rVert_{{\mathrm{H}}^{m_k,p_k}(\mathbb{R}^n)}\\
    &\lesssim_{M,s_k,p_k} \lVert [(\Theta_{R}-1)u_{\tilde{N}}]\rVert_{{\mathrm{H}}^{m_k,p_k}(\mathbb{R}^n)}\text{. }
\end{align*}
Since one may check that $u_{\tilde{N}}\in {\mathrm{H}}^{m_k,p_k}(\mathbb{R}^n)$ for $k\in\{0,1\}$, by dominated convergence theorem it follows that
\begin{align*}
    \lVert u_{\tilde{N},M}^R - u_{\tilde{N}}\rVert_{\dot{\mathrm{H}}^{s_k,p_k}(\mathbb{R}^n)} \xrightarrow[R\rightarrow +\infty]{} 0 \text{. }
\end{align*}
Thus, for $R>0$ big enough, we have for $k\in\{0,1\}$
\begin{align*}
    \lVert u - u_{\tilde{N},M}^R\rVert_{\dot{\mathrm{H}}^{s_k,p_k}(\mathbb{R}^n)} < 2\varepsilon\text{. }
\end{align*}
The proof ends here since $u_{\tilde{N},M}^R \in\eus{S}_0(\mathbb{R}^n)$.
\end{proof}

\subsection{Interpolation, duality and the fundamental Sobolev multiplier result}

We recall the usual interpolation properties for inhomogeneous function spaces,
\begin{align*}
    [\mathrm{H}^{s_0,p_0}(\mathbb{R}^n),\mathrm{H}^{s_1,p_1}(\mathbb{R}^n)]_\theta=\mathrm{H}^{s,p_\theta}(\mathbb{R}^n)\text{, }&\qquad (\mathrm{B}^{s_0}_{p,q_0}(\mathbb{R}^n),\mathrm{B}^{s_1}_{p,q_1}(\mathbb{R}^n))_{\theta,q} = \mathrm{B}^{s}_{p,q}(\mathbb{R}^n)\text{, }\\
    (\mathrm{H}^{s_0,p}(\mathbb{R}^n),\mathrm{H}^{s_1,p}(\mathbb{R}^n))_{\theta,q}= \mathrm{B}^{s}_{p,q}(\mathbb{R}^n)\text{, }&\qquad [\mathrm{B}^{s_0}_{p_0,q_0}(\mathbb{R}^n),\mathrm{B}^{s_1}_{p_1,q_1}(\mathbb{R}^n)]_{\theta} = \mathrm{B}^{s}_{p_\theta,q_\theta}(\mathbb{R}^n)\text{, }
\end{align*}
whenever $(p_0,q_0),(p_1,q_1),(p,q)\in[1,+\infty]^2$($p\neq 1,+\infty$, when dealing with Sobolev (Riesz potential) spaces), $\theta\in(0,1)$, $s_0\neq s_1$ two real numbers, such that
\begin{align*}
    \left(s,\frac{1}{p_\theta},\frac{1}{q_\theta}\right):= (1-\theta)\left(s_0,\frac{1}{p_0},\frac{1}{q_0}\right)+ \theta\left(s_1,\frac{1}{p_1},\frac{1}{q_1}\right)\text{, }
\end{align*}
see \cite[Theorem~6.4.5]{BerghLofstrom1976}. A similar statement is available for our homogeneous function spaces.

\begin{theorem}\label{thm:InterpHomSpacesRn}Let $(p_0,p_1,p,q,q_0,q_1)\in(1,+\infty)^3\times[1,+\infty]^3$, $s_0,s_1\in\mathbb{R}$, such that $s_0\neq s_1$, and for $\theta\in(0,1)$, let
\begin{align*}
    \left(s,\frac{1}{p_\theta},\frac{1}{q_\theta}\right):= (1-\theta)\left(s_0,\frac{1}{p_0},\frac{1}{q_0}\right)+ \theta\left(s_1,\frac{1}{p_1},\frac{1}{q_1}\right)\text{. }
\end{align*}
Assuming $(\mathcal{C}_{s_0,p})$ (resp. $(\mathcal{C}_{s_0,p,q_0})$), we get the following 
\begin{align}
    (\dot{\mathrm{H}}^{s_0,p}(\mathbb{R}^n),\dot{\mathrm{H}}^{s_1,p}(\mathbb{R}^n))_{\theta,q}=(\dot{\mathrm{B}}^{s_0}_{p,q_0}(\mathbb{R}^n),\dot{\mathrm{B}}^{s_1}_{p,q_1}(\mathbb{R}^n))_{\theta,q}=\dot{\mathrm{B}}^{s}_{p,q}(\mathbb{R}^n)\text{.}\label{eq:realInterpHomBspqRn}
\end{align}
If moreover $(\mathcal{C}_{s_0,p_0})$ and $(\mathcal{C}_{s_1,p_1})$ are true then also is $(\mathcal{C}_{s,p_\theta})$ and
\begin{align}
    [\dot{\mathrm{H}}^{s_0,p_0}(\mathbb{R}^n),\dot{\mathrm{H}}^{s_1,p_1}(\mathbb{R}^n)]_{\theta} = \dot{\mathrm{H}}^{s,p_\theta}(\mathbb{R}^n) \text{,}\label{eq:complexInterpHomHspRn}
\end{align}
and similarly if $(\mathcal{C}_{s_0,p_0,q_0})$ and $(\mathcal{C}_{s_1,p_1,q_1})$ are satisfied then $(\mathcal{C}_{s,p_\theta,q_\theta})$ is also satisfied and when $q_\theta < +\infty$,
\begin{align}
    [\dot{\mathrm{B}}^{s_0}_{p_0,q_0}(\mathbb{R}^n),\dot{\mathrm{B}}^{s_1}_{p_1,q_1}(\mathbb{R}^n)]_{\theta} = \dot{\mathrm{B}}^{s}_{p_\theta,q_\theta}(\mathbb{R}^n)\text{.}\label{eq:complexInterpHomBspqRn}
\end{align}
\end{theorem}

\begin{proof} \textbf{Step 1:} Let us deal with the real interpolation identity \eqref{eq:realInterpHomBspqRn}. Let us consider first the case of Sobolev spaces, with $u\in \dot{\mathrm{H}}^{s_0,p}(\mathbb{R}^n)+\dot{\mathrm{H}}^{s_1,p}(\mathbb{R}^n)$. For $(a,b)\in \dot{\mathrm{H}}^{s_0,p}(\mathbb{R}^n)\times\dot{\mathrm{H}}^{s_1,p}(\mathbb{R}^n)$, such that $u=a+b$, by point \textit{(v)} in Proposition \ref{prop:PropertiesHomSobolevSpacesRn} we have
\begin{align*}
    (\dot{\Delta}_j u)_{j\in\mathbb{Z}} = (\dot{\Delta}_j a)_{j\in\mathbb{Z}} + (\dot{\Delta}_j b)_{j\in\mathbb{Z}} \in \mathrm{L}^p(\mathbb{R}^n,\ell^2_{s_0}(\mathbb{Z})) + \mathrm{L}^p(\mathbb{R}^n,\ell^2_{s_1}(\mathbb{Z}))\text{.}
\end{align*}
Therefore, by the definition of the $K$-functional and point \textit{(v)} in Proposition \ref{prop:PropertiesHomSobolevSpacesRn}, for $t>0$,
\begin{align*}
    K(t,(\dot{\Delta}_j u)_{j\in\mathbb{Z}},\mathrm{L}^p(\mathbb{R}^n,\ell^2_{s_0}(\mathbb{Z})),\mathrm{L}^p(\mathbb{R}^n,\ell^2_{s_1}(\mathbb{Z})))&\leqslant \lVert a \lVert_{\dot{\mathrm{F}}^{s_0}_{p,2}(\mathbb{R}^n)} + t \lVert b \lVert_{\dot{\mathrm{F}}^{s_1}_{p,2}(\mathbb{R}^n)}\\
    &\lesssim_{p,s_0,s_1,n} \lVert a \lVert_{\dot{\mathrm{H}}^{s_0,p}(\mathbb{R}^n)} + t \lVert b \lVert_{\dot{\mathrm{H}}^{s_1,p}(\mathbb{R}^n)}\text{.}
\end{align*}
We then take the infimum on an all such pairs $(a,b)$,
\begin{align}\label{eq:KfuncLpl2sControledByHsp}
    K(t,(\dot{\Delta}_j u)_{j\in\mathbb{Z}},\mathrm{L}^p(\mathbb{R}^n,\ell^2_{s_0}(\mathbb{Z})),\mathrm{L}^p(\mathbb{R}^n,\ell^2_{s_1}(\mathbb{Z})))\lesssim_{p,s_0,s_1,n} K(t,u,\dot{\mathrm{H}}^{s_0,p}(\mathbb{R}^n),\dot{\mathrm{H}}^{s_1,p}(\mathbb{R}^n))\text{.}
\end{align}
Now, we want to prove the reverse estimate. Since $(\dot{\Delta}_j u)_{j\in\mathbb{Z}} \in \mathrm{L}^p(\mathbb{R}^n,\ell^2_{s_0}(\mathbb{Z})) + \mathrm{L}^p(\mathbb{R}^n,\ell^2_{s_1}(\mathbb{Z}))$, let $(A,B)\in \mathrm{L}^p(\mathbb{R}^n,\ell^2_{s_0}(\mathbb{Z}))\times \mathrm{L}^p(\mathbb{R}^n,\ell^2_{s_1}(\mathbb{Z}))$ such that
\begin{align}\label{eq:Lpl2decompofDeltaju}
    (\dot{\Delta}_j u)_{j\in\mathbb{Z}} = A+B \text{.}
\end{align}
For $(w_j)_{j\in\mathbb{Z}}\subset{}\eus{S}'(\mathbb{R}^n)$, say, for simplicity, with finite support in the discrete variable, we define the map
\begin{align}\label{eq:retractionMapsLp(l2s)intoHsp}
    \tilde{\Sigma}((w_j)_{j\in\mathbb{Z}}) := \sum_{j=-\infty}^{+\infty} \dot{\Delta}_j [w_{j-1}+w_{j}+w_{j+1}] \text{,}
\end{align}
and it satisfies for $v\in \eus{S}'_h(\mathbb{R}^n)$
\begin{align*}
    \tilde{\Sigma}((\dot{\Delta}_j v)_{j\in\mathbb{Z}})=v\text{.}
\end{align*}
By point \textit{(v)} in Proposition \ref{prop:PropertiesHomSobolevSpacesRn} and \cite[Proposition~6.1.4]{bookGrafakos2014Classical}, one can check that
\begin{align}\label{eq:ReconstrucOpBnddssProp1}
    \tilde{\Sigma}\,:\,\mathrm{L}^p(\mathbb{R}^n,\ell^2_{s_0}(\mathbb{Z}))\longrightarrow \dot{\mathrm{H}}^{s_0,p}(\mathbb{R}^n)
\end{align}
is well-defined and bounded since $(\mathcal{C}_{s_0,p})$ is satisfied. Now, we apply $\tilde{\Sigma}$ to \eqref{eq:Lpl2decompofDeltaju} to deduce from $\tilde{\Sigma} (\dot{\Delta}_j u)_{j\in\mathbb{Z}} = u \in\eus{S}'_h(\mathbb{R}^n)$, and $\tilde{\Sigma} A \in \dot{\mathrm{H}}^{s_0,p}(\mathbb{R}^n)\subset \eus{S}'_h(\mathbb{R}^n)$, that
\begin{align*}
    \tilde{\Sigma} B = u - \tilde{\Sigma} A \in  \eus{S}'_h(\mathbb{R}^n)\text{.}
\end{align*}
By mean of \cite[Proposition~6.1.4]{bookGrafakos2014Classical}, we obtain
\begin{align}\label{eq:ReconstrucOpBnddssProp2}
    \lVert \tilde{\Sigma}  B \lVert_{\dot{\mathrm{F}}^{s_1}_{p,2}(\mathbb{R}^n)}=\lVert (\dot{\Delta}_j \tilde{\Sigma} B )_{j\in\mathbb{Z}}\rVert_{\mathrm{L}^p(\mathbb{R}^n,\ell^2_{s_1}(\mathbb{Z}))} \lesssim_{p,s_1,n}  \lVert B \rVert_{\mathrm{L}^p(\mathbb{R}^n,\ell^2_{s_1}(\mathbb{Z}))}\text{.}
\end{align}
Hence, by point \textit{(v)} in Proposition \ref{prop:PropertiesHomSobolevSpacesRn}, $\tilde{\Sigma} B$ is an element of $\dot{\mathrm{H}}^{s_1,p}(\mathbb{R}^n)$. Therefore, by the definition of the $K$-functional, the boundedness properties of $\tilde{\Sigma}$, point \textit{(v)} in Proposition \ref{prop:PropertiesHomSobolevSpacesRn}, for $t>0$,
\begin{align*}
    K(t,u,\dot{\mathrm{H}}^{s_0,p}(\mathbb{R}^n),\dot{\mathrm{H}}^{s_1,p}(\mathbb{R}^n)) &\leqslant \lVert \tilde{\Sigma} A \lVert_{\dot{\mathrm{H}}^{s_0,p}(\mathbb{R}^n)} + t \lVert \tilde{\Sigma}  B \lVert_{\dot{\mathrm{H}}^{s_1,p}(\mathbb{R}^n)}\\
    &\lesssim_{p,s_0,s_1,n} \lVert A \rVert_{\mathrm{L}^p(\mathbb{R}^n,\ell^2_{s_0}(\mathbb{Z}))} + t \lVert B \rVert_{\mathrm{L}^p(\mathbb{R}^n,\ell^2_{s_1}(\mathbb{Z}))}\text{.}
\end{align*}
Thus, let us take the infimum on all such pairs $(A,B)$, and invoke \eqref{eq:KfuncLpl2sControledByHsp} to obtain for all $t>0$, and all $u\in \dot{\mathrm{H}}^{s_0,p}(\mathbb{R}^n)+\dot{\mathrm{H}}^{s_1,p}(\mathbb{R}^n)$,
\begin{align}\label{eq:KfuncLpl2sEquivHsp}
    K(t,(\dot{\Delta}_j u)_{j\in\mathbb{Z}},\mathrm{L}^p(\mathbb{R}^n,\ell^2_{s_0}(\mathbb{Z})),\mathrm{L}^p(\mathbb{R}^n,\ell^2_{s_1}(\mathbb{Z})))\sim_{p,s_0,s_1,n} K(t,u,\dot{\mathrm{H}}^{s_0,p}(\mathbb{R}^n),\dot{\mathrm{H}}^{s_1,p}(\mathbb{R}^n))\text{.}
\end{align}
We recall that \cite[Theorems~5.6.1~\&~3.5.3]{BerghLofstrom1976} and \cite[Theorem,~Section~1.18.4]{bookTriebel1978} give, all together, the well known real interpolation identity
\begin{align}\label{eq:realInterpLp(l2s)intolqs(Lp)}
    (\mathrm{L}^p(\mathbb{R}^n,\ell^2_{s_0}(\mathbb{Z})),\mathrm{L}^p(\mathbb{R}^n,\ell^2_{s_1}(\mathbb{Z})))_{\theta,q} = \ell^q_{s}(\mathbb{Z},\mathrm{L}^p(\mathbb{R}^n))\text{.}
\end{align}
Thus, up to multiply the estimate \eqref{eq:KfuncLpl2sEquivHsp} by $t^{-\theta}$ and taking its $\mathrm{L}^q_{\ast}$-norm, it can be turned into
\begin{align*}
    \lVert u\rVert_{\dot{\mathrm{B}}^{s}_{p,q}(\mathbb{R}^n)} = \lVert (\dot{\Delta}_j u)_{j\in\mathbb{Z}}\rVert_{\ell^q_{s}(\mathbb{Z},\mathrm{L}^p(\mathbb{R}^n))} &\sim_{p,s_0,s_1,\theta,n} \lVert (\dot{\Delta}_j u)_{j\in\mathbb{Z}}\rVert_{(\mathrm{L}^p(\mathbb{R}^n,\ell^2_{s_0}(\mathbb{Z})),\mathrm{L}^p(\mathbb{R}^n,\ell^2_{s_1}(\mathbb{Z})))_{\theta,q}}\\
    &\sim_{p,s_0,s_1,\theta,n} \lVert u\rVert_{(\dot{\mathrm{H}}^{s_0,p}(\mathbb{R}^n),\dot{\mathrm{H}}^{s_1,p}(\mathbb{R}^n))_{\theta,q}}\text{.}
\end{align*}
Therefore \eqref{eq:realInterpHomBspqRn} is proved.

\textbf{Step 2:} For $p\in(1,+\infty)$, $q\in[1,+\infty]$ and $s\in\mathbb{R}$ such that \eqref{AssumptionCompletenessExponents} is satisfied, for $\tilde{\Sigma}$ introduced in \eqref{eq:retractionMapsLp(l2s)intoHsp}, we want to show the boundedness of
\begin{align}\label{eq:VectorVallqs(Lp)LittlewoodPaleyIntoBesov}
    \tilde{\Sigma}\,:\,\ell^q_{s}(\mathbb{Z},\mathrm{L}^p(\mathbb{R}^n))\longrightarrow \dot{\mathrm{B}}^{s}_{p,q}(\mathbb{R}^n)\text{.}
\end{align}

The idea is to consider instead the operator $(\dot{\Delta}_j\Tilde{\Sigma})_{j\in\mathbb{Z}}$ and to show that it is a bounded operator, seen as
\begin{align}\label{eq:VectorVallqs(Lp)LittlewoodPaley}
    (\dot{\Delta}_j\Tilde{\Sigma})_{j\in\mathbb{Z}}\,:\, \ell^q_{s}(\mathbb{Z},\mathrm{L}^p(\mathbb{R}^n)) \longrightarrow \ell^q_{s}(\mathbb{Z},\mathrm{L}^p(\mathbb{R}^n))\text{, } p\in(1,+\infty)\text{, } q\in[1,+\infty]\text{, } s\in\mathbb{R}.
\end{align}

In fact, as for \eqref{eq:ReconstrucOpBnddssProp1} and \eqref{eq:ReconstrucOpBnddssProp2}, by \cite[Proposition~6.1.4]{bookGrafakos2014Classical}, for given $p\in(1,+\infty)$ and $s,s_0,s_1\in\mathbb{R}$ with $s_0<s<s_1$, one has the following boundedness property:
\begin{align*}
    (\dot{\Delta}_j\Tilde{\Sigma})_{j\in\mathbb{Z}}\,:\, \mathrm{L}^p(\mathbb{R}^n,\ell^2_{s_k}(\mathbb{Z})) \longrightarrow \mathrm{L}^p(\mathbb{R}^n,\ell^2_{s_k}(\mathbb{Z}))\text{, } \quad k\in\{0,1\}.
\end{align*}
Therefore, \eqref{eq:VectorVallqs(Lp)LittlewoodPaley} follows from real interpolation, thanks to the previous identity \eqref{eq:realInterpLp(l2s)intolqs(Lp)}.

Now, in order to obtain \eqref{eq:VectorVallqs(Lp)LittlewoodPaleyIntoBesov} when \eqref{AssumptionCompletenessExponents} is satisfied, it suffices to consider an element $(u_j)_{j\in\mathbb{Z}}$ in the space $\ell^q_{s}(\mathbb{Z},\mathrm{L}^p(\mathbb{R}^n))$, then to approximate it by truncation with respect to the discrete variable, so that
\begin{align*}
    \left(\Tilde{\Sigma}(u_j)_{j\in\llb -N,N\rrb}\right)_{N\in\mathbb{N}}\subset \mathrm{L}^p(\mathbb{R}^n)\cap \dot{\mathrm{B}}^{s}_{p,q}(\mathbb{R}^n).
\end{align*}

Then, the map $\tilde{\Sigma}$ extends uniquely to a bounded map with values in $\dot{\mathrm{B}}^{s}_{p,q}(\mathbb{R}^n)$ whenever \eqref{AssumptionCompletenessExponents} is satisfied and $q<+\infty$.

For the case $q=+\infty$ when \eqref{AssumptionCompletenessExponents} is satisfied, \textit{i.e.} when $s<n/p$ is satisfied, the result follows in fact directly from \textbf{Step 1}.

In fact, the above manual real interpolation procedure was mainly performed to reach the endpoint couple $(\dot{\mathrm{B}}^{n/p}_{p,1}(\mathbb{R}^n),\ell^1_{n/p}(\mathbb{Z},\mathrm{L}^p(\mathbb{R}^n)))$.

\textbf{Step 3:} For the real interpolation identity \eqref{eq:realInterpHomBspqRn} in the case of Besov spaces, by the previous \textbf{Step 2}, the proof presented in \textbf{Step 1} is still valid if we replace $(\dot{\mathrm{H}}^{s_0,p},\dot{\mathrm{H}}^{s_1,p})$ and the condition $(\mathcal{C}_{s_0,p})$ by $(\dot{\mathrm{B}}^{s_0}_{p,q_0},\dot{\mathrm{B}}^{s_1}_{p,q_1})$ with the condition $(\mathcal{C}_{s_0,p,q_0})$.

\textbf{Step 4:} As in the proof of \cite[Theorem~6.4.5]{BerghLofstrom1976}, being aware of \cite[Definition~6.4.1]{BerghLofstrom1976}, we can claim, thanks to previous steps, that
\begin{itemize}
    \item thanks to its definition, for all $s\in\mathbb{R}$, $p\in(1,+\infty)$, $q\in[1,+\infty]$, when \eqref{AssumptionCompletenessExponents} is satisfied, $\dot{\mathrm{B}}^{s}_{p,q}(\mathbb{R}^n)$ is a retraction of $\ell^q_s(\mathbb{Z},\mathrm{L}^p(\mathbb{R}^n))$ on $\eus{S}'_h(\mathbb{R}^n)$ through the homogeneous Littlewood-Paley decomposition $(\dot{\Delta}_{j})_{j\in\mathbb{Z}}$, and projection map $\tilde{\Sigma}$;
    \item similarly, due to point \textit{(v)} in Proposition \ref{prop:PropertiesHomSobolevSpacesRn}, for all $s\in\mathbb{R}$, $p\in(1,+\infty)$, when $(\mathcal{C}_{s,p})$ is satisfied $\dot{\mathrm{H}}^{s,p}(\mathbb{R}^n)$ is a retraction of $\mathrm{L}^p(\mathbb{R}^n,\ell^2_s(\mathbb{Z}))$ on $\eus{S}'_h(\mathbb{R}^n)$ through the homogeneous Littlewood-Paley decomposition $(\dot{\Delta}_{j})_{j\in\mathbb{Z}}$, and projection map $\tilde{\Sigma}$. 
\end{itemize}
Thus, one may apply \cite[Theorem~6.4.2]{BerghLofstrom1976}, with \cite[Theorem~5.6.3]{BerghLofstrom1976} for complex interpolation of Besov spaces and \cite[Theorem,~Section~1.18.4]{bookTriebel1978} for complex interpolation of Sobolev spaces, to obtain respectively \eqref{eq:complexInterpHomBspqRn} and \eqref{eq:complexInterpHomHspRn}.

The completeness assumption is necessary in the case of complex interpolation, since one can not provide in general an appropriate sense of holomorphic functions (then of the definition of complex interpolation spaces) in non-complete normed vector spaces.
\end{proof}

\begin{proposition}\label{prop:dualityRieszpotential} For any $s\in\mathbb{R}$, $p\in(1,+\infty)$,
\begin{align*}
\left\{\begin{array}{cl}
 \dot{\mathrm{H}}^{s,p}\times \dot{\mathrm{H}}^{-s,p'} &\longrightarrow \mathbb{C}\\
(u, v) &\longmapsto \sum\limits_{\left|j-j'\right| \leq 1} \left\langle\dot{\Delta}_{j} u, \dot{\Delta}_{j'} v\right\rangle_{\mathbb{R}^n}
\end{array}\right.
\end{align*}
defines a continuous bilinear functional on $\dot{\mathrm{H}}^{s,p}(\mathbb{R}^n)\times \dot{\mathrm{H}}^{-s,p'}(\mathbb{R}^n)$. Denote by $\mathcal{V}^{-s,p'}$ the set of functions $v\in\eus{S}(\mathbb{R}^n)\cap\dot{\mathrm{H}}^{-s,p'}(\mathbb{R}^n)$ such that $\left\lVert{v}\right\rVert_{\dot{\mathrm{H}}^{-s,p'}(\mathbb{R}^n)}\leqslant 1$. If $u\in\eus{S}'_h(\mathbb{R}^n)$, then we have
\begin{align*}
    \left\lVert{u}\right\rVert_{\dot{\mathrm{H}}^{s,p}(\mathbb{R}^n)} = \sup\limits_{\substack{v\in \mathcal{V}^{-s,p'}}} \big\lvert\big\langle u,  v\big\rangle_{\mathbb{R}^n}\big\rvert\text{. }
\end{align*}
Moreover, if $(\mathcal{C}_{s,p})$ is satisfied, $\dot{\mathrm{H}}^{s,p}(\mathbb{R}^n)$ is reflexive and we have
\begin{align}\label{eq:dualityRieszPotential}
    (\dot{\mathrm{H}}^{-s,p'}(\mathbb{R}^n))' = \dot{\mathrm{H}}^{s,p}(\mathbb{R}^n)\text{. }
\end{align}
\end{proposition}

\begin{proof} For simplicity, we will first work with the norm provided point \textit{(v)} in Proposition \ref{prop:PropertiesHomSobolevSpacesRn}, by equivalence of norms, the result will remain true. Let $(u, v)\in\dot{\mathrm{H}}^{s,p}(\mathbb{R}^n)\times \dot{\mathrm{H}}^{-s,p'}(\mathbb{R}^n)$, the $\mathrm{L}^{p}(\ell^2)$-$\mathrm{L}^{p'}(\ell^2)$ H\"{o}lder's inequality gives,
\begin{align*}
    \big\lvert\big\langle u,  v\big\rangle_{\mathbb{R}^n}\big\rvert &\leqslant \left\lVert{u}\right\rVert_{\dot{\mathrm{F}}^{s}_{p,2}(\mathbb{R}^n)}\left\lVert \left\lVert(2^{-js} [\dot{\Delta}_{j-1}+\dot{\Delta}_{j}+\dot{\Delta}_{j+1}] v )_{j\in\mathbb{Z}}\right\rVert_{\ell^{2}(\mathbb{Z})}\right\rVert_{\mathrm{L}^{p'}(\mathbb{R}^n)}\\
    &\leqslant (2^{|s|+1}+1) \left\lVert{u}\right\rVert_{\dot{\mathrm{F}}^{s}_{p,2}(\mathbb{R}^n)}\left\lVert{v}\right\rVert_{\dot{\mathrm{F}}^{-s}_{p',2}(\mathbb{R}^n)}\text{.}
\end{align*}
Now, we know that it is a well-defined quantity, we can compute
\begin{align*}
    \big\langle u,  v\big\rangle_{\mathbb{R}^n} &= \sum_{|j-j'| \leq 1} \left\langle\dot{\Delta}_{j} u, \dot{\Delta}_{j'} v\right\rangle_{\mathbb{R}^n}\\
    &= \sum_{|j-j'| \leqslant 1} \left\langle (-\Delta)^\frac{s}{2}\dot{\Delta}_{j} u, (-\Delta)^{-\frac{s}{2}}\dot{\Delta}_{j'} v\right\rangle_{\mathbb{R}^n}\\
    &= \left\langle (-\Delta)^\frac{s}{2} u, (-\Delta)^{-\frac{s}{2}} v\right\rangle_{\mathbb{R}^n}\text{.}
\end{align*}
Hence, H\"{o}lder's inequality gives
\begin{align*}
    \big\lvert\big\langle u,  v\big\rangle_{\mathbb{R}^n}\big\rvert \leqslant \lVert{u}\rVert_{\dot{\mathrm{H}}^{s,p}(\mathbb{R}^n)} \lVert{v}\rVert_{\dot{\mathrm{H}}^{-s,p'}(\mathbb{R}^n)}\text{, }
\end{align*}
which can be turned effortless into
\begin{align*}
     \sup\limits_{\substack{v\in \mathcal{V}^{-s,p'}}} \big\lvert\big\langle u,  v\big\rangle_{\mathbb{R}^n}\big\rvert \leqslant \left\lVert{u}\right\rVert_{\dot{\mathrm{H}}^{s,p}(\mathbb{R}^n)}\text{. }
\end{align*}
This also proves the continuous embedding $\dot{\mathrm{H}}^{s,p}(\mathbb{R}^n)\hookrightarrow (\dot{\mathrm{H}}^{-s,p'}(\mathbb{R}^n))'$. For the reverse inequality, but not the reverse embedding, from $\mathrm{L}^p-\mathrm{L}^{p'}$ duality, by density of $\eus{S}_0(\mathbb{R}^n)$, we have
\begin{align*}
    \left\lVert{u}\right\rVert_{\dot{\mathrm{H}}^{s,p}(\mathbb{R}^n)} &= \sup\limits_{\substack{v\in \eus{S}_0(\mathbb{R}^n)\text{,}\\ \lVert v\rVert_{\mathrm{L}^{p'}}\leqslant 1}} \big\lvert\big\langle (-\Delta)^\frac{s}{2} u,  v\big\rangle_{\mathbb{R}^n}\big\rvert = \sup\limits_{\substack{w\in \eus{S}_0(\mathbb{R}^n)\text{,}\\ \lVert w\rVert_{\dot{\mathrm{H}}^{-s,p'}}\leqslant 1}} \big\lvert\big\langle u,  w\big\rangle_{\mathbb{R}^n}\big\rvert \leqslant \sup\limits_{\substack{v\in \mathcal{V}^{-s,p'}}} \big\lvert\big\langle u,  v\big\rangle_{\mathbb{R}^n}\big\rvert \text{. }
\end{align*}
In particular, the embedding $\dot{\mathrm{H}}^{s,p}(\mathbb{R}^n)\hookrightarrow (\dot{\mathrm{H}}^{-s,p'}(\mathbb{R}^n))'$ always holds and is isometric.

Now, assume that $(\mathcal{C}_{s,p})$ holds. We recall that point \textit{(ii)} of Proposition \ref{prop:PropertiesHomSobolevSpacesRn} yields the reflexivity of $\dot{\mathrm{H}}^{s,p}(\mathbb{R}^n)$. Let $\tilde{U}\in (\dot{\mathrm{H}}^{-s,p'}(\mathbb{R}^n))'$, we have
\begin{align*}
    \big\lvert\big\langle \tilde{U}, (-\Delta)^{\frac{s}{2}}  v\big\rangle\big\rvert \leqslant \lVert{\tilde{U}}\rVert_{(\dot{\mathrm{H}}^{-s,p'}(\mathbb{R}^n))'} \lVert{v}\rVert_{{\mathrm{L}}^{p'}(\mathbb{R}^n)}, \quad v\in\eus{S}_0(\mathbb{R}^n)\text{. }
\end{align*}
Since the space $\eus{S}_0(\mathbb{R}^n)$ is dense in ${\mathrm{L}}^{p'}(\mathbb{R}^n)$, we deduce there exists a unique function $w\in\mathrm{L}^p(\mathbb{R}^n)$ such that,
\begin{align*}
    \big\langle \tilde{U},  v\big\rangle = \big\langle w, (-\Delta)^{-\frac{s}{2}}  v\big\rangle_{\mathbb{R}^n}, \quad v\in\eus{S}(\mathbb{R}^n)\text{. }
\end{align*}
Thus $u:=(-\Delta)^{-\frac{s}{2}}w\in\dot{\mathrm{H}}^{s,p}(\mathbb{R}^n)$ by point \textit{(ii)} in Proposition \ref{prop:PropertiesHomSobolevSpacesRn}, and yields that the canonical embedding $\dot{\mathrm{H}}^{s,p}(\mathbb{R}^n)\hookrightarrow (\dot{\mathrm{H}}^{-s,p'}(\mathbb{R}^n))'$ is surjective.
\end{proof}

\begin{proposition}\label{prop:dualityHomBesov} For any $s\in\mathbb{R}$, $p\in(1,+\infty)$, $q\in[1,+\infty]$,
\begin{align*}
\left\{\begin{array}{cl}
 \dot{\mathrm{B}}^{s}_{p,q}\times \dot{\mathrm{B}}^{-s}_{p',q'} &\longrightarrow \mathbb{C}\\
(u, v) &\longmapsto \sum\limits_{\left|j-j'\right| \leq 1} \left\langle\dot{\Delta}_{j} u, \dot{\Delta}_{j'} v\right\rangle_{\mathbb{R}^n}
\end{array}\right.
\end{align*}
defines a continuous bilinear functional on $\dot{\mathrm{B}}^{s}_{p,q}(\mathbb{R}^n)\times \dot{\mathrm{B}}^{-s}_{p',q'}(\mathbb{R}^n)$. Denote by $\mathcal{Q}^{-s}_{p',q'}$ the set of functions $v\in\eus{S}(\mathbb{R}^n)\cap\dot{\mathrm{B}}^{-s}_{p',q'}(\mathbb{R}^n)$ such that $\left\lVert{v}\right\rVert_{\dot{\mathrm{B}}^{-s}_{p',q'}(\mathbb{R}^n)}\leqslant 1$. If $u\in\eus{S}'_h(\mathbb{R}^n)$, then we have
\begin{align*}
    \left\lVert{u}\right\rVert_{\dot{\mathrm{B}}^{s}_{p,q}(\mathbb{R}^n)} \lesssim_{p,s,n} \sup\limits_{\substack{v\in \mathcal{Q}^{-s}_{p',q'}}} \big\lvert\big\langle u,  v\big\rangle_{\mathbb{R}^n}\big\rvert\text{. }
\end{align*}
Moreover, if $-n/p'<s<n/p$ is satisfied and  $q\in(1,+\infty]$ then
\begin{align}
    (\dot{\mathrm{B}}^{-s}_{p',q'}(\mathbb{R}^n))' = \dot{\mathrm{B}}^{s}_{p,q}(\mathbb{R}^n)\text{ and } (\dot{\mathcal{B}}^{-s}_{p',\infty}(\mathbb{R}^n))' =  \dot{\mathrm{B}}^{s}_{p,1}(\mathbb{R}^n)\text{. }
\end{align}
The space $\dot{\mathrm{B}}^{s}_{p,q}(\mathbb{R}^n)$ is reflexive whenever both \eqref{AssumptionCompletenessExponents} and  $q\neq 1,+\infty$ are satisfied.
\end{proposition}

\begin{proof}The first part of the claim is just \cite[Proposition~2.29]{bookBahouriCheminDanchin}. The claimed part about reflexivity and duality follows directly from the application of \cite[Theorem~3.7.1]{BerghLofstrom1976} and of Theorem \ref{thm:InterpHomSpacesRn} and Proposition \ref{prop:dualityRieszpotential}.
\end{proof}

We also have a Sobolev-Besov multiplier result, which is useful for the construction of homogeneous Sobolev and Besov space on domains.
The first presentation of the next result in the setting of inhomogeneous function spaces is due to Strichartz \cite[Chapter~II,~Corollary~3.7]{Strichartz1967}, one may also check \cite[Proposition~3.5]{JerisonKenig1995}. We are going to use it to prove a straightforward generalization. The next result was also known but only stated for homogeneous Besov spaces up to now, see e.g. \cite[Appendix]{DanchinMucha2009}.

\begin{proposition}\label{prop:SobolevMultiplier} For all $p\in(1,+\infty)$, $q\in[1,+\infty]$, for all $s\in (-1+\frac{1}{p},\frac{1}{p})$, for all $u\in\dot{\mathrm{H}}^{s,p}(\mathbb{R}^n)$ (resp. $\dot{\mathrm{B}}^{s}_{p,q}(\mathbb{R}^{n})$),
\begin{align*}
    \lVert \mathbbm{1}_{\mathbb{R}^n_+} u \rVert_{\dot{\mathrm{H}}^{s,p}(\mathbb{R}^{n})} \lesssim_{s,p,n} \lVert u \rVert_{\dot{\mathrm{H}}^{s,p}(\mathbb{R}^{n})} \text{ }\text{ (resp. }  \lVert \mathbbm{1}_{\mathbb{R}^n_+} u \rVert_{\dot{\mathrm{B}}^{s}_{p,q}(\mathbb{R}^{n})} \lesssim_{s,p,n} \lVert u \rVert_{\dot{\mathrm{B}}^{s}_{p,q}(\mathbb{R}^{n})} \text{ ). }
\end{align*}
The same results still hold with $({\mathrm{H}},{\mathrm{B}})$ instead of $(\dot{\mathrm{H}},\dot{\mathrm{B}})$.
\end{proposition}

\begin{proof}We start from the result stated in the inhomogeneous case \cite[Chapter~II,~Corollary~3.7]{Strichartz1967}, which states the following: for all $p\in(1,+\infty)$, for all $s\in [0,\frac{1}{p})$, for all $u\in{\mathrm{H}}^{s,p}(\mathbb{R}^n)$
\begin{align*}
    \lVert \mathbbm{1}_{\mathbb{R}^n_+} u \rVert_{{\mathrm{H}}^{s,p}(\mathbb{R}^n)} \lesssim_{s,p,n} \lVert u \rVert_{{\mathrm{H}}^{s,p}(\mathbb{R}^n)} \text{. }
\end{align*}
If $s=0$, there is nothing to achieve since ${\mathrm{H}}^{0,p}(\mathbb{R}^n)=\dot{\mathrm{H}}^{0,p}(\mathbb{R}^n)=\mathrm{L}^p(\mathbb{R}^n)$ with equality of norms. Now for $s>0$, by the equivalence of norms, we obtain
\begin{align*}
    \lVert \mathbbm{1}_{\mathbb{R}^n_+} u \rVert_{\mathrm{L}^p(\mathbb{R}^n)} + \lVert \mathbbm{1}_{\mathbb{R}^n_+} u \rVert_{\dot{\mathrm{H}}^{s,p}(\mathbb{R}^n)} \lesssim_{s,p,n} \lVert u \rVert_{\mathrm{L}^p(\mathbb{R}^n)}  +  \lVert u \rVert_{\dot{\mathrm{H}}^{s,p}(\mathbb{R}^n)} \text{. }
\end{align*}
Plugging $u_\lambda:=u(\lambda\cdot)$ in the above inequality, provided $\lambda$ is a positive real number, since one has $\mathbbm{1}_{\mathbb{R}^n_+}(\lambda\cdot ) u_\lambda = \mathbbm{1}_{\mathbb{R}^n_+} u_\lambda$, we obtain that
\begin{align*}
    \lambda^{-\frac{n}{p}}\lVert \mathbbm{1}_{\mathbb{R}^n_+} u \rVert_{\mathrm{L}^p(\mathbb{R}^n)} + \lambda^{s-\frac{n}{p}}\lVert \mathbbm{1}_{\mathbb{R}^n_+} u \rVert_{\dot{\mathrm{H}}^{s,p}(\mathbb{R}^n)} \lesssim_{s,p,n} \lambda^{-\frac{n}{p}}\lVert u \rVert_{\mathrm{L}^p(\mathbb{R}^n)}  +  \lambda^{s-\frac{n}{p}}\lVert u \rVert_{\dot{\mathrm{H}}^{s,p}(\mathbb{R}^n)} \text{. }
\end{align*}
Thus one may divide by $\lambda^{s-\frac{n}{p}}$, and then as $\lambda$ tends to infinity, we deduce
\begin{align*}
     \lVert \mathbbm{1}_{\mathbb{R}^n_+} u \rVert_{\dot{\mathrm{H}}^{s,p}(\mathbb{R}^n)} \lesssim_{s,p,n}   \lVert u \rVert_{\dot{\mathrm{H}}^{s,p}(\mathbb{R}^n)} \text{. }
\end{align*}
Therefore, the result follows by density argument.

The result for $s\in(-1+\frac{1}{p},0)$ is a consequence of duality and density using the duality bracket defined on $\eus{S}_0(\mathbb{R}^n)\times \eus{S}_0(\mathbb{R}^n)$.

The Besov space case follows by real interpolation.
\end{proof}


\section{Function spaces on the upper half-space}\label{sec:FunctionSpacesRn+}

From now on, until the end of the paper, we assume that the dimension $n$ satisfies $n \geqslant2$.

Let $s\in\mathbb{R}$, $p\in(1,+\infty)$, $q\in[1,+\infty]$. Then for any $\mathrm{X}\in\{ \mathrm{B}^{s}_{p,q}, \dot{\mathrm{B}}^{s}_{p,q}, \mathrm{\mathrm{H}}^{s,p}, \dot{\mathrm{H}}^{s,p}\}$, we define
\begin{align*}
    \mathrm{X}(\mathbb{R}^n_+):= \mathrm{X}(\mathbb{R}^n)_{|_{\mathbb{R}^n_+}}\text{, }
\end{align*}
with the quotient norm $\lVert u \rVert_{\mathrm{X}(\mathbb{R}^n_+)}:= \inf\limits_{\substack{\Tilde{u}\in \mathrm{X}(\mathbb{R}^n),\\ \tilde{u}_{|_{\mathbb{R}^n_+}}=u\, .}} \lVert \Tilde{u} \rVert_{\mathrm{X}(\mathbb{R}^n)}$. A direct consequence of the definition of those spaces is the density of $\eus{S}_0(\overline{\mathbb{R}^n_+})\subset\eus{S}(\overline{\mathbb{R}^n_+})$ in each of them. The completeness and reflexivity is also carried over when their counterpart on $\mathbb{R}^n$ are respectively complete and reflexive. We also define
\begin{align*}
    \mathrm{X}_0(\mathbb{R}^n_+):= \left\{\,u\in \mathrm{X}(\mathbb{R}^n) \,\Big{|}\, \supp u \subset \overline{\mathbb{R}^n_+} \right\}\text{, }
\end{align*}
with natural induced norm $\lVert  u \rVert_{\mathrm{X}_0(\mathbb{R}^n_+)}:= \lVert  u \rVert_{\mathrm{X}(\mathbb{R}^n)}$. We always have the canonical continuous injection,
\begin{align*}
    \mathrm{X}_0(\mathbb{R}^n_+)\hookrightarrow \mathrm{X}(\mathbb{R}^n_+) \text{. }
\end{align*}
Since there is a natural embedding $\eus{S}'(\mathbb{R}^n)\hookrightarrow \eus{D}'(\mathbb{R}^n_+)$, we also have the inclusion
\begin{align*}
    \mathrm{X}(\mathbb{R}^n_+) \subset \eus{D}'(\mathbb{R}^n_+)\text{.}
\end{align*}

If $\mathrm{X}$ and $\mathrm{Y}$ are different function spaces
\begin{itemize}
    \item  if one has continuous embedding
\begin{align*}
    \mathrm{Y}(\mathbb{R}^n)\hookrightarrow \mathrm{X}(\mathbb{R}^n) \text{. }
\end{align*}
a direct consequence of the definition is 
\begin{align*}
    \mathrm{Y}(\mathbb{R}^n_+)\hookrightarrow \mathrm{X}(\mathbb{R}^n_+) \text{, }
\end{align*}
and similarly with $\mathrm{X}_0$ and $\mathrm{Y}_0$.
    
    \item We write $[\mathrm{X}\cap \mathrm{Y}](\mathbb{R}^n_+)$ the restriction of $\mathrm{X}(\mathbb{R}^n)\cap \mathrm{Y}(\mathbb{R}^n)$ to $\mathbb{R}^n_+$, in general there is nothing to ensure more than 
    \begin{align*}
        [\mathrm{X}\cap \mathrm{Y}](\mathbb{R}^n_+)\hookrightarrow \mathrm{X}(\mathbb{R}^n_+)\cap \mathrm{Y}(\mathbb{R}^n_+) \text{. }
    \end{align*}
\end{itemize}

Results corresponding to those obtained for the whole space $\mathbb{R}^n$ in the previous section are usually carried over by the existence of an appropriate
extension operator
\begin{align*}
    \mathcal{E}\,:\, \eus{S}'(\mathbb{R}^n_+)\longrightarrow \eus{S}'(\mathbb{R}^n)\text{, }
\end{align*}
bounded from $\mathrm{X}(\mathbb{R}^n_+)$ to $\mathrm{X}(\mathbb{R}^n)$.

\subsection{Quick overview of inhomogeneous function spaces on \texorpdfstring{$\mathbb{R}^n_+$}{Rn+}}\label{subsec:InhomSpacesRn+} For inhomogeneous spaces on special Lipschitz domains (in particular on $\mathbb{R}^n_+$), an approach was done by Stein in \cite[Chapter~VI]{Stein1970}, for Sobolev spaces with non-negative index, and Besov spaces of positive index of regularity (this follows by real interpolation). A full and definitive result for the inhomogeneous case on Lipschitz domains, and even in a more general case (allowing $p,q$ to be less than $1$ considering the whole Besov and Triebel-Lizorkin scales), was given by Rychkov in \cite{Rychkov1999} where the extension operator is known to be universal and to cover even negative regularity index. 

The extension operator provided by Rychkov can be used to prove, thanks to \cite[Theorem~6.4.2]{BerghLofstrom1976}, if $(\mathfrak{h},\mathfrak{b})\in\{(\mathrm{H},\mathrm{B}), (\mathrm{H}_0,\mathrm{B}_{\cdot,\cdot,0})\}$,
\begin{align}
    [\mathfrak{h}^{s_0,p_0}(\mathbb{R}^n_+),\mathfrak{h}^{s_1,p_1}(\mathbb{R}^n_+)]_\theta=\mathfrak{h}^{s,p_\theta}(\mathbb{R}^n_+)\text{, }&\qquad (\mathfrak{b}^{s_0}_{p,q_0}(\mathbb{R}^n_+),\mathfrak{b}^{s_1}_{p,q_1}(\mathbb{R}^n_+))_{\theta,q} = \mathfrak{b}^{s}_{p,q}(\mathbb{R}^n_+)\text{, }\\
    (\mathfrak{h}^{s_0,p}(\mathbb{R}^n_+),\mathfrak{h}^{s_1,p}(\mathbb{R}^n_+))_{\theta,q}= \mathfrak{b}^{s}_{p,q}(\mathbb{R}^n_+)\text{, }&\qquad [\mathfrak{b}^{s_0}_{p_0,q_0}(\mathbb{R}^n_+),\mathfrak{b}^{s_1}_{p_1,q_1}(\mathbb{R}^n_+)]_{\theta} = \mathfrak{b}^{s}_{p_\theta,q_\theta}(\mathbb{R}^n_+)\text{, }
\end{align}
whenever $(p_0,q_0),(p_1,q_1),(p,q)\in[1,+\infty]^2$($p\neq 1,+\infty$, when dealing with Sobolev (Bessel potential) spaces), $\theta\in(0,1)$, $s_0\neq s_1$ two real numbers, such that
\begin{align*}
    \left(s,\frac{1}{p_\theta},\frac{1}{q_\theta}\right):= (1-\theta)\left(s_0,\frac{1}{p_0},\frac{1}{q_0}\right)+ \theta\left(s_1,\frac{1}{p_1},\frac{1}{q_1}\right)\text{,}
\end{align*}
with $q_\theta<+\infty$.

A nice property is that the description of the boundary yields the following density results, for all $p\in(1,+\infty)$, $q\in[1,+\infty)$, $s\in\mathbb{R}$,
\begin{align}
    \mathrm{\mathrm{H}}^{s,p}_0(\mathbb{R}^n_+)= \overline{\mathrm{C}_c^\infty(\mathbb{R}^n_+)}^{\lVert \cdot \rVert_{\mathrm{\mathrm{H}}^{s,p}(\mathbb{R}^n)}}\text{,}\quad\text{and}\quad \mathrm{B}^{s}_{p,q,0}(\mathbb{R}^n_+)= \overline{\mathrm{C}_c^\infty(\mathbb{R}^n_+)}^{\lVert \cdot \rVert_{\mathrm{B}^{s}_{p,q}(\mathbb{R}^n)}}\text{. }
\end{align}
One may check \cite[Section~2]{JerisonKenig1995} for the treatment of Sobolev spaces case, the Besov spaces case follows by interpolation argument, see \cite[Theorem~3.4.2]{BerghLofstrom1976}. As a direct consequence, one has from \cite[Proposition~2.9]{JerisonKenig1995} and \cite[Theorem~3.7.1]{BerghLofstrom1976}, that for all $s\in\mathbb{R}$, $p\in(1,+\infty)$, $q\in[1,+\infty)$,
\begin{align}
    (\mathrm{\mathrm{H}}^{s,p}(\mathbb{R}^n_+))' = &\mathrm{H}^{-s,p'}_0(\mathbb{R}^n_+)\text{, }\, (\mathrm{B}^{s}_{p,q}(\mathbb{R}^n_+))'=\mathrm{B}^{-s}_{p',q',0}(\mathbb{R}^n_+) \text{, }\\
    &(\mathrm{B}^{s}_{p,q,0}(\mathbb{R}^n_+))'=\mathrm{B}^{-s}_{p',q'}(\mathbb{R}^n_+)\text{. }
\end{align}

And finally, thanks to the inhomogeneous version of Proposition \ref{prop:SobolevMultiplier}, we also have a particular case of equality of Sobolev spaces, with equivalent norms, for all $p\in(1,+\infty)$, $q\in[1,+\infty]$, $s\in(-1+\frac{1}{p},\frac{1}{p})$,
\begin{align}
    \mathrm{\mathrm{H}}^{s,p}(\mathbb{R}^n_+) = \mathrm{\mathrm{H}}^{s,p}_0(\mathbb{R}^n_+)\text{, }\, \mathrm{B}^{s}_{p,q}(\mathbb{R}^n_+)=\mathrm{B}^{s}_{p,q,0}(\mathbb{R}^n_+) \text{. }
\end{align}

The interested reader may also find an explicit and way more general (and still valid, for the most part of it, in the case of the half-space) treatment for bounded Lipschitz domains in \cite{KaltonMayborodaMitrea2007}, where the Triebel-Lizorkin scale, including Hardy spaces, and other endpoint function spaces are also treated. 

All the results presented above will be used without being mentioned and are assumed to be well known to the reader.

\subsection{Homogeneous function spaces on \texorpdfstring{$\mathbb{R}^n_+$}{Rn+}}

One may expect to recover similar results for the scale of homogeneous Sobolev and Besov as the one mentioned in the subsection \ref{subsec:InhomSpacesRn+}. However, due to the setting involving the use of $\eus{S}'_h(\mathbb{R}^n)$, we have a lack of completeness so that one can no longer use complex interpolation theory and density argument on the whole scale to provide boundedness of linear operators. A first approach we could have in mind is that one would expect Rychkov's extension operator to preserve $\eus{S}'_h$, say $\mathcal{E}(\eus{S}'_h(\mathbb{R}^n_+))\subset \eus{S}'_h(\mathbb{R}^n)$ with \textbf{\textit{homogeneous}} estimates, which is not known yet.

However, if we consider a more naive extension operator like by reflection around the boundary, as in \cite[Chapter~3]{DanchinHieberMuchaTolk2020}, a certain amount of results remains true, up to consider index $s>-1+\tfrac{1}{p}$, provided $p\in(1,+\infty)$. This is what we are going to achieve here: this subsection is devoted to proofs of usual results on homogeneous Sobolev and Besov spaces on $\mathbb{R}^n_+$. 

This subsection contains 3 subparts: the first one is about extension-restriction and density results for our homogeneous Sobolev spaces, from which for the second, we are going to build corresponding ones for Besov spaces, via some ersatz of real interpolation procedure. Both will be used to build the third subpart, which concerns effective interpolation results for our homogeneous Sobolev and Besov spaces. 

\subsubsection{Homogeneous Sobolev spaces}

We start proving the boundedness of extension operators defined by higher order reflection principle but for homogeneous Sobolev spaces with fractional index of regularity. This is done as in \cite[Lemma~3.15, Proposition~3.19]{DanchinHieberMuchaTolk2020}, where it was achieved only for homogeneous Besov spaces.

\begin{proposition}\label{prop:ExtOpHomSobSpaces} For $m\in\mathbb{N}$, there exists a linear extension operator $\mathrm{E}$, depending on $m$, such that for all $p\in(1,+\infty)$, $ -1+\tfrac{1}{p}< s < m+ 1+\tfrac{1}{p} $, so that if either,
\begin{itemize}
    \item $s \geqslant 0$ and $u\in {\mathrm{H}}^{s,p}(\mathbb{R}^n_+)$ ;
    \item $s\in (-1+\tfrac{1}{p},\tfrac{1}{p})$ and $u\in \dot{\mathrm{H}}^{s,p}(\mathbb{R}^n_+)$ ;
\end{itemize}
we have
\begin{align*}
    {\mathrm{E}u}_{|_{\mathbb{R}^n_+}} = u\text{, }
\end{align*}
with the estimate
\begin{align*}
    \left\lVert \mathrm{E}u \right\rVert_{\dot{\mathrm{H}}^{s,p}(\mathbb{R}^n)}\lesssim_{p,s,n,m} \left\lVert u \right\rVert_{\dot{\mathrm{H}}^{s,p}(\mathbb{R}^n_+)}\text{. }
\end{align*}
In particular, $\mathrm{E}\,:\, \dot{\mathrm{H}}^{s,p}(\mathbb{R}^n_+)\longrightarrow\dot{\mathrm{H}}^{s,p}(\mathbb{R}^n)$ extends uniquely to a bounded operator whenever $(\mathcal{C}_{s,p})$ is satisfied.
\end{proposition}

\begin{proof} As in \cite[Lemma~3.15]{DanchinHieberMuchaTolk2020}, let  us  introduce the higher order reflection operator $\mathrm{E}$, defined for all measurable function $u\,:\,\mathbb{R}^n_+\longrightarrow \mathbb{C}$ by
\begin{align*}
    \mathrm{E}u  (x)\,:=\,\left\{\begin{array}{lc}
            u(x) &\text{, if } x\in \mathbb{R}^n_+  \text{, }\\
            \sum_{j=0}^m \alpha_j u(x', -\tfrac{x_n}{j+1}) &\text{, if } x\in \mathbb{R}^n\setminus\mathbb{R}^n_+  \text{. }
        \end{array}\right.
\end{align*}
where, as in \cite[Lemma~3.15]{DanchinHieberMuchaTolk2020},  $x=(x_1,\ldots,x_{n-1},x_n)=(x',x_n)\in\mathbb{R}^{n-1}\times\mathbb{R}$, and $(\alpha_j)_{j\in\llb0,m\rrb}$ is such that $\mathrm{E}$ maps $\mathrm{C}^m$-functions on $\mathbb{R}^n_+$ to $\mathrm{C}^m$-functions on $\mathbb{R}^n$. This is indeed true since $\alpha_j$, $j\in\llb 0,m\rrb$, is chosen so that it satisfies for all $\kappa \in\llb0,m\rrb$,
\begin{align*}
    \sum_{j=0}^{m} \left(\frac{-1}{j+1}\right)^\kappa\alpha_j = 1.
\end{align*}
By construction, the operator $\mathrm{E}$ also maps boundedly $\mathrm{H}^{k,p}(\mathbb{R}^n_+)$ to $\mathrm{H}^{k,p}(\mathbb{R}^n)$ for all $k\in\llb 0, m+1\rrb$. The boundedness of the operator $\mathrm{E}$ from $\mathrm{H}^{s,p}(\mathbb{R}^n_+)$ to $\mathrm{H}^{s,p}(\mathbb{R}^n)$ for all $s\in[0,m+1]$ follows from complex interpolation.

Notice also that Proposition \ref{prop:SobolevMultiplier} and the formulation, given for $x\in\mathbb{R}^{n}$,
\begin{align*}
    \mathrm{E}u  (x) = [\mathbbm{1}_{\mathbb{R}^n_+}u](x) + \sum_{j=0}^m \alpha_j [\mathbbm{1}_{\mathbb{R}^n_+}u](x', -\tfrac{x_n}{j+1})
\end{align*}
implies that $\mathrm{E}\,:\, \dot{\mathrm{H}}^{s,p}(\mathbb{R}^n_+)\longrightarrow \dot{\mathrm{H}}^{s,p}(\mathbb{R}^n)$ is bounded for all $s\in(-1+\tfrac{1}{p},\tfrac{1}{p})$.

Now, for $p\in(1,+\infty)$, $s\in[ 0, m+1+\tfrac{1}{p})$, $s-\frac{1}{p}\notin \mathbb{N}$, $u\in \mathrm{H}^{s,p}(\mathbb{R}^n_+)$, and $\mathrm{E}\,:\, {\mathrm{H}}^{s,p}(\mathbb{R}^n_+)\longrightarrow \dot{\mathrm{H}}^{s,p}(\mathbb{R}^n)$, we can choose $\ell\in\mathbb{N}$ such that $s-\ell\in (-1+\tfrac{1}{p},\tfrac{1}{p})$, so that
\begin{align*}
    \partial_{x_k}^\ell\mathrm{E}u &= \mathrm{E}[\partial_{x_\ell}^\ell u] \text{, provided } k\in\llb 1,n-1\rrb \text{, }\\ \partial_{x_n}^\ell \mathrm{E}u &=  \mathrm{E}^{(\ell)} \partial_{x_n}^\ell u= \sum_{j=0}^m \alpha_j \left(\tfrac{-1}{j+1}\right)^\ell \partial_{x_n}^\ell u(x', -\tfrac{x_n}{j+1})\text{. }
\end{align*}
For the same reasons as in the beginning of the present proof, $\mathrm{E}^{(\ell)}$ maps $\mathrm{H}^{s,p}(\mathbb{R}^n_+)$ to $\mathrm{H}^{s,p}(\mathbb{R}^n)$ for all $s\in[0,m-\ell+1]$, and $\dot{\mathrm{H}}^{s,p}(\mathbb{R}^n_+)$ to $\dot{\mathrm{H}}^{s,p}(\mathbb{R}^n)$ for $s\in(-1+1/p,1/p)$, thanks to Proposition~\ref{prop:SobolevMultiplier}.

From the fact that $\partial_{x_j}^\ell u \in \dot{\mathrm{H}}^{s-\ell,p}(\mathbb{R}^n_{+})$, we deduce
\begin{align} \label{ineq:boundednessPartialkE}
    \left\lVert \mathrm{E} u\right\rVert_{\dot{\mathrm{H}}^{s,p}(\mathbb{R}^n)} \sim_{\ell,p,n} \sum_{j=1}^{n-1}\lVert \partial_{x_j}^\ell \mathrm{E} u\rVert_{\dot{\mathrm{H}}^{s-\ell,p}(\mathbb{R}^n)} + \lVert \mathrm{E}^{(\ell)}\partial_{x_n}^\ell u\rVert_{\dot{\mathrm{H}}^{s-\ell,p}(\mathbb{R}^n)} \lesssim_{s,\ell,p,n,m} \sum_{j=1}^{n}\lVert \partial_{x_j}^\ell u\rVert_{\dot{\mathrm{H}}^{s-\ell,p}(\mathbb{R}^n)}\text{. }
\end{align}
To be more synthetic, we have obtained
\begin{align*}
    \lVert \mathrm{E}u \rVert_{\dot{\mathrm{H}}^{s,p}(\mathbb{R}^n)}\lesssim_{p,k,n,m} \lVert u \rVert_{\dot{\mathrm{H}}^{s,p}(\mathbb{R}^n_+)}\text{, }
\end{align*}
so that $\mathrm{E}\,:\,\dot{\mathrm{H}}^{s,p}(\mathbb{R}^n_+)\longrightarrow \dot{\mathrm{H}}^{s,p}(\mathbb{R}^n)$ is bounded on the subspace ${\mathrm{H}}^{s,p}(\mathbb{R}^n_+)$, in particular it extends uniquely to a bounded linear operator on the whole $\dot{\mathrm{H}}^{s,p}(\mathbb{R}^n_+)$ when it is complete, i.e. $s<\frac{n}{p}$, this follows from the fact that $\eus{S}(\overline{\mathbb{R}^n_+})\subset{\mathrm{H}}^{s,p}(\mathbb{R}^n_+)$ is dense in $\dot{\mathrm{H}}^{s,p}(\mathbb{R}^n_+)$.

It remains to cover cases when $s-\frac{1}{p}\in \llb 0,m\rrb$. To do so, we want to reproduce the above procedure, proving first that $\mathrm{E}$ (\textit{resp.} $\mathrm{E}^{(\ell)}$, $\ell\in\llb 1,m\rrb$) is bounded from $\dot{\mathrm{H}}^{\tfrac{1}{p},p}(\mathbb{R}^n_+)$ to $\dot{\mathrm{H}}^{\tfrac{1}{p},p}(\mathbb{R}^n)$, via some complex interpolation scheme.

Now let  $p_0,p_1\in(1,+\infty)$, $p_1<n$, $\theta\in(0,1)$. Notice that, here, $n\geqslant 2$ is a really important assumption, otherwise there wouldn't any such $p_1$. Consider $u\in [\mathrm{L}^{p_0}(\mathbb{R}^n_+),\dot{\mathrm{H}}^{1,p_1}(\mathbb{R}^n_+)]_{\theta}$. Let $f\in F(\mathrm{L}^{p_0}(\mathbb{R}^n_+),\dot{\mathrm{H}}^{1,p_1}(\mathbb{R}^n_+))$, such that $f(\theta)=u$, it follows from the previous considerations that $\mathrm{E}f\in F(\mathrm{L}^{p_0}(\mathbb{R}^n),\dot{\mathrm{H}}^{1,p_1}(\mathbb{R}^n))$. Thus, from Theorem \ref{thm:InterpHomSpacesRn}, one has
\begin{align*}
    \mathrm{E}f(\theta)\in \dot{\mathrm{H}}^{\theta,p}(\mathbb{R}^n)\text{, where } \left(\theta,\frac{1}{p}\right):= (1-\theta)\left(0,\frac{1}{p_0}\right)+ \theta\left(1,\frac{1}{p_1}\right)\text{. }
\end{align*}
So $u= \mathrm{E}f(\theta)_{|_{\mathbb{R}^n_+}}\in \dot{\mathrm{H}}^{\theta,p}(\mathbb{R}^n_+)$ with the norm estimate
\begin{align*}
    \lVert u \rVert_{\dot{\mathrm{H}}^{\theta,p}(\mathbb{R}^n_+)}\lesssim_{m_1,p,n} \lVert u \rVert_{[\mathrm{L}^{p_0}(\mathbb{R}^n_+),\dot{\mathrm{H}}^{1,p_1}(\mathbb{R}^n_+)]_{\theta}}
\end{align*}
which is a direct consequence of the definition of restriction space, the equivalence of the complex interpolation norm \eqref{eq:complexInterpHomHspRn} from Theorem \ref{thm:InterpHomSpacesRn}, the definition of the complex interpolation norm, and then of the boundedness of $\mathrm{E}$ from  $\mathrm{L}^{p_0}(\mathbb{R}^n)$ to $\mathrm{L}^{p_0}(\mathbb{R}^n_+)$ and from $\dot{\mathrm{H}}^{1,p_1}(\mathbb{R}^n)$ to $\dot{\mathrm{H}}^{1,p_1}(\mathbb{R}^n_+)$. Now, if $u\in \dot{\mathrm{H}}^{\theta,p}(\mathbb{R}^n_+)$, by definition of restriction spaces there exists $U\in \dot{\mathrm{H}}^{\theta,p}(\mathbb{R}^n)$, such that
\begin{align*}
    U_{|_{\mathbb{R}^n_+}}=u,\quad \text{and}\quad\frac{1}{2} \lVert U \rVert_{\dot{\mathrm{H}}^{\theta,p}(\mathbb{R}^n)}\leqslant \lVert u \rVert_{\dot{\mathrm{H}}^{\theta,p}(\mathbb{R}^n_+)}\leqslant \lVert U \rVert_{\dot{\mathrm{H}}^{\theta,p}(\mathbb{R}^n)}\text{. }
\end{align*}
By Theorem \ref{thm:InterpHomSpacesRn}, there exists $f \in F(\mathrm{L}^{p_0}(\mathbb{R}^n),\dot{\mathrm{H}}^{1,p_1}(\mathbb{R}^n))$ such that $f(\theta)=U$, we deduce
$f(\cdot)_{|_{\mathbb{R}^n_+}}\in F(\mathrm{L}^{p_0}(\mathbb{R}^n_+),\dot{\mathrm{H}}^{1,p_1}(\mathbb{R}^n_+))$, so $u=f(\theta)_{|_{\mathbb{R}^n_+}}\in [\mathrm{L}^{p_0}(\mathbb{R}^n_+),\dot{\mathrm{H}}^{1,p_1}(\mathbb{R}^n_+)]_{\theta}$ with the following estimate which is a direct consequence of the definition of function spaces by restriction, and complex interpolation spaces,
\begin{align*}
    \lVert u \rVert_{[\mathrm{L}^{p_0}(\mathbb{R}^n_+),\dot{\mathrm{H}}^{1,p_1}(\mathbb{R}^n_+)]_{\theta}} \lesssim \lVert u \rVert_{\dot{\mathrm{H}}^{\theta,p}(\mathbb{R}^n_+)}.
\end{align*}
Hence, homogeneous (Riesz potential) Sobolev spaces on the half-space are still a complex interpolation scale provided that $p\in(1,+\infty)$, $s\in[0,1]$, $(\mathcal{C}_{s,p})$ being satisfied, so the boundedness of $\mathrm{E}\,:\, \dot{\mathrm{H}}^{\theta,p}(\mathbb{R}^n_+)\rightarrow \dot{\mathrm{H}}^{\theta,p}(\mathbb{R}^n)$ follows by interpolation.

In particular, $\mathrm{E}\,:\, \dot{\mathrm{H}}^{s,p}(\mathbb{R}^n_+) \longrightarrow \dot{\mathrm{H}}^{s,p}(\mathbb{R}^n)$ is bounded for all $s\in(-1+\tfrac{1}{p},\tfrac{1}{p}]$. Hence, the result has been proved for $s-\tfrac{1}{p}=0$. The same result is obtained for $\mathrm{E}^{(\ell)}$, provided $\ell \in\llb 1,m\rrb$.

Now, let $p\in (1,+\infty)$, $s-\frac{1}{p}\in \llb 1,m\rrb$, for $u\in{\mathrm{H}}^{s,p}(\mathbb{R}^n_+)$, we have $\mathrm{E}u\in \mathrm{H}^{s,p}(\mathbb{R}^n)$, $\nabla^\ell \mathrm{E}u \in \dot{\mathrm{H}}^{s-\ell,p}(\mathbb{R}^n)$, $s-\ell = \tfrac{1}{p}$, so that, similarly as in \eqref{ineq:boundednessPartialkE},
\begin{align*}
    \lVert \mathrm{E}u \rVert_{\dot{\mathrm{H}}^{s,p}(\mathbb{R}^n)} \lesssim_{s,p,n,\ell} \lVert u \rVert_{\dot{\mathrm{H}}^{s,p}(\mathbb{R}^n_+)}\text{. }
\end{align*}
Therefore, we have obtained the desired estimate and can conclude about the boundedness of $\mathrm{E}$ via density argument whenever $(\mathcal{C}_{s,p})$ is satisfied. 
\end{proof}

In the proof of Proposition \ref{prop:ExtOpHomSobSpaces}, we used boundedness of derivatives, \textit{i.e.}, for all $p\in(1,+\infty)$, $s\in\mathbb{R}$, $u\in \dot{\mathrm{H}}^{s,p}(\mathbb{R}^n_+)$, $m\in\mathbb{N}$,
\begin{equation}\label{eq:normcontrolHomogeneousSobolev}
    \lVert \nabla^m u\rVert_{\dot{\mathrm{H}}^{s-m,p}(\mathbb{R}^n_+)} \lesssim_{p,s,n,m} \lVert u\rVert_{\dot{\mathrm{H}}^{s,p}(\mathbb{R}^n_+)} \text{. }
\end{equation}
The estimate above is a direct consequence of definition of function spaces by restriction and can be turned into an equivalence under some additional assumptions.

\begin{proposition}\label{prop:EqNormNablakHsp} Let $p\in(1,+\infty)$, $k\in\llb 1,+\infty\llb$, $s>k-1+\tfrac{1}{p} $, for all $u\in {\mathrm{H}}^{s,p}(\mathbb{R}^n_+)$,
\begin{align*}
    \sum_{j=1}^{n} \lVert \partial_{x_j}^k u \rVert_{\dot{\mathrm{H}}^{s-k,p}(\mathbb{R}^n_+)} \sim_{s,k,p,n}\lVert \nabla^k u \rVert_{\dot{\mathrm{H}}^{s-k,p}(\mathbb{R}^n_+)} \sim_{s,k,p,n} \lVert u \rVert_{\dot{\mathrm{H}}^{s,p}(\mathbb{R}^n_+)}\text{. }
\end{align*}
In particular, $\lVert \nabla^k \cdot \rVert_{\dot{\mathrm{H}}^{s-k,p}(\mathbb{R}^n_+)}$ and $\sum_{j=1}^{n} \lVert \partial_{x_j}^k \cdot \rVert_{\dot{\mathrm{H}}^{s-k,p}(\mathbb{R}^n_+)}$ provide equivalent norms on $\dot{\mathrm{H}}^{s,p}(\mathbb{R}^n_+)$, whenever $(\mathcal{C}_{s-k,p})$ is satisfied.
\end{proposition}

\begin{proof}Let us  prove it for $k=1$, the higher order case can be achieved similarly. Consider $p\in (1,+\infty)$, $s>\tfrac{1}{p} $, for $u\in {\mathrm{H}}^{s,p}(\mathbb{R}^n_+)$, we have $\mathrm{E}u\in \dot{\mathrm{H}}^{s,p}(\mathbb{R}^n)$, where $\mathrm{E}$ is an extension operator provided by Proposition \ref{prop:ExtOpHomSobSpaces} (for some big enough $m\geqslant 1$), $\nabla\mathrm{E}u\in\dot{\mathrm{H}}^{s-1,p}(\mathbb{R}^n)$, with $s-1> -1+\tfrac{1}{p}$. We can write on $\overline{\mathbb{R}^n_+}^c$
\begin{align*}
    \partial_{x_\ell}\mathrm{E}u = \mathrm{E}[\partial_{x_\ell} u] \text{, provided } \ell\in\llb 1,n-1\rrb \text{, and } \partial_{x_n} \mathrm{E}u = \sum_{j=0}^m \alpha_j \left(\tfrac{-1}{j+1}\right) \partial_{x_n} u(x', -\tfrac{x_n}{j+1})\text{. }
\end{align*}
Hence, we can use definition of restriction space, apply point \textit{(iii)} in Proposition \ref{prop:PropertiesHomSobolevSpacesRn}, and the boundedness of $\mathrm{E}$, since $m$ is large enough, to obtain,
\begin{align*}
    \lVert u \rVert_{\dot{\mathrm{H}}^{s,p}(\mathbb{R}^n_+)} \leqslant \lVert \mathrm{E}u \rVert_{\dot{\mathrm{H}}^{s,p}(\mathbb{R}^n)} \lesssim_{s,p,n} \lVert \nabla \mathrm{E}u \rVert_{\dot{\mathrm{H}}^{s-1,p}(\mathbb{R}^n)} \lesssim_{s,p,n,m} \lVert \nabla u \rVert_{\dot{\mathrm{H}}^{s-1,p}(\mathbb{R}^n_+)}\text{. }
\end{align*}
Therefore by \eqref{eq:normcontrolHomogeneousSobolev}, the equivalence of norms on $\dot{\mathrm{H}}^{s,p}(\mathbb{R}^n_+)$ holds by density when $(\mathcal{C}_{s-k,p})$ is true.
\end{proof}

The next proposition is about identifying intersection of homogeneous Sobolev spaces on $\mathbb{R}^n_+$, and give a dense subspace. As we can see later, this will help for real interpolation.

\begin{proposition}\label{prop:IntersecHomHspRn+} Let $p_j\in(1,+\infty)$, $s_j> -1+ \tfrac{1}{p_j}$, $j\in\{0,1\}$, if $(\mathcal{C}_{s_0,p_0})$ is satisfied then the following equality of vector spaces holds with equivalence of norms 
\begin{align*}
    \dot{\mathrm{H}}^{s_0,p_0}(\mathbb{R}^n_+)\cap \dot{\mathrm{H}}^{s_1,p_1}(\mathbb{R}^n_+) = [\dot{\mathrm{H}}^{s_0,p_0}\cap \dot{\mathrm{H}}^{s_1,p_1}](\mathbb{R}^n_+)\text{. }
\end{align*}
In particular, $\dot{\mathrm{H}}^{s_0,p_0}(\mathbb{R}^n_+)\cap \dot{\mathrm{H}}^{s_1,p_1}(\mathbb{R}^n_+)$ is a Banach space which admits $\eus{S}_0(\overline{\mathbb{R}^n_+})$ as a dense subspace.
\end{proposition}

\begin{proof} Let $p\in(1,+\infty)$, $s_0,s_1\in\mathbb{R}$, such that $(\mathcal{C}_{s_0,p_0})$. By definition of restriction spaces and Lemma \ref{lem:IntersecHomHsp}, $[\dot{\mathrm{H}}^{s_0,p_0}\cap \dot{\mathrm{H}}^{s_1,p_1}](\mathbb{R}^n_+)$ is complete and admits  $\eus{S}_0(\overline{\mathbb{R}^n_+})$ as a dense subspace. The following continuous embedding also holds by definition,
\begin{align*}
    [\dot{\mathrm{H}}^{s_0,p_0}\cap \dot{\mathrm{H}}^{s_1,p_1}](\mathbb{R}^n_+)\hookrightarrow\dot{\mathrm{H}}^{s_0,p_0}(\mathbb{R}^n_+)\cap \dot{\mathrm{H}}^{s_1,p_1}(\mathbb{R}^n_+)\text{. }
\end{align*}
Hence, it suffices to prove the reverse one. To do so, let  us choose $\ell\in\mathbb{N}$ such that $(\mathcal{C}_{s_1-\ell,p_1})$ is satisfied, and $s_1-\ell> -1+ \tfrac{1}{p_1}$, then choosing $\mathrm{E}$ from Proposition \ref{prop:ExtOpHomSobSpaces} with $m+1+\tfrac{1}{p_j}> s_j$, $j\in\{0,1\}$ ($m$ big enough), for all $j\in\llb 1,n\rrb$, and all $u\in \dot{\mathrm{H}}^{s_0,p_0}(\mathbb{R}^n_+)\cap \dot{\mathrm{H}}^{s_1,p_1}(\mathbb{R}^n_+)$, $\mathrm{E}u$ makes sense in $\dot{\mathrm{H}}^{s_0,p_0}(\mathbb{R}^n)$, then in $\eus{S}'_h(\mathbb{R}^n)$, and one may use an estimate similar to \eqref{ineq:boundednessPartialkE}, to deduce
\begin{align*}
    \sum_{k=1}^{n}\lVert \partial_{x_k}^\ell \mathrm{E} u \rVert_{\dot{\mathrm{H}}^{s_1-\ell,p_1}(\mathbb{R}^n)} = \sum_{k=1}^{n-1} \lVert  \mathrm{E} \partial_{x_k}^\ell u \rVert_{\dot{\mathrm{H}}^{s_1-\ell,p_1}(\mathbb{R}^n)} + \lVert  \mathrm{E}^{(\ell)} \partial_{x_n}^\ell u \rVert_{\dot{\mathrm{H}}^{s_1-\ell,p_1}(\mathbb{R}^n)} \lesssim_{s_1,m,\ell}^{p_1,n}  \lVert  u \rVert_{\dot{\mathrm{H}}^{s_1,p_1}(\mathbb{R}^n_+)}\text{. }
\end{align*}
The above operator $\mathrm{E}^{(\ell)}$ is given via the identity $\partial_{x_n}^\ell \mathrm{E}  =\mathrm{E}^{(\ell)} \partial_{x_n}^\ell$. Hence, it follows that for all $u\in \dot{\mathrm{H}}^{s_0,p_0}(\mathbb{R}^n_+)\cap \dot{\mathrm{H}}^{s_1,p_1}(\mathbb{R}^n_+)$,
\begin{align*}
    \lVert \mathrm{E} u \rVert_{\dot{\mathrm{H}}^{s_0,p_0}(\mathbb{R}^n)} + \sum_{k=1}^{n}\lVert \partial_{x_k}^\ell \mathrm{E} u \rVert_{\dot{\mathrm{H}}^{s_1-\ell,p_1}(\mathbb{R}^n)} \lesssim_{s_0,s_1,m,\ell}^{p_0,p_1,n}   \lVert  u \rVert_{\dot{\mathrm{H}}^{s_0,p_0}(\mathbb{R}^n_+)} + \lVert  u \rVert_{\dot{\mathrm{H}}^{s_1,p_1}(\mathbb{R}^n_+)}\text{. }
\end{align*}
In particular, since $\mathrm{E}u \in \eus{S}_h'(\mathbb{R}^n)$, and by uniqueness of representation of $\partial_{x_j}^\ell \mathrm{E} u$ in $\eus{S}'(\mathbb{R}^n)$, we deduce from point \textit{(iii)} in Proposition \ref{prop:PropertiesHomSobolevSpacesRn} that $\mathrm{E}u\in \dot{\mathrm{H}}^{s_0,p_0}(\mathbb{R}^n)\cap \dot{\mathrm{H}}^{s_1,p_1}(\mathbb{R}^n)$.

Thus $u\in [\dot{\mathrm{H}}^{s_0,p_0}\cap \dot{\mathrm{H}}^{s_1,p_1}](\mathbb{R}^n_+)$, and by definition of restriction spaces,
\begin{align*}
    \lVert u \rVert_{[\dot{\mathrm{H}}^{s_0,p_0}\cap \dot{\mathrm{H}}^{s_1,p_1}](\mathbb{R}^n_+)} \leqslant \lVert \mathrm{E} u \rVert_{\dot{\mathrm{H}}^{s_0,p_0}(\mathbb{R}^n)} + \lVert \mathrm{E} u \rVert_{\dot{\mathrm{H}}^{s_1,p_1}(\mathbb{R}^n)} \lesssim_{s_0,s_1,m,\ell}^{p_0,p_1,n}   \lVert  u \rVert_{\dot{\mathrm{H}}^{s_0,p_0}(\mathbb{R}^n_+)} + \lVert  u \rVert_{\dot{\mathrm{H}}^{s_1,p_1}(\mathbb{R}^n_+)}\text{. }
\end{align*}
This proves the claim.
\end{proof}

So one can deduce the following corollary, which allows separate homogeneous estimates for intersection of homogeneous Sobolev spaces on $\mathbb{R}^n_+$. Since the estimates below are decoupled, it provides an ersatz of extension-restriction operators for homogeneous Sobolev spaces of higher order, thanks to the taken intersection yielding a complete space. For instance, this will be of use to circumvent the lack of completeness when we will want to (real-)interpolate between a “higher” order homogeneous Sobolev space, and one that is known to be complete.
\begin{corollary}\label{cor:ExtOpIntersecHomHspRn+}Let $p_j\in(1,+\infty)$, $s_j> -1+\tfrac{1}{p_j}$, $j\in\{0,1\}$, such that $(\mathcal{C}_{s_0,p_0})$ is satisfied, consider $m\in\mathbb{N}$ such that $ s_j<m+1+\frac{1}{p_j}$, and the extension operator $\mathrm{E}$ given by Proposition \ref{prop:ExtOpHomSobSpaces}.

Then for all $u\in \dot{\mathrm{H}}^{s_0,p_0}(\mathbb{R}^n_+)\cap \dot{\mathrm{H}}^{s_1,p_1}(\mathbb{R}^n_+)$, we have  $\mathrm{E}u\in \dot{\mathrm{H}}^{s_j,p_j}(\mathbb{R}^n)$, $j\in\{0,1\}$, with the estimate
\begin{align*}
    \lVert \mathrm{E} u \rVert_{\dot{\mathrm{H}}^{s_j,p_j}(\mathbb{R}^n)} \lesssim_{s_j,p_j,m,n}   \lVert  u \rVert_{\dot{\mathrm{H}}^{s_j,p_j}(\mathbb{R}^n_+)}\text{. }
\end{align*}
\end{corollary}

Corollary \ref{cor:ExtOpIntersecHomHspRn+} and the proof of Proposition \ref{prop:EqNormNablakHsp} lead to

\begin{corollary}\label{cor:EqNormNablakmHspHaq} Let $p_j\in(1,+\infty)$, $m_j\in\llb 1,+\infty\llb$, $s_1>m_1-1+\tfrac{1}{p_1} $, $j\in\{ 0,1\}$, such that $(\mathcal{C}_{s_0,p_0})$ is satisfied. Then for all $u\in \dot{\mathrm{H}}^{s_0,p_0}(\mathbb{R}^n_+)\cap \dot{\mathrm{H}}^{s_1,p_1} (\mathbb{R}^n_+)$,
\begin{align*}
    \sum_{k=1}^{n} \lVert \partial_{x_k}^{m_1} u \rVert_{\dot{\mathrm{H}}^{s_1-m_1,p_1}(\mathbb{R}^n_+)} \sim_{s_1,m_1,p_1,n} \lVert \nabla^{m_1} u \rVert_{\dot{\mathrm{H}}^{s_1-m_1,p_1}(\mathbb{R}^n_+)} \sim_{s_1,m_1,p_j,n} \lVert u \rVert_{\dot{\mathrm{H}}^{s_1,p_1}(\mathbb{R}^n_+)}\,\text{. }
\end{align*}
\end{corollary}

Now, one may also be interested into Sobolev spaces with $0$-boundary condition, we introduce a projection operator that allows to deal with the interpolation property, and to recover, later on, some appropriate density results.

\begin{lemma}\label{lemma:ProjHsp0}Let $p\in(1,+\infty)$, $s\in\mathbb{R}$, $m\in\mathbb{N}$, such that $-1+\tfrac{1}{p}<s<m+1+\tfrac{1}{p}$, then there exists a bounded projection operator $\mathcal{P}_0$, depending on $m$, such that it maps $\mathrm{H}^{s,p}(\mathbb{R}^n)$ to $\mathrm{H}^{s,p}_0(\mathbb{R}^n_+)$, which also satisfies that :

If either one of the following conditions holds
\begin{itemize}
    \item $s\geqslant 0$ and $u\in {\mathrm{H}}^{s,p}(\mathbb{R}^n)$;
    \item $s\in (-1+\tfrac{1}{p},\tfrac{1}{p})$ and $u\in \dot{\mathrm{H}}^{s,p}(\mathbb{R}^n)$;
\end{itemize}
then we have the estimate
\begin{align*}
    \lVert \mathcal{P}_0 u \rVert_{\dot{\mathrm{H}}^{s,p}(\mathbb{R}^n)} \lesssim_{s,m,p,n}   \lVert  u \rVert_{\dot{\mathrm{H}}^{s,p}(\mathbb{R}^n)}\text{. }
\end{align*}
In particular, $\mathcal{P}_0$ extends uniquely to a bounded projection from $\dot{\mathrm{H}}^{s,p}(\mathbb{R}^n)$ to $\dot{\mathrm{H}}^{s,p}_0(\mathbb{R}^n_+)$ whenever $(\mathcal{C}_{s,p})$ is satisfied.
\end{lemma}

\begin{proof}Let $p\in(1,+\infty)$, $s> -1+\tfrac{1}{p}$, $m\in\mathbb{N}$, such that $s<m+1+\tfrac{1}{p}$. Then we consider the operator $\mathrm{E}$ given by Proposition \ref{prop:ExtOpHomSobSpaces}, but we modify it into an operator $\mathrm{E}^{-}$, for any measurable function $u\,:\,\mathbb{R}^n_{-}\longrightarrow \mathbb{C}$, we set for almost every $x\in\mathbb{R}^n$
\begin{align*}
    \mathrm{E}^{-}u  (x)\,:=\,\left\{\begin{array}{lc}
            u(x) &\text{, if } x\in \mathbb{R}^n_-  \text{, }\\
            \sum_{j=0}^m \alpha_j u(x', -\tfrac{x_n}{j+1}) &\text{, if } x\in \mathbb{R}^n\setminus\mathbb{R}^n_-  \text{. }
        \end{array}\right.
\end{align*}
Hence, for any measurable function $u\,:\,\mathbb{R}^n\longrightarrow \mathbb{C}$, we set for almost every $x\in\mathbb{R}^n$,
\begin{align*}
    \mathcal{P}_0u:= u - \mathrm{E}^{-}[\mathbbm{1}_{\mathbb{R}^n_-}u]\text{.}
\end{align*}
The fact that $\mathcal{P}_0^2=\mathcal{P}_0$ is clear by definition, and we have $\mathcal{P}_0 \mathrm{H}^{s,p}(\mathbb{R}^n) \subset \mathrm{H}^{s,p}_0(\mathbb{R}^n_+)$, and that ${\mathcal{P}_0}_{|_{\mathrm{H}^{s,p}_0(\mathbb{R}^n_+)}}=\mathrm{I}$.  The boundedness properties, as claimed, follow from Proposition~\ref{prop:SobolevMultiplier} and Proposition~\ref{prop:ExtOpHomSobSpaces}.
\end{proof}

As well as the extension operator given by higher order reflection principle, the projection operator on homogeneous Sobolev spaces including the “$0$-boundary condition” satisfies homogeneous estimates on intersection spaces. The proof is a direct consequence of Proposition \ref{prop:IntersecHomHspRn+} and its formula introduced in the proof of Lemma \ref{lemma:ProjHsp0}.

\begin{corollary}\label{cor:ProjIntersecHomHsp0}Let $p_j\in(1,+\infty)$, $s_j> -1+\tfrac{1}{p_j}$, $j\in\{0,1\}$, $m\in\mathbb{N}$, such that $(\mathcal{C}_{s_0,p_0})$ is satisfied and $ s_j < m+1+\tfrac{1}{p_j}$, and consider the projection operator $\mathcal{P}_0$ given by Lemma \ref{lemma:ProjHsp0}.

Then for all $u\in \dot{\mathrm{H}}^{s_0,p_0}(\mathbb{R}^n)\cap \dot{\mathrm{H}}^{s_1,p_1}(\mathbb{R}^n)$, we have  $\mathcal{P}_0u\in \dot{\mathrm{H}}^{s_j,p_j}_0(\mathbb{R}^n_+)$, $j\in\{0,1\}$, with the estimate
\begin{align*}
    \lVert \mathcal{P}_0 u \rVert_{\dot{\mathrm{H}}^{s_j,p_j}(\mathbb{R}^n)} \lesssim_{s_j,m,p,n}   \lVert  u \rVert_{\dot{\mathrm{H}}^{s_j,p_j}(\mathbb{R}^n)}\text{. }
\end{align*}
\end{corollary}

We still have Sobolev embeddings by definition of function spaces by restriction.

\begin{proposition}\label{prop:SobolevEmbeddingsRn+} Let $p,q\in(1,+\infty)$, $s\in[0,n)$, such that
\begin{align*}
    \frac{1}{q}=\frac{1}{p}-\frac{s}{n}\text{.}
\end{align*}
We have the estimates,
\begin{align}
    \lVert  u \rVert_{\mathrm{L}^q(\mathbb{R}^n_+)} \lesssim_{n,s,p,q} \lVert  u \rVert_{\dot{\mathrm{H}}^{s,p}(\mathbb{R}^n_+)}\text{, } &\forall u\in \dot{\mathrm{H}}^{s,p}(\mathbb{R}^n_+)\text{, }\label{ineq:SobEmbeddingRn+1}\\
    \lVert  u \rVert_{\dot{\mathrm{H}}^{-s,q}_0(\mathbb{R}^n_+)} \lesssim_{n,s,p,q} \lVert  u \rVert_{\mathrm{L}^p(\mathbb{R}^n_+)}\text{, } &\forall u\in \mathrm{L}^p(\mathbb{R}^n_+)\label{ineq:SobEmbeddingRn+2}\text{, }
\end{align}
for which each underlying embedding is dense.
\end{proposition}

\begin{proof} First, let us recall the Hardy-Littlewood-Sobolev inequality from \textit{(i)} in Proposition \ref{prop:PropertiesHomSobolevSpacesRn}, which says that 
\begin{align*}
    \lVert  u \rVert_{\dot{\mathrm{H}}^{-s,q}(\mathbb{R}^n)} \lesssim_{n,s,p,q} \lVert  u \rVert_{\mathrm{L}^p(\mathbb{R}^n)}\text{, } &\forall u\in \mathrm{L}^p(\mathbb{R}^n)\text{. }
\end{align*}
Hence, the embedding \eqref{ineq:SobEmbeddingRn+2} is a direct consequence of plugging $\tilde{u}$ the extension to the whole $\mathbb{R}^n$ of $u\in \mathrm{L}^p(\mathbb{R}^n_+)$.

The embedding \eqref{ineq:SobEmbeddingRn+1} is a direct consequence of \textit{(i)} in Proposition \ref{prop:PropertiesHomSobolevSpacesRn} and function spaces defined by restriction. Indeed, for $u\in \dot{\mathrm{H}}^{s,p}(\mathbb{R}^n_+)$, we have for any extension $U\in\dot{\mathrm{H}}^{s,p}(\mathbb{R}^n)\subset \mathrm{L}^q(\mathbb{R}^n)$ such that $u=U_{|_{\mathbb{R}^n_+}}\in\mathrm{L}^q(\mathbb{R}^n_+)$ the estimate
\begin{align*}
    \lVert  u \rVert_{\mathrm{L}^q(\mathbb{R}^n_+)} \leqslant \lVert  U \rVert_{\mathrm{L}^q(\mathbb{R}^n)} \lesssim_{s,p,q,n} \lVert U \rVert_{\dot{\mathrm{H}}^{s,p}(\mathbb{R}^n)} \text{. }
\end{align*}
Looking at the infimum on all such $U$ gives the result.

The density for the first embedding follows from the fact that $\eus{S}_0(\overline{\mathbb{R}^n_+})\subset \dot{\mathrm{H}}^{s,p}(\mathbb{R}^n_+)$ is dense in $\mathrm{L}^q(\mathbb{R}^n_+)$. The density in the second case, follows from the canonical embedding,
\begin{align*}
    \mathrm{L}^{p}(\mathbb{R}^n_+) \hookrightarrow \dot{\mathrm{H}}^{-s,q}_0(\mathbb{R}^n_+) \hookrightarrow \mathrm{H}^{-s,q}_0(\mathbb{R}^n_+)\text{,}
\end{align*}
which turn, by duality into embeddings,
\begin{align*}
\mathrm{H}^{s,q'}(\mathbb{R}^n_+) \hookrightarrow (\dot{\mathrm{H}}^{-s,q}_0(\mathbb{R}^n_+))' \hookrightarrow \mathrm{L}^{p'}(\mathbb{R}^n_+) \text{.}
\end{align*}
In particular, the following is a dense embedding
\begin{align*}
    (\dot{\mathrm{H}}^{-s,q}_0(\mathbb{R}^n_+))' \hookrightarrow \mathrm{L}^{p'}(\mathbb{R}^n_+)
\end{align*}
Hence, by reflexivity, the one below also is
\begin{align*}
   \mathrm{L}^{p}(\mathbb{R}^n_+) \hookrightarrow \dot{\mathrm{H}}^{-s,q}_0(\mathbb{R}^n_+) \text{.}
\end{align*}
\end{proof}

Now, all the ingredients are there in order to build the main usual density result for our $0$-boundary conditions homogeneous Sobolev spaces.

\begin{proposition}\label{prop:densityCcinftyHsp0} The space $\mathrm{C}_c^\infty(\mathbb{R}^n_+)$ is dense in the spaces
\begin{enumerate}
    \item $\dot{\mathrm{H}}^{s,p}_0(\mathbb{R}^n_+)$, when $p\in(1,+\infty)$, $s\in (-\frac{n}{p'},\frac{n}{p})$,
    \item $\dot{\mathrm{H}}^{s_0,p_0}_0(\mathbb{R}^n_+)\cap\dot{\mathrm{H}}^{s_1,p_1}_0(\mathbb{R}^n_+)$, when $p_j\in(1,+\infty)$, $s_j\geqslant 0$, $j\in\{0,1\}$, such that $(\mathcal{C}_{s_0,p_0})$ is satisfied.
\end{enumerate}
\end{proposition}

\begin{proof} \textbf{Step 1:} Point \textit{(i)} with $s\geqslant 0$. Let $p\in (1,+\infty)$, and assume that $(\mathcal{C}_{s,p})$ is true, and consider $u\in \dot{\mathrm{H}}^{s,p}_0(\mathbb{R}^n_+)$.

In particular, we have $u\in \dot{\mathrm{H}}^{s,p}(\mathbb{R}^n)$. Hence, there exists $(u_k)_{k\in\mathbb{N}}\subset {\mathrm{H}}^{s,p}(\mathbb{R}^n)$ such that
\begin{align*}
    u_k \xrightarrow[k\rightarrow +\infty]{} u \text{ in } \dot{\mathrm{H}}^{s,p}(\mathbb{R}^n) \text{.}
\end{align*}
Thus, it follows from Lemma \ref{lemma:ProjHsp0}, that $(\mathcal{P}_0 u_k)_{k\in\mathbb{N}} \subset {\mathrm{H}}^{s,p}_0(\mathbb{R}^n_+)\subset \dot{\mathrm{H}}^{s,p}_0(\mathbb{R}^n_+)$ converge to $\mathcal{P}_0 u = u$ in $\dot{\mathrm{H}}^{s,p}(\mathbb{R}^n)$. For $\varepsilon>0$, there exists some $k_0$, such that for all $k\geqslant k_0$, we have
\begin{align*}
    \lVert  u - \mathcal{P}_0 u_k \rVert_{\dot{\mathrm{H}}^{s,p}_0(\mathbb{R}^n_+)} <\varepsilon\text{.}
\end{align*}
Now, we use density of $\mathrm{C}_c^\infty(\mathbb{R}^n_+)$ in ${\mathrm{H}}^{s,p}_0(\mathbb{R}^n_+)$, \cite[Remark~2.7]{JerisonKenig1995}, to assert that there exists $w\in \mathrm{C}_c^\infty(\mathbb{R}^n_+)$ so that,
\begin{align*}
    \lVert  \mathcal{P}_0 u_k - w \rVert_{\dot{\mathrm{H}}^{s,p}_0(\mathbb{R}^n_+)}\leqslant\lVert  \mathcal{P}_0 u_k - w \rVert_{{\mathrm{H}}^{s,p}_0(\mathbb{R}^n_+)} <\varepsilon\text{.}
\end{align*}
This proves the density of  $\mathrm{C}_c^\infty(\mathbb{R}^n_+)$ in the space $\dot{\mathrm{H}}^{s,p}_0(\mathbb{R}^n_+)$, since
\begin{align*}
    \lVert  u - w \rVert_{\dot{\mathrm{H}}^{s,p}_0(\mathbb{R}^n_+)}\leqslant \lVert  u - \mathcal{P}_0 u_k \rVert_{\dot{\mathrm{H}}^{s,p}_0(\mathbb{R}^n_+)} + \lVert  \mathcal{P}_0 u_k - w \rVert_{\dot{\mathrm{H}}^{s,p}_0(\mathbb{R}^n_+)} <2\varepsilon\text{.}
\end{align*}

\textbf{Step 2:} Now for the second part of \textit{(i)}, let us consider $s\in(0,\frac{n}{p'})$. For $u\in \dot{\mathrm{H}}^{-s,p}_0(\mathbb{R}^n_+)$, applying Proposition \ref{prop:SobolevEmbeddingsRn+}, for $\varepsilon >0$ there exists a function $v\in\mathrm{L}^q(\mathbb{R}^n_+)$, (with $\frac{1}{p}=\frac{1}{q}-\frac{s}{n}$) such that,
\begin{align*}
    \lVert  u - v \rVert_{\dot{\mathrm{H}}^{-s,p}_0(\mathbb{R}^n_+)} < \varepsilon\text{.}
\end{align*}
But recalling that $\mathrm{C}_c^\infty(\mathbb{R}^n_+)$ is dense in $\mathrm{L}^q(\mathbb{R}^n_+)$, there exists $w\in \mathrm{C}_c^\infty(\mathbb{R}^n_+)$ such that
\begin{align*}
    \lVert  v - w \rVert_{\dot{\mathrm{H}}^{-s,p}_0(\mathbb{R}^n_+)} \lesssim_{n,s,p,q} \lVert  v - w \rVert_{\mathrm{L}^{q}(\mathbb{R}^n_+)} \lesssim_{n,s,p,q}\varepsilon \text{,}
\end{align*}
so the triangle inequality gives
\begin{align*}
    \lVert  u - w \rVert_{\dot{\mathrm{H}}^{-s,p}_0(\mathbb{R}^n_+)} \lesssim_{n,s,p,q}\varepsilon\text{,}
\end{align*}
which concludes the proof since $w\in \mathrm{C}_c^\infty(\mathbb{R}^n_+)$.

\textbf{Step 3:} For the density in the intersection spaces, it suffices to reproduce the above \textbf{Step 1} by means of Corollary \ref{cor:ProjIntersecHomHsp0}.
\end{proof}

\begin{corollary}\label{cor:DensityHsp0=HspRn+} For all $p\in(1,+\infty)$, $s\in(-1+\frac{1}{p},\frac{1}{p})$,
\begin{align*}
    \dot{\mathrm{H}}^{s,p}_0(\mathbb{R}^n_+)=\dot{\mathrm{H}}^{s,p}(\mathbb{R}^n_+)\text{. }
\end{align*}
In particular, $\mathrm{C}_c^\infty(\mathbb{R}^n_+)$ is dense in $\dot{\mathrm{H}}^{s,p}(\mathbb{R}^n_+)$ for same range of indices.
\end{corollary}

\begin{proof} This is a direct consequence of the definition of restriction spaces and Proposition \ref{prop:SobolevMultiplier}, the density result follows from Proposition \ref{prop:densityCcinftyHsp0}.
\end{proof}

\begin{proposition}\label{prop:DualitySobolevDomain} Let $p\in(1,+\infty)$, $s\in(-\frac{n}{p'},\frac{n}{p})$, we have
\begin{align*}
    (\dot{\mathrm{H}}^{s,p}(\mathbb{R}^n_+))' = \dot{\mathrm{H}}^{-s,p'}_0(\mathbb{R}^n_+) \text{ and }
    (\dot{\mathrm{H}}^{s,p}_0(\mathbb{R}^n_+))' = \dot{\mathrm{H}}^{-s,p'}(\mathbb{R}^n_+)\text{. }
\end{align*}
\end{proposition}

\begin{proof}First, consider $s\in(-\frac{n}{p'},\frac{n}{p})$, let  $\Phi\in \dot{\mathrm{H}}^{-s,p'}_0(\mathbb{R}^n_+)\subset \dot{\mathrm{H}}^{-s,p'}(\mathbb{R}^n)$, then using definition of restriction spaces, the following map defines a linear functional on $\dot{\mathrm{H}}^{s,p}(\mathbb{R}^n_+)$,
\begin{align*}
    u\longmapsto \big\langle\Phi,\tilde{u}\big\rangle_{\mathbb{R}^n}\text{, }
\end{align*}
where $\tilde{u}$ is any extension of $u$, and notice that the action of $\Phi$ does not depend on the choice of such extension of $u$. Indeed, if $U\in \dot{\mathrm{H}}^{s,p}(\mathbb{R}^n)$ is another extension of $u$,  we obtain that $w:=U-\tilde{u}\in \dot{\mathrm{H}}^{s,p}_0(\overline{\mathbb{R}^n_+}^c)$. It follows from Proposition \ref{prop:densityCcinftyHsp0} that $w$ is a strong limit in $\dot{\mathrm{H}}^{s,p}_0(\overline{\mathbb{R}^n_+}^c)$ of a sequence of functions $(w_k)_{k\in\mathbb{N}}\subset \mathrm{C}_c^\infty(\overline{\mathbb{R}^n_+}^c)$ so that, passing to the limit, in the duality bracket, we obtain
\begin{align*}
    \big\langle\Phi,U\big\rangle_{\mathbb{R}^n}-\big\langle\Phi,\tilde{u}\big\rangle_{\mathbb{R}^n} = \big\langle\Phi,w\big\rangle_{\mathbb{R}^n}=0\text{.}
\end{align*}
This gives a well-defined continuous injective map
\begin{align}\label{eq:mapduality1Hsp}
\left\{\begin{array}{cl}
\dot{\mathrm{H}}^{-s,p'}_0(\mathbb{R}^n_+) &\longrightarrow (\dot{\mathrm{H}}^{s,p}(\mathbb{R}^n_+))'\\
\Phi &\longmapsto \,\,\, \big\langle\Phi,\tilde{\cdot}\big\rangle_{\mathbb{R}^n}
\end{array}\right.\text{. }
\end{align}
Now, let  $\Psi\in (\dot{\mathrm{H}}^{s,p}(\mathbb{R}^n_+))'$, for all $u\in \dot{\mathrm{H}}^{s,p}(\mathbb{R}^n_+)$, since $\mathbbm{1}_{\mathbb{R}^n_+} u=u$, we may write,
\begin{align*}
    \langle\Psi,u\rangle = \langle\Psi, \mathbbm{1}_{\mathbb{R}^n_+}\widetilde{ u}\rangle\text{, }
\end{align*}
for any extension $\tilde{u}\in \dot{\mathrm{H}}^{s,p}(\mathbb{R}^n)$ of $u$, hence, as a direct consequence of the definition of restriction space $\mathbbm{1}_{\mathbb{R}^n_+} \Psi \in (\dot{\mathrm{H}}^{s,p}(\mathbb{R}^n))'=\dot{\mathrm{H}}^{-s,p'}(\mathbb{R}^n)$, so $\mathbbm{1}_{\mathbb{R}^n_+} \Psi \in  \dot{\mathrm{H}}^{-s,p'}_0(\mathbb{R}^n_+)$. The following map is well-defined continuous and injective
\begin{align}\label{eq:mapduality2Hsp}
\left\{\begin{array}{cl}
 (\dot{\mathrm{H}}^{s,p}(\mathbb{R}^n_+))' &\longrightarrow \dot{\mathrm{H}}^{-s,p'}_0(\mathbb{R}^n_+)\\
\Psi &\longmapsto \,\,\, \mathbbm{1}_{\mathbb{R}^n_+} \Psi
\end{array}\right.\text{. }
\end{align}
Both maps \eqref{eq:mapduality1Hsp} and \eqref{eq:mapduality2Hsp} are even isometric, and we obtain
\begin{align*}
    (\dot{\mathrm{H}}^{s,p}(\mathbb{R}^n_+))' = \dot{\mathrm{H}}^{-s,p'}_0(\mathbb{R}^n_+)\text{, }
\end{align*}
which was the first statement. The second statement follows from duality and reflexivity exchanging roles of involved exponents.
\end{proof}

One would also like to carry over density results in intersection spaces in order to deduce its counterpart in their real interpolation spaces.

\begin{corollary}\label{cor:densityCcinftyIntersecHsp0}Let $p\in (1,+\infty)$, $-n/p'<s_0<s_1<n/p$, i.e. such that $(\mathcal{C}_{-s_0,p'})$ and  $(\mathcal{C}_{s_1,p})$ are both satisfied.

The space $\mathrm{C}_c^\infty(\mathbb{R}^n_+)$ is dense in $\dot{\mathrm{H}}^{s_0,p}_0(\mathbb{R}^n_+)\cap \dot{\mathrm{H}}^{s_1,p}_0(\mathbb{R}^n_+)$.
\end{corollary} 

\begin{proof}Let $p\in (1,+\infty)$, $-n/p'<s_0<s_1<n/p$. There are three subcases, $0\leqslant s_0<s_1$, $s_0<0<s_1$, and $s_0<s_1\leqslant 0$. 

The case $0\leqslant s_0<s_1$ follows the lines of Proposition \ref{prop:densityCcinftyHsp0} thanks to Corollary \ref{cor:ProjIntersecHomHsp0}.

The case $s_0<0<s_1$, can be done via duality argument as in Proposition \ref{prop:densityCcinftyHsp0} for the negative index of regularity. Let us  consider $\tfrac{1}{q}=\tfrac{1}{p}-\tfrac{s_0}{n}$, the following embeddings are true
\begin{align*}
    \mathrm{H}^{s_1-s_0,q}_0(\mathbb{R}^n_+) \hookrightarrow\mathrm{L}^q(\mathbb{R}^n_+)\cap \mathrm{H}^{s_1,p}_0(\mathbb{R}^n_+) \hookrightarrow \dot{\mathrm{H}}^{s_0,p}_0(\mathbb{R}^n_+)\cap \dot{\mathrm{H}}^{s_1,p}_0(\mathbb{R}^n_+) \hookrightarrow \dot{\mathrm{H}}^{s_1,p}_0(\mathbb{R}^n_+)\text{. }
\end{align*}
One may dualize it to deduce
\begin{align*}
     \dot{\mathrm{H}}^{-s_1,p'}(\mathbb{R}^n_+) \hookrightarrow (\dot{\mathrm{H}}^{s_0,p}_0(\mathbb{R}^n_+)\cap \dot{\mathrm{H}}^{s_1,p}_0(\mathbb{R}^n_+))' \hookrightarrow \mathrm{H}^{s_0-s_1,q'}(\mathbb{R}^n_+)\text{. }
\end{align*}
We deduce that the last embedding is dense, since $(\dot{\mathrm{H}}^{s_0,p}_0(\mathbb{R}^n_+)\cap \dot{\mathrm{H}}^{s_1,p}_0(\mathbb{R}^n_+))'$ contains $\dot{\mathrm{H}}^{-s_1,p'}(\mathbb{R}^n_+)$ via canonical embedding, so that by duality and reflexivity of all involved spaces, the following embedding is dense:
\begin{align*}
     \mathrm{H}^{s_1-s_0,q}_0(\mathbb{R}^n_+) \hookrightarrow \dot{\mathrm{H}}^{s_0,p}_0(\mathbb{R}^n_+)\cap \dot{ \mathrm{H}}^{s_1,p}_0(\mathbb{R}^n_+)\text{. }
\end{align*}
Since $\mathrm{C}_c^\infty(\mathbb{R}^n_+) \hookrightarrow \mathrm{H}^{s_1-s_0,q}_0(\mathbb{R}^n_+)$ is dense,  \cite[Remark~2.7]{JerisonKenig1995}, the result follows.

We end the proof claiming that the third case $s_0<s_1\leqslant 0$ can be done similarly via duality and reflexivity arguments.
\end{proof}

\subsubsection{Homogeneous Besov spaces, Interpolation}

We are done with properties of homogeneous Sobolev spaces. We continue with a real interpolation embedding lemma, that allows us to transfer all nice properties, like boundedness of extension and projection operators, from homogeneous Sobolev spaces to homogeneous Besov spaces.

\begin{lemma}\label{lem:EmbeddingInterpHomSobspacesRn+} Let $(p,q,q_0,q_1)\in(1,+\infty)\times[1,+\infty]^3$, $s_0,s_1\in\mathbb{R}$, such that $s_0< s_1$, and set
\begin{align*}
    s:= (1-\theta)s_0+ \theta s_1\text{. }
\end{align*}
If $(\mathcal{C}_{s_0,p})$ is satisfied we have,
\begin{align}
    &\dot{\mathrm{B}}^{s}_{p,q}(\mathbb{R}^n_+)\hookrightarrow(\dot{\mathrm{H}}^{s_0,p}(\mathbb{R}^n_+),\dot{\mathrm{H}}^{s_1,p}(\mathbb{R}^n_+))_{\theta,q}\text{, }\label{eq:HomInterpEmbedding1}\\
    &\dot{\mathrm{B}}^{s}_{p,q,0}(\mathbb{R}^n_+)\hookleftarrow(\dot{\mathrm{H}}^{s_0,p}_0(\mathbb{R}^n_+),\dot{\mathrm{H}}^{s_1,p}_0(\mathbb{R}^n_+))_{\theta,q} \text{.}\label{eq:HomInterpEmbedding3}
\end{align}
Similarly if $(\mathcal{C}_{s_0,p,q_0})$ is satisfied, we also have
\begin{align}
    &\dot{\mathrm{B}}^{s}_{p,q}(\mathbb{R}^n_+)\hookrightarrow(\dot{\mathrm{B}}^{s_0}_{p,q_0}(\mathbb{R}^n_+),\dot{\mathrm{B}}^{s_1}_{p,q_1}(\mathbb{R}^n_+))_{\theta,q}\text{, }\label{eq:HomInterpEmbedding2}\\
    &\dot{\mathrm{B}}^{s}_{p,q,0}(\mathbb{R}^n_+)\hookleftarrow(\dot{\mathrm{B}}^{s_0}_{p,q_0,0}(\mathbb{R}^n_+),\dot{\mathrm{B}}^{s_1}_{p,q_1,0}(\mathbb{R}^n_+))_{\theta,q}\label{eq:HomInterpEmbedding4}\text{. }
\end{align}
\end{lemma}

\begin{proof}For embeddings \eqref{eq:HomInterpEmbedding1} and \eqref{eq:HomInterpEmbedding2}, one may follow the first part of the proof of \cite[Proposition~3.22]{DanchinHieberMuchaTolk2020}.

The third embedding \eqref{eq:HomInterpEmbedding3}, (the fourth one \eqref{eq:HomInterpEmbedding4} can be treated similarly) is straightforward since,
\begin{align*}
    (\dot{\mathrm{H}}^{s_0,p}_0(\mathbb{R}^n_+),\dot{\mathrm{H}}^{s_1,p}_0(\mathbb{R}^n_+))_{\theta,q} \hookrightarrow(\dot{\mathrm{H}}^{s_0,p}(\mathbb{R}^n),\dot{\mathrm{H}}^{s_1,p}(\mathbb{R}^n))_{\theta,q} = \dot{\mathrm{B}}^{s}_{p,q}(\mathbb{R}^n)\text{.}
\end{align*}
By definition, $f\in (\dot{\mathrm{H}}^{s_0,p}_0(\mathbb{R}^n_+),\dot{\mathrm{H}}^{s_1,p}_0(\mathbb{R}^n_+))_{\theta,q}\subset \dot{\mathrm{H}}^{s_0,p}_0(\mathbb{R}^n_+)+\dot{\mathrm{H}}^{s_1,p}_0(\mathbb{R}^n_+)$, hence $\supp f \subset \overline{\mathbb{R}^n_+}$ and $f\in \dot{\mathrm{B}}^{s}_{p,q,0}(\mathbb{R}^n_+)$.
\end{proof}

As we mentioned, the above lemma can be used to prove the boundedness of some operators on a sufficiently wide range of indices on Besov spaces via  some sort of interpolation method, without the exact description of the interpolation space; see below.

\begin{corollary}\label{cor:ExtProj0HomBspq}Let $p\in(1,+\infty), q\in [1,+\infty]$, $s> -1+\tfrac{1}{p}$, $m\in\mathbb{N}$, such that  $s < m+1+\tfrac{1}{p}$. Let us  consider the extension operator $\mathrm{E}$ (resp. $\mathcal{P}_0$) given by Proposition \ref{prop:ExtOpHomSobSpaces} (resp.  Lemma \ref{lemma:ProjHsp0}).

If either one of the following conditions holds
\begin{itemize}
    \item $s > 0$ and $u\in {\mathrm{B}}^{s}_{p,q}(\mathbb{R}^n_+)$ (resp. $u \in {\mathrm{B}}^{s}_{p,q}(\mathbb{R}^n)$) ;
    \item $s\in (-1+\tfrac{1}{p},\tfrac{1}{p})$ and $u\in \dot{\mathrm{B}}^{s}_{p,q}(\mathbb{R}^n_+)$ (resp. $u \in \dot{\mathrm{B}}^{s}_{p,q}(\mathbb{R}^n)$) ;
\end{itemize}
then we have the estimate
\begin{align*}
    \lVert \mathrm{E}u \rVert_{\dot{\mathrm{B}}^{s}_{p,q}(\mathbb{R}^n)} \lesssim_{s,m,p,n}   \lVert  u \rVert_{\dot{\mathrm{B}}^{s}_{p,q}(\mathbb{R}^n_+)}\textrm{. (resp. } \lVert \mathcal{P}_0 u \rVert_{\dot{\mathrm{B}}^{s}_{p,q}(\mathbb{R}^n)} \lesssim_{s,m,p,n}   \lVert  u \rVert_{\dot{\mathrm{B}}^{s}_{p,q}(\mathbb{R}^n)}\textrm{.)}
\end{align*}
In particular, $\mathrm{E}$ (resp. $\mathcal{P}_0$) is a bounded operator from $\dot{\mathrm{B}}^{s}_{p,q}(\mathbb{R}^n_+)$ to $\dot{\mathrm{B}}^{s}_{p,q}(\mathbb{R}^n)$ (resp. from $\dot{\mathrm{B}}^{s}_{p,q}(\mathbb{R}^n)$ to $\dot{\mathrm{B}}^{s}_{p,q,0}(\mathbb{R}^n_+)$) whenever \eqref{AssumptionCompletenessExponents} is satisfied.
\end{corollary}

\begin{proof}Let $p\in(1,+\infty), q\in [1,+\infty)$, $s> -1+\tfrac{1}{p}$, $m\in\mathbb{N}$, such that  $s < m+1+\tfrac{1}{p}$. Without loss of generality, it suffices to prove the result for the operator $\mathrm{E}$, since we have the identity $\mathcal{P}_0 = \mathrm{I} - \mathrm{E}^{-}[\mathbbm{1}_{\mathbb{R}^n_-}]$, as written in the proof of Lemma \ref{lemma:ProjHsp0}.

The boundedness of $\mathrm{E}$ on $\dot{\mathrm{B}}^{s}_{p,q}(\mathbb{R}^n)$ for $(p,q)\in(1,+\infty)\times[1,+\infty]$, $s\in(-1+\tfrac{1}{p},\tfrac{1}{p})$ is again a direct consequence of Proposition \ref{prop:SobolevMultiplier}.

It remains to prove boundedness for $s\geqslant \tfrac{1}{p}$. To do so, we proceed via a manual real interpolation scheme.

Let $u\in {\mathrm{B}}^{s}_{p,q}(\mathbb{R}^n_+)$, $\theta\in(0,1)$ such that $\theta s_1 =s$, where $s_1\in(s,m+1+\tfrac{1}{p})$.
One has, 
\begin{align*}
     u\in (\mathrm{L}^p(\mathbb{R}^n_+), \mathrm{H}^{s_1,p}(\mathbb{R}^n_+))_{\theta,q} \hookrightarrow (\mathrm{L}^p(\mathbb{R}^n_+), \dot{\mathrm{H}}^{s_1,p}(\mathbb{R}^n_+))_{\theta,q}\subset \mathrm{L}^p(\mathbb{R}^n_+)+ \dot{\mathrm{H}}^{s_1,p}(\mathbb{R}^n_+)\text{. }
\end{align*}
Hence, for $a\in\mathrm{L}^p(\mathbb{R}^n_+)$, $b\in \dot{\mathrm{H}}^{s_1,p}(\mathbb{R}^n_+)$ such that $f=a+b$, we can deduce that
\begin{align*}
    b=u-a\in {\mathrm{B}}^{s}_{p,q}(\mathbb{R}^n_+)+ \mathrm{L}^p(\mathbb{R}^n_+)\subset \mathrm{L}^p(\mathbb{R}^n_+)\text{,} 
\end{align*}
so that $b\in \mathrm{L}^p(\mathbb{R}^n_+)\cap \dot{\mathrm{H}}^{s_1,p}(\mathbb{R}^n_+) = \mathrm{H}^{s_1,p}(\mathbb{R}^n_+)$ thanks to Proposition \ref{prop:IntersecHomHspRn+}. Hence, $\mathrm{E}u=\mathrm{E}a+\mathrm{E}b$, with $\mathrm{E}a\in\mathrm{L}^p(\mathbb{R}^n_+)$, $\mathrm{E}b\in \mathrm{H}^{s_1,p}(\mathbb{R}^n_+)$, with the homogeneous estimates provided by Proposition~\ref{prop:ExtOpHomSobSpaces}. Then $\mathrm{E}u_{|_{\mathbb{R}^n_+}}=u$, and we have the estimates
\begin{align*}
    K(t,\mathrm{E}u,\mathrm{L}^p(\mathbb{R}^n),\dot{\mathrm{H}}^{s_1,p}(\mathbb{R}^n)) \leqslant \lVert  \mathrm{E}a \rVert_{\mathrm{L}^p(\mathbb{R}^n)} + t \lVert  \mathrm{E}b \rVert_{\dot{\mathrm{H}}^{s_1,p}(\mathbb{R}^n)} \lesssim_{p,m,n} \lVert a \rVert_{\mathrm{L}^p(\mathbb{R}^n_+)} + t \lVert  b \rVert_{\dot{\mathrm{H}}^{s_1,p}(\mathbb{R}^n_+)}\text{. }
\end{align*}
Hence, taking infimum on all such functions $a$ and $b$, and multiplying by $t^{-\theta}$ leads to
\begin{align*}
    t^{-\theta} K(t,\mathrm{E}u,\mathrm{L}^p(\mathbb{R}^n),\dot{\mathrm{H}}^{s_1,p}(\mathbb{R}^n)) \lesssim_{p,s,s_1,n} t^{-\theta} K(t,u,\mathrm{L}^p(\mathbb{R}^n_+),\dot{\mathrm{H}}^{s_1,p}(\mathbb{R}^n_+))\text{, }
\end{align*}
so one may take the $\mathrm{L}^q_{\ast}$-norm of the above inequality and use \eqref{eq:HomInterpEmbedding1} from Lemma \ref{lem:EmbeddingInterpHomSobspacesRn+} to deduce that
\begin{align*}
    \lVert  \mathrm{E}u \rVert_{\dot{\mathrm{B}}^{s}_{p,q}(\mathbb{R}^n)} \lesssim_{p,s,q,n}\lVert  u \rVert_{\dot{\mathrm{B}}^{s}_{p,q}(\mathbb{R}^n_+)}\text{. }
\end{align*}
If $q<+\infty$, then ${\mathrm{B}}^{s}_{p,q}(\mathbb{R}^n_+)$ is dense in $\dot{\mathrm{B}}^{s}_{p,q}(\mathbb{R}^n_+)$, so that the conclusion holds by density whenever \eqref{AssumptionCompletenessExponents} is satisfied.

If $q=+\infty$, and \eqref{AssumptionCompletenessExponents} is satisfied, necessarily $s<\tfrac{n}{p}$. We introduce $\mathcal{E}:= \mathrm{E}[ \mathbbm{1}_{\mathbb{R}^n_+} \cdot]$ which is bounded, thanks to the above step, seen as an operator
\begin{align*}
    \mathcal{E}\,:\, \dot{\mathrm{B}}^{s_j}_{p,q_j}(\mathbb{R}^n) \longrightarrow \dot{\mathrm{B}}^{s_j}_{p,q_j}(\mathbb{R}^n) \text{, }
\end{align*}
provided $s_0<s<s_1 <\frac{n}{p}$, and $q_j\in[1,\infty)$, $j\in\{ 0,1\}$. Thus, by real interpolation argument, thanks to Theorem \ref{thm:InterpHomSpacesRn}, for all $U\in  \dot{\mathrm{B}}^{s}_{p,\infty}(\mathbb{R}^n)$, we have
\begin{align*}
    \lVert  \mathcal{E}U \rVert_{\dot{\mathrm{B}}^{s}_{p,\infty}(\mathbb{R}^n)} \lesssim_{p,s,q,n}\lVert  U \rVert_{\dot{\mathrm{B}}^{s}_{p,\infty}(\mathbb{R}^n)}\text{. }
\end{align*}
In particular, for all $u\in \dot{\mathrm{B}}^{s}_{p,\infty}(\mathbb{R}^n_+)$, and all $U\in  \dot{\mathrm{B}}^{s}_{p,\infty}(\mathbb{R}^n)$ such that $U_{|_{\mathbb{R}^n_+}}=u$, we have
\begin{align*}
    \lVert  \mathrm{E}u \rVert_{\dot{\mathrm{B}}^{s}_{p,\infty}(\mathbb{R}^n)} \lesssim_{p,s,q,n}\lVert  U \rVert_{\dot{\mathrm{B}}^{s}_{p,\infty}(\mathbb{R}^n)}\text{.}
\end{align*}
Therefore, taking the infimum on all such functions $U$ gives the result when $q=+\infty$ and \eqref{AssumptionCompletenessExponents} is satisfied.
\end{proof}

\begin{proposition}\label{prop:SobolevEmbeddingsBesovRn+} Let $p,q\in[1,+\infty]$, $s\in(0,n)$, such that
\begin{align*}
    \frac{1}{q}=\frac{1}{p}-\frac{s}{n}\text{.}
\end{align*}
We have the following estimates, 
\begin{align*}
    \lVert  u \rVert_{\mathrm{L}^q(\mathbb{R}^n_+)} \lesssim_{n,s,p,q,r}\lVert  u \rVert_{\dot{\mathrm{B}}^{s}_{p,r}(\mathbb{R}^n_+)}\text{, } &\forall u\in \dot{\mathrm{B}}^{s}_{p,r}(\mathbb{R}^n_+)\text{, } r\in[1,q]\\
    \lVert  u \rVert_{\dot{\mathrm{B}}^{-s}_{q,r,0}(\mathbb{R}^n_+)} \lesssim_{n,s,p,q,r} \lVert  u \rVert_{\mathrm{L}^p(\mathbb{R}^n_+)}\text{, } &\forall u\in \mathrm{L}^p(\mathbb{R}^n_+)\text{, } r\in[q,+\infty] \text{. }
\end{align*}
Moreover, we also have $\dot{\mathrm{B}}^{\frac{n}{p}}_{p,1}(\mathbb{R}^n_+)\hookrightarrow \mathrm{C}^0_0(\overline{\mathbb{R}^n_+})$, whenever $p$ is finite.
\end{proposition}

\begin{proof} Each embedding is a direct consequence of the definition of each space and the corresponding ones on $\mathbb{R}^n$, see point \textit{(iv)} of Proposition \ref{prop:PropertiesHomBesovSpacesRn}.
\end{proof}

\begin{lemma}\label{lem:densityCcinftyBesov0s+} Let $p\in(1,+\infty)$, $q\in[1,+\infty)$ and $s>0$. The function space $\mathrm{C}_c^\infty(\mathbb{R}^n_+)$ is dense in $\dot{\mathrm{B}}^{s}_{p,q,0}(\mathbb{R}^n_+)$ whenever \eqref{AssumptionCompletenessExponents} is satisfied.
\end{lemma}

\begin{proof}As in the proof of Proposition \ref{prop:densityCcinftyHsp0}, in the case of non-negative index: by a successive approximations scheme, we use density of $\mathrm{B}^{s}_{p,q}(\mathbb{R}^n)$ in $\dot{\mathrm{B}}^{s}_{p,q}(\mathbb{R}^n)$, to approximate functions in $\dot{\mathrm{B}}^{s}_{p,q,0}(\mathbb{R}^n_+)$. Then the boundedness of $\mathcal{P}_0$ on $\dot{\mathrm{B}}^{s}_{p,q}(\mathbb{R}^n_+)$, and the density of $\mathrm{C}_c^\infty(\mathbb{R}^n_+)$ in ${\mathrm{B}}^{s}_{p,q,0}(\mathbb{R}^n_+)$ yields the result.
\end{proof}

\begin{proposition}\label{prop:InterpHomSpacesRn+}Let $(p_0,p_1,p,q)\in(1,+\infty)^3\times[1,+\infty]$, $s_0,s_1\in\mathbb{R}$, such that $s_0< s_1$, let  $(\mathfrak{h},\mathfrak{b})\in\{ (\mathrm{H},\mathrm{B}),(\mathrm{H}_0,\mathrm{B}_{\cdot,\cdot,0})\}$, and set
\begin{align*}
    \left(s,\frac{1}{p_\theta}\right):= (1-\theta)\left(s_0,\frac{1}{p_0}\right)+ \theta\left(s_1,\frac{1}{p_1}\right)\text{. }
\end{align*}

If either one of following assertions is satisfied,
\begin{enumerate}[label=(\roman*)]
    \item $q\in[1,+\infty)$, $s_j>-1+\tfrac{1}{p_j}$, $j\in\{0,1\}$; \label{assertion(i)InterpSob}
    \item $q\in[1,+\infty]$, $s_j>-1+\tfrac{1}{p_j}$, and $(\mathcal{C}_{s_j,p_j})$ is satisfied, $j\in\{0,1\}$; \label{assertion(ii)InterpSob}
\end{enumerate}
If $p_0=p_1=p$ and \eqref{AssumptionCompletenessExponents} is satisfied, the following equality is true with equivalence of norms
\begin{align}
    (\dot{\mathfrak{h}}^{s_0,p}(\mathbb{R}^n_+),\dot{\mathfrak{h}}^{s_1,p}(\mathbb{R}^n_+))_{\theta,q}&=\dot{\mathfrak{b}}^{s}_{p,q}(\mathbb{R}^n_+)\text{.}\label{eqRealInterpSobolevintoBesovRn+}
\end{align}
If $(\mathcal{C}_{s_0,p_0})$ and $(\mathcal{C}_{s_1,p_1})$ are true then also is $(\mathcal{C}_{s,p_\theta})$ and
\begin{align}
[\dot{\mathfrak{h}}^{s_0,p_0}(\mathbb{R}^n_+),\dot{\mathfrak{h}}^{s_1,p_1}(\mathbb{R}^n_+)]_{\theta} = \dot{\mathfrak{h}}^{s,p_\theta}(\mathbb{R}^n_+) \text{.}\label{eqCompInterpSobolevIntoSobolevRn+}
\end{align}
\end{proposition}

\begin{proof}We start noticing that \eqref{eqCompInterpSobolevIntoSobolevRn+} only makes sense under assertion \textit{\ref{assertion(ii)InterpSob}}. 

\textbf{Step 1:} We prove first \eqref{eqCompInterpSobolevIntoSobolevRn+} and \eqref{eqRealInterpSobolevintoBesovRn+} under assertion \textit{\ref{assertion(ii)InterpSob}}.

It suffices to assert that $\{ \dot{\mathfrak{h}}^{s_0,p_0}(\mathbb{R}^n_+),\dot{\mathfrak{h}}^{s_1,p_1}(\mathbb{R}^n_+)\}$ is a retraction of $\{ \dot{\mathrm{H}}^{s_0,p_0}(\mathbb{R}^n),\dot{\mathrm{H}}^{s_1,p_1}(\mathbb{R}^n)\}$, thanks to \cite[Theorem~6.4.2]{BerghLofstrom1976}. Indeed, both retractions are given by
\begin{align*}
    \mathrm{E} \,:\, \dot{\mathrm{H}}^{s_j,p_j}(\mathbb{R}^n_+)\longrightarrow \dot{\mathrm{H}}^{s_j,p_j}(\mathbb{R}^n) &\text{ and } \mathrm{R}_{{\mathbb{R}^n_+}} \,:\, \dot{\mathrm{H}}^{s_j,p_j}(\mathbb{R}^n)\longrightarrow \dot{\mathrm{H}}^{s_j,p_j}(\mathbb{R}^n_+) \text{, }\\
    \iota \,:\, \dot{\mathrm{H}}^{s_j,p_j}_0(\mathbb{R}^n_+)\longrightarrow \dot{\mathrm{H}}^{s_j,p_j}(\mathbb{R}^n) &\text{ and }\,\,\,\, \mathcal{P}_{0} \,:\, \dot{\mathrm{H}}^{s_j,p_j}(\mathbb{R}^n)\longrightarrow \dot{\mathrm{H}}^{s_j,p_j}_0(\mathbb{R}^n_+) \text{. }
\end{align*}
Here, $\mathrm{R}_{{\mathbb{R}^n_+}}$ and $\iota$ stand respectively for the restriction and the canonical injection operator. Boundedness and range of $\mathrm{E}$ and $\mathcal{P}_0$ provided by Lemma \ref{lemma:ProjHsp0} and Corollary \ref{cor:ExtProj0HomBspq} lead to \eqref{eqCompInterpSobolevIntoSobolevRn+} and \eqref{eqRealInterpSobolevintoBesovRn+} under assertion \textit{\ref{assertion(ii)InterpSob}}.

\textbf{Step 2:} We prove \eqref{eqRealInterpSobolevintoBesovRn+} under assertion \textit{\ref{assertion(i)InterpSob}}.

\textbf{Step 2.1:} $(\mathfrak{h},\mathfrak{b})=(\mathrm{H},\mathrm{B})$.\\
Thanks to Lemma \ref{lem:EmbeddingInterpHomSobspacesRn+}, we have continuous embedding,
\begin{align}\label{eq:interpEmbeddingproofStep21}
    \dot{\mathrm{B}}^{s}_{p,q}(\mathbb{R}^n_+)\hookrightarrow(\dot{\mathrm{H}}^{s_0,p}(\mathbb{R}^n_+),\dot{\mathrm{H}}^{s_1,p}(\mathbb{R}^n_+))_{\theta,q}\text{. }
\end{align}
Let us  prove the reverse embedding,
\begin{align*}
    \dot{\mathrm{B}}^{s}_{p,q}(\mathbb{R}^n_+)\hookleftarrow(\dot{\mathrm{H}}^{s_0,p}(\mathbb{R}^n_+),\dot{\mathrm{H}}^{s_1,p}(\mathbb{R}^n_+))_{\theta,q}\text{. }
\end{align*}
Without loss of generality, we can assume $s_1\geqslant \tfrac{n}{p}$. Let $f\in \eus{S}_0(\overline{\mathbb{R}^n_+})\subset \dot{\mathrm{B}}^{s}_{p,q}(\mathbb{R}^n_+)$, it follows that $f\in(\dot{\mathrm{H}}^{s_0,p}(\mathbb{R}^n_+),\dot{\mathrm{H}}^{s_1,p}(\mathbb{R}^n_+))_{\theta,q}\subset \dot{\mathrm{H}}^{s_0,p}(\mathbb{R}^n_+)+\dot{\mathrm{H}}^{s_1,p}(\mathbb{R}^n_+)$. Thus, for all $(a,b)\in \dot{\mathrm{H}}^{s_0,p}(\mathbb{R}^n_+)\times\dot{\mathrm{H}}^{s_1,p}(\mathbb{R}^n_+)$ such that $f=a+b$, we have,
\begin{align*}
    b=f-a \in (\eus{S}_0(\overline{\mathbb{R}^n_+})+ \dot{\mathrm{H}}^{s_0,p}(\mathbb{R}^n_+)) \cap \dot{\mathrm{H}}^{s_1,p}(\mathbb{R}^n_+)\text{. }
\end{align*}
In particular, we have $a \in\dot{\mathrm{H}}^{s_0,p}(\mathbb{R}^n_+)$ and $b\in\dot{\mathrm{H}}^{s_0,p}(\mathbb{R}^n_+) \cap \dot{\mathrm{H}}^{s_1,p}(\mathbb{R}^n_+)$. Hence, we can introduce $F:= \mathrm{E}a + \mathrm{E}b$, where $F_{|_{\mathbb{R}^n_+}} = f $, $\mathrm{E}a \in\dot{\mathrm{H}}^{s_0,p}(\mathbb{R}^n)$ and $\mathrm{E} b\in\dot{\mathrm{H}}^{s_0,p}(\mathbb{R}^n) \cap \dot{\mathrm{H}}^{s_1,p}(\mathbb{R}^n)$, with the estimates, given by Corollary \ref{cor:ExtOpIntersecHomHspRn+},
\begin{align*}
    \lVert \mathrm{E} a \rVert_{\dot{\mathrm{H}}^{s_0,p}(\mathbb{R}^n)} \lesssim_{s_0,m,p,n}   \lVert  a \rVert_{\dot{\mathrm{H}}^{s_0,p}(\mathbb{R}^n_+)}\text{ and }\lVert \mathrm{E} b \rVert_{\dot{\mathrm{H}}^{s_1,p}(\mathbb{R}^n)} \lesssim_{s_1,m,p,n}   \lVert  b \rVert_{\dot{\mathrm{H}}^{s_1,p}(\mathbb{R}^n_+)}\text{. }
\end{align*}
Then, one may bound the $K$-functional of $F$, for $t>0$,
\begin{align*}
    K(t,F,\dot{\mathrm{H}}^{s_0,p}(\mathbb{R}^n),\dot{\mathrm{H}}^{s_1,p}(\mathbb{R}^n)) \leqslant \lVert \mathrm{E} a \rVert_{\dot{\mathrm{H}}^{s_0,p}(\mathbb{R}^n)}+ t\lVert \mathrm{E} b \rVert_{\dot{\mathrm{H}}^{s_1,p}(\mathbb{R}^n)} \lesssim_{s_j,p,n} \lVert a \rVert_{\dot{\mathrm{H}}^{s_0,p}(\mathbb{R}^n_+)}+ t\lVert  b \rVert_{\dot{\mathrm{H}}^{s_1,p}(\mathbb{R}^n_+)}
\end{align*}
Taking the infimum over all such functions $a$ and $b$, we obtain
\begin{align*}
    K(t,F,\dot{\mathrm{H}}^{s_0,p}(\mathbb{R}^n),\dot{\mathrm{H}}^{s_1,p}(\mathbb{R}^n))  \lesssim_{s_j,p,n} K(t,f,\dot{\mathrm{H}}^{s_0,p}(\mathbb{R}^n_+),\dot{\mathrm{H}}^{s_1,p}(\mathbb{R}^n_+))\text{, }
\end{align*}
from which we obtain, after multiplying by $t^{-\theta}$, taking the $\mathrm{L}^q_\ast$-norm with respect to $t$, and applying Theorem \ref{thm:InterpHomSpacesRn},
\begin{align*}
    \lVert f \rVert_{\dot{\mathrm{B}}^{s}_{p,q}(\mathbb{R}^n_+)}\leqslant\lVert F \rVert_{\dot{\mathrm{B}}^{s}_{p,q}(\mathbb{R}^n)}  \lesssim_{s,p,n} \lVert f \rVert_{(\dot{\mathrm{H}}^{s_0,p}(\mathbb{R}^n_+),\dot{\mathrm{H}}^{s_1,p}(\mathbb{R}^n_+))_{\theta,q}}\text{. }
\end{align*}
Finally, thanks to the first embedding \eqref{eq:interpEmbeddingproofStep21}, we have
\begin{align*}
    \lVert f \rVert_{\dot{\mathrm{B}}^{s}_{p,q}(\mathbb{R}^n_+)}\sim_{p,s,n} \lVert f \rVert_{(\dot{\mathrm{H}}^{s_0,p}(\mathbb{R}^n_+),\dot{\mathrm{H}}^{s_1,p}(\mathbb{R}^n_+))_{\theta,q}}\text{, } \forall f\in\eus{S}_0(\overline{\mathbb{R}^n_+})\text{. }
\end{align*}
Since $q<+\infty$, we can conclude by density of $\eus{S}_0(\overline{\mathbb{R}^n_+})$ in both $\dot{\mathrm{B}}^{s}_{p,q}(\mathbb{R}^n_+)$ and in the interpolation space $(\dot{\mathrm{H}}^{s_0,p}(\mathbb{R}^n_+),\dot{\mathrm{H}}^{s_1,p}(\mathbb{R}^n_+))_{\theta,q}$. The density argument for the later one is carried over by Lemma \ref{lem:IntersecHomHsp} and \cite[Theorem~3.4.2]{BerghLofstrom1976}\text{. }

\textbf{Step 2.2:} $\mathrm{C}_c^\infty(\mathbb{R}^n_+)$ is dense in $\dot{\mathrm{B}}^{s}_{p,q,0}$, provided $-1+\tfrac{1}{p}<s<\tfrac{1}{p}$, $p\in(1,+\infty)$, $q\in[1,+\infty)$.

Thanks to \textbf{Step 1} one may find, $-1+\tfrac{1}{p}<s_0<s<s_1<\tfrac{1}{p}$, $\theta\in(0,1)$, such that, as a consequence of \cite[Theorem~3.4.2]{BerghLofstrom1976}, we have the following dense embedding,
\begin{align*}
    \dot{\mathrm{H}}^{s_0,p}(\mathbb{R}^n_+)\cap\dot{\mathrm{H}}^{s_1,p}(\mathbb{R}^n_+) \hookrightarrow (\dot{\mathrm{H}}^{s_0,p}(\mathbb{R}^n_+),\dot{\mathrm{H}}^{s_1,p}(\mathbb{R}^n_+))_{\theta,q} = \dot{\mathrm{B}}^{s}_{p,q}(\mathbb{R}^n_+) = \dot{\mathrm{B}}^{s}_{p,q,0}(\mathbb{R}^n_+)\text{. }
\end{align*}
The equality in the above line is a direct consequence of Proposition \ref{prop:SobolevMultiplier}. In this case, the density of $\mathrm{C}_c^\infty(\mathbb{R}^n_+)$ is a straightforward application of Corollary \ref{cor:densityCcinftyIntersecHsp0} by successive approximations.

\textbf{Step 2.3:} $(\mathfrak{h},\mathfrak{b})=(\mathrm{H}_0,\mathrm{B}_{\cdot,\cdot,0})$.\\
Thanks to Lemma \ref{lem:EmbeddingInterpHomSobspacesRn+}, we have continuous embedding,
\begin{align*}
    (\dot{\mathrm{H}}^{s_0,p}_0(\mathbb{R}^n_+),\dot{\mathrm{H}}^{s_1,p}_0(\mathbb{R}^n_+))_{\theta,q} \hookrightarrow \dot{\mathrm{B}}^{s}_{p,q,0}(\mathbb{R}^n_+)\text{. }
\end{align*}
We are going to prove the reverse embedding, 
\begin{align*}
    (\dot{\mathrm{H}}^{s_0,p}_0(\mathbb{R}^n_+),\dot{\mathrm{H}}^{s_1,p}_0(\mathbb{R}^n_+))_{\theta,q} \hookleftarrow \dot{\mathrm{B}}^{s}_{p,q,0}(\mathbb{R}^n_+)\text{. }
\end{align*}
Again, without loss of generality we can assume $s_1\geqslant \tfrac{n}{p}$, otherwise one can go back to \textbf{Step~1}. Let us  consider $u\in \mathrm{C}_c^\infty(\mathbb{R}^n_+)$, then, $u$ belongs to $\dot{\mathrm{H}}^{s_0,p}(\mathbb{R}^n)+\dot{\mathrm{H}}^{s_1,p}(\mathbb{R}^n)$. In particular for $(a,b)\in \dot{\mathrm{H}}^{s_0,p}(\mathbb{R}^n)\times\dot{\mathrm{H}}^{s_1,p}(\mathbb{R}^n)$, such that $u=a+b$ we have
\begin{align*}
    b=u-a \in (\mathrm{C}_c^\infty(\mathbb{R}^n_+)+\dot{\mathrm{H}}^{s_0,p}(\mathbb{R}^n))\cap\dot{\mathrm{H}}^{s_1,p}(\mathbb{R}^n)\text{. }
\end{align*}
in particular we have $a \in\dot{\mathrm{H}}^{s_0,p}(\mathbb{R}^n)$ and $b\in\dot{\mathrm{H}}^{s_0,p}(\mathbb{R}^n) \cap \dot{\mathrm{H}}^{s_1,p}(\mathbb{R}^n)$. Consequently, we have $u = \mathcal{P}_0 u = \mathcal{P}_0 a + \mathcal{P}_0b$, with $\mathcal{P}_0a \in\dot{\mathrm{H}}^{s_0,p}_0(\mathbb{R}^n_+)$ and $\mathcal{P}_0 b\in\dot{\mathrm{H}}^{s_0,p}_0(\mathbb{R}^n_+) \cap \dot{\mathrm{H}}^{s_1,p}_0(\mathbb{R}^n_+)$, with the estimates
\begin{align*}
    \lVert \mathcal{P}_0 a \rVert_{\dot{\mathrm{H}}^{s_0,p}_0(\mathbb{R}^n_+)} \lesssim_{s_0,m,p,n}   \lVert  a \rVert_{\dot{\mathrm{H}}^{s_0,p}(\mathbb{R}^n)}\text{ and }\lVert \mathcal{P}_0 b \rVert_{\dot{\mathrm{H}}^{s_1,p}_0(\mathbb{R}^n_+)} \lesssim_{s_1,m,p,n}   \lVert  b \rVert_{\dot{\mathrm{H}}^{s_1,p}(\mathbb{R}^n)}\text{, }
\end{align*}
thanks to Corollary \ref{cor:ProjIntersecHomHsp0}. Thus, one may follow the lines of \textbf{Step 2.1}, to obtain for all $u\in \mathrm{C}_c^\infty(\mathbb{R}^n_+)$,
\begin{align*}
    \lVert u \rVert_{\dot{\mathrm{B}}^{s}_{p,q,0}(\mathbb{R}^n_+)} \sim_{s,p,n} \lVert u \rVert_{(\dot{\mathrm{H}}^{s_0,p}_0(\mathbb{R}^n_+),\dot{\mathrm{H}}^{s_1,p}_0(\mathbb{R}^n_+))_{\theta,q}} \text{. }
\end{align*}
Again, one can conclude via density arguments since $q<+\infty$, and $\mathrm{C}_c^\infty(\mathbb{R}^n_+)$ is dense in $\dot{\mathrm{B}}^{s}_{p,q,0}(\mathbb{R}^n_+)$ thanks to \textbf{Step 2.2} and Lemma \ref{lem:densityCcinftyBesov0s+}.
\end{proof}

The \textbf{Step 2.2} in the proof above can be turned more formally into the following corollary.
\begin{corollary}\label{cor:Bspq=Bspq0Rn+} Let $p\in(1,+\infty)$, $q\in[1,+\infty]$, $s\in(-1+\tfrac{1}{p},\tfrac{1}{p})$. Then the following equality holds with equivalence of norms,
\begin{align*}
    \dot{\mathrm{B}}^{s}_{p,q}(\mathbb{R}^n_+) = \dot{\mathrm{B}}^{s}_{p,q,0}(\mathbb{R}^n_+)\text{. }
\end{align*}
Moreover, the space $\mathrm{C}_c^\infty(\mathbb{R}^n_+)$ is dense whenever $q<+\infty$.
\end{corollary}

From general interpolation theory, we are able to deduce the following,

\begin{corollary}\label{cor:weakstardensity}Let $p\in(1,+\infty)$, $s>-1+1/p$, such that $(\mathcal{C}_{s,p,\infty})$ is satisfied.
\begin{itemize}
    \item The space $\mathrm{C}^\infty_c(\mathbb{R}^n_+)$ is weak${}^\ast$ dense in $\dot{\mathrm{B}}^{s}_{p,\infty,0}(\mathbb{R}^n_+)$.
    \item The space $\eus{S}_0(\overline{\mathbb{R}^n_+})$ is weak${}^\ast$ dense in $\dot{\mathrm{B}}^{s}_{p,\infty}(\mathbb{R}^n_+)$.
\end{itemize}
\end{corollary}

\begin{proof}The \cite[Theorem~3.7.1]{BerghLofstrom1976} with the remark at the end of its proof in combination with Lemma \ref{cor:densityCcinftyIntersecHsp0}, with the use of \cite[Theorem~3.4.2]{BerghLofstrom1976}, and Proposition \ref{prop:InterpHomSpacesRn+} imply that, for some $-1+1/p<s_0<s<s_1$, with $\theta\in (0,1)$, such that $s=(1-\theta)s_0 + \theta s_1$, we have the following strongly dense embedding,
\begin{align*}
     \mathrm{C}_c^\infty(\mathbb{R}^n_+)\hookrightarrow \dot{\mathrm{H}}^{s_0,p}_0(\mathbb{R}^n_+)\cap\dot{\mathrm{H}}^{s_1,p}_0(\mathbb{R}^n_+) \hookrightarrow (\dot{\mathrm{H}}^{s_0,p}_0(\mathbb{R}^n_+),\dot{\mathrm{H}}^{s_1,p}_0(\mathbb{R}^n_+))_{\theta}\text{, }
\end{align*}
and the following weak${}^\ast$ dense embedding
\begin{align*}
    (\dot{\mathrm{H}}^{s_0,p}_0(\mathbb{R}^n_+),\dot{\mathrm{H}}^{s_1,p}_0(\mathbb{R}^n_+))_{\theta}\hookrightarrow (\dot{\mathrm{H}}^{s_0,p}_0(\mathbb{R}^n_+),\dot{\mathrm{H}}^{s_1,p}_0(\mathbb{R}^n_+))_{\theta}'' = (\dot{\mathrm{H}}^{s_0,p}_0(\mathbb{R}^n_+),\dot{\mathrm{H}}^{s_1,p}_0(\mathbb{R}^n_+))_{\theta,\infty} = \dot{\mathrm{B}}^{s}_{p,\infty,0}(\mathbb{R}^n_+)\text{, }
\end{align*}
so that the result follows. We mention that $(\cdot,\cdot)_{\theta}$ is the real interpolation functor asking the $K$-functional to decay at infinity and near the origin, see for instance \cite[Definition~1.2]{bookLunardiInterpTheory}.

The same argument applies for the weak${}^\ast$ density of $\eus{S}_0(\overline{\mathbb{R}^n_+})$ in $\dot{\mathrm{B}}^{s}_{p,\infty}(\mathbb{R}^n_+)$.
\end{proof}

We state below the Besov analogue of Corollary \ref{cor:ProjIntersecHomHsp0}, Lemma \ref{lem:IntersecHomHsp} and Proposition \ref{prop:IntersecHomHspRn+}, for which the proofs are similar and left to the reader. 

\begin{proposition}\label{prop:IntersLpHomBesovRn+=BesovRn+}Let $p_j\in(1,+\infty)$, $q_j\in[1,+\infty]$, $s_j> -1+\tfrac{1}{p_j}$, $j\in\{0,1\}$, $m\in\mathbb{N}$, such that $(\mathcal{C}_{s_0,p_0,q_0})$ is satisfied and $s_j < m+1+\tfrac{1}{p_j}$, and consider the extension operator $\mathrm{E}$ given by Proposition \ref{prop:ExtOpHomSobSpaces}.

Then for all $u\in \dot{\mathrm{B}}^{s_0}_{p_0,q_0}(\mathbb{R}^n_+)\cap \dot{\mathrm{B}}^{s_1}_{p_1,q_1}(\mathbb{R}^n_+)$, we have $\mathrm{E}u\in \dot{\mathrm{B}}^{s_j}_{p_j,q_j}(\mathbb{R}^n)$, $j\in\{0,1\}$, with the estimate
\begin{align*}
    \lVert \mathrm{E} u \rVert_{\dot{\mathrm{B}}^{s_j}_{p_j,q_j}(\mathbb{R}^n)} \lesssim_{s_j,m,p,n}   \lVert  u \rVert_{\dot{\mathrm{B}}^{s_j}_{p_j,q_j}(\mathbb{R}^n_+)}\text{. }
\end{align*}
The same result holds replacing $(\mathrm{E},\dot{\mathrm{B}}^{s_j}_{p_j,q_j}(\mathbb{R}^n_+),\dot{\mathrm{B}}^{s_j}_{p_j,q_j}(\mathbb{R}^n))$ by $(\mathcal{P}_0,\dot{\mathrm{B}}^{s_j}_{p_j,q_j}(\mathbb{R}^n),\dot{\mathrm{B}}^{s_j}_{p_j,q_j,0}(\mathbb{R}^n_+))$, where $\mathcal{P}_0$ is the projection operator given in Lemma \ref{lemma:ProjHsp0}.

Thus, the following equality of vector spaces holds with equivalence of norms 
\begin{align*}
    \dot{\mathrm{B}}^{s_0}_{p_0,q_0}(\mathbb{R}^n_+)\cap \dot{\mathrm{B}}^{s_1}_{p_1,q_1}(\mathbb{R}^n_+) = [\dot{\mathrm{B}}^{s_0}_{p_0,q_0}\cap \dot{\mathrm{B}}^{s_1}_{p_1,q_1}](\mathbb{R}^n_+)\text{. }
\end{align*}
In particular, $\dot{\mathrm{B}}^{s_0}_{p_0,q_0}(\mathbb{R}^n_+)\cap \dot{\mathrm{B}}^{s_1}_{p_1,q_1}(\mathbb{R}^n_+)$ is a Banach space, and it admits $\eus{S}_0(\overline{\mathbb{R}^n_+})$ as a dense subspace whenever $q_j<+\infty$, $j\in\{ 0,1\}$.

Similarly, the following equality with equivalence of norms holds for all $s>0$, $q\in[1,+\infty]$,
\begin{align*}
    {\mathrm{L}}^{p}(\mathbb{R}^n_+)\cap \dot{\mathrm{B}}^{s}_{p,q}(\mathbb{R}^n_+) = {\mathrm{B}}^{s}_{p,q}(\mathbb{R}^n_+)\text{. }
\end{align*}
\end{proposition}

With direct consequence similar to Corollary \ref{cor:EqNormNablakmHspHaq}:

\begin{corollary}\label{cor:EqNormNablakmBspBaq} Let $p_j\in(1,+\infty)$, $q_j\in[1,+\infty]$ $m_j\in\llb 1,+\infty\llb$, $s_j>m_j-1+\tfrac{1}{p_j} $, $j\in\{ 0,1\}$, such that $(\mathcal{C}_{s_0,p_0,q_0})$ is satisfied. For all $u\in [\dot{\mathrm{B}}^{s_0}_{p_0,q_0}\cap \dot{\mathrm{B}}^{s_1}_{p_1,q_1}] (\mathbb{R}^n_+)$,
\begin{align*}
    \lVert \nabla^{m_j} u \rVert_{\dot{\mathrm{B}}^{s_j-m_j}_{p_j,q_j}(\mathbb{R}^n_+)} \sim_{s_j,m_j,p_j,n} \lVert u \rVert_{\dot{\mathrm{B}}^{s_j}_{p_j,q_j}(\mathbb{R}^n_+)}\,\text{. }
\end{align*}
\end{corollary}

Above Proposition \ref{prop:IntersLpHomBesovRn+=BesovRn+} also implies the expected interpolation result for Besov spaces, for which the proof is similar to the one of Proposition \ref{prop:InterpHomSpacesRn+} and left again to the reader. 
\begin{proposition}\label{prop:InterpHomBesSpacesRn+}Let $(p_0,p_1,p,q,q_0,q_1)\in(1,+\infty)^3\times[1,+\infty]^3$, $s_0,s_1\in\mathbb{R}$, such that $s_0< s_1$, and let  $\mathfrak{b}\in\{\mathrm{B},\mathrm{B}_{\cdot,\cdot,0}\}$, and set
\begin{align*}
    \left(s,\frac{1}{p_\theta},\frac{1}{q_\theta}\right):= (1-\theta)\left(s_0,\frac{1}{p_0},\frac{1}{q_0}\right)+ \theta\left(s_1,\frac{1}{p_1},\frac{1}{q_1}\right)\text{. }
\end{align*}
such that the following assertion is satisfied,
\begin{itemize}
    \item $s_j>-1+\frac{1}{p_j}$, $j\in\{0,1\}$, and $(\mathcal{C}_{s_0,p_0,q_0})$ is true;
\end{itemize}
Then if $p_0=p_1=p$, and \eqref{AssumptionCompletenessExponents} is satisfied, the following equality holds with equivalence of norms
\begin{align}
    (\dot{\mathfrak{b}}^{s_0}_{p,q_0}(\mathbb{R}^n_+),\dot{\mathfrak{b}}^{s_1}_{p,q_1}(\mathbb{R}^n_+))_{\theta,q}&=\dot{\mathfrak{b}}^{s}_{p,q}(\mathbb{R}^n_+)\text{.}
\end{align}
If $(\mathcal{C}_{s_0,p_0,q_0})$ and $(\mathcal{C}_{s_1,p_1,q_1})$ are true then also is $(\mathcal{C}_{s,p_\theta,q_\theta})$ and with equivalence of norms,
    \begin{align}
    [\dot{\mathfrak{b}}^{s_0}_{p_0,q_0}(\mathbb{R}^n_+),\dot{\mathfrak{b}}^{s_1}_{p_1,q_1}(\mathbb{R}^n_+)]_{\theta} &= \dot{\mathfrak{b}}^{s}_{p_\theta,q_\theta}(\mathbb{R}^n_+) \text{,}
    \end{align}
whenever $q_\theta<+\infty$.
\end{proposition}

We finish stating a duality result for homogeneous Besov space son the half-space.
\begin{proposition}\label{prop:dualityBesovRn+}Let $p\in(1,+\infty)$, $q\in(1,+\infty]$, $s>-1+\frac{1}{p}$, if \eqref{AssumptionCompletenessExponents} is satisfied then the following isomorphisms hold
\begin{align*}
    (\dot{\mathrm{B}}^{-s}_{p',q',0}(\mathbb{R}^n_+))'=\dot{\mathrm{B}}^{s}_{p,q}(\mathbb{R}^n_+)\, \text{ and }\, (\dot{\mathrm{B}}^{-s}_{p',q'}(\mathbb{R}^n_+))'=\dot{\mathrm{B}}^{s}_{p,q,0}(\mathbb{R}^n_+)\text{. }
\end{align*}
\end{proposition}

\begin{proof}We only prove $(\dot{\mathrm{B}}^{-s}_{p',q'}(\mathbb{R}^n_+))'=\dot{\mathrm{B}}^{s}_{p,q,0}(\mathbb{R}^n_+)$, the other equality can be shown similarly. First let  $q<+\infty$, and choose $u\in \dot{\mathrm{B}}^{s}_{p,q,0}(\mathbb{R}^n_+)$, it follows that $u$ induce a linear form on $\dot{\mathrm{B}}^{-s}_{p',q'}(\mathbb{R}^n_+)$,
\begin{align*}
    v\longmapsto \big\langle u,\tilde{v}\big\rangle_{\mathbb{R}^n}
\end{align*}
where $\tilde{v}\in\dot{\mathrm{B}}^{-s}_{p',q'}(\mathbb{R}^n)$ is any extension of $v\in\dot{\mathrm{B}}^{-s}_{p',q'}(\mathbb{R}^n_+)$. If one choose $v'$ to be any other extension of $v$, we have that $\tilde{v}-v'\in\dot{\mathrm{B}}^{-s}_{p',q',0}(\mathbb{R}^n_-)$. Since the space $\mathrm{C}_c^\infty(\mathbb{R}^n_+)$ is dense in $\dot{\mathrm{B}}^{s}_{p,q,0}(\mathbb{R}^n_+)$, see either Lemma~\ref{lem:densityCcinftyBesov0s+} or Corollary~\ref{cor:Bspq=Bspq0Rn+}, for $(u_k)_{k\in\mathbb{N}}\subset \mathrm{C}_c^\infty(\mathbb{R}^n_+)$ converging to $u$, we have
\begin{align*}
    \big\langle u,\tilde{v}-v'\big\rangle_{\mathbb{R}^n}=\lim_{k\rightarrow +\infty} \big\langle u_k,\tilde{v}-v'\big\rangle_{\mathbb{R}^n} = 0
\end{align*}
due to the fact that $\mathbb{R}^n_+\cap\mathbb{R}^n_- = \emptyset$. Thus, the map does not depend on the choice of the extension, but is entirely and uniquely determined by $u$. We have the continuous canonical embedding
\begin{align*}
    \dot{\mathrm{B}}^{s}_{p,q,0}(\mathbb{R}^n_+)\hookrightarrow(\dot{\mathrm{B}}^{-s}_{p',q'}(\mathbb{R}^n_+))'\text{. }
\end{align*}
In fact, the same result holds for $q=+\infty$: the space $\mathrm{C}_c^\infty(\mathbb{R}^n_+)$ is sequentially weak${}^\ast$ dense in $\dot{\mathrm{B}}^{s}_{p,\infty,0}(\mathbb{R}^n_+)$ by Corollary \ref{cor:weakstardensity}. 

For the reverse embedding, if $ U \in (\dot{\mathrm{B}}^{-s}_{p',q'}(\mathbb{R}^n_+))'$, it induces a continuous linear functional on $\dot{\mathrm{B}}^{-s}_{p',q'}(\mathbb{R}^n)$ by the mean of
\begin{align*}
    v\longmapsto \big\langle U, \mathbbm{1}_{\mathbb{R}^n_+}\tilde{v}\big\rangle \textit{ , }
\end{align*}
where again $\tilde{v}\in\dot{\mathrm{B}}^{-s}_{p',q'}(\mathbb{R}^n)$ is any extension of $v\in\dot{\mathrm{B}}^{-s}_{p',q'}(\mathbb{R}^n_+)$. Thus, $\mathbbm{1}_{\mathbb{R}^n_+} U \in (\dot{\mathrm{B}}^{-s}_{p',q'}(\mathbb{R}^n))'$ and by Proposition \ref{prop:dualityHomBesov} there exists a unique $u\in \dot{\mathrm{B}}^{s}_{p,q}(\mathbb{R}^n)$ such that, for all $\tilde{v}\in \dot{\mathrm{B}}^{-s}_{p',q'}(\mathbb{R}^n)$,
\begin{align*}
    \big\langle U, \mathbbm{1}_{\mathbb{R}^n_+}\tilde{v}\big\rangle = \big\langle u, \tilde{v}\big\rangle_{\mathbb{R}^n}\text{. }
\end{align*}
Finally, if we test with $\tilde{v}\in \mathrm{C}_c^\infty(\mathbb{R}^n_-)$, it shows that $\supp u \subset \overline{\mathbb{R}^n_+}$, then $u\in \dot{\mathrm{B}}^{s}_{p,q,0}(\mathbb{R}^n_+)$ which closes the proof.
\end{proof}


\section{On traces of functions}\label{Sec:TracesofFunctions}

Dealing with function spaces on domains implies that one may need to investigate the meaning of traces at the boundary if those exist, \textit{i.e.}, to see in our setting if the trace operator
\begin{align*}
    \gamma_0\,:\, u \,\longmapsto\, u_{|_{\partial\mathbb{R}^n_+}}
\end{align*}
still has the expected behavior on $\dot{\mathrm{H}}^{s,p}(\mathbb{R}^n_+)$ and $\dot{\mathrm{B}}^{s}_{p,q}(\mathbb{R}^n_+)$. In fact, in the complete case, it behaves as in the case of inhomogeneous function spaces.

The idea here is to give some appropriate trace theorems for homogeneous Sobolev and Besov spaces. It seems there is no clear trace theorem for homogeneous function spaces in the literature, except maybe \cite{Jawerth78}, but in this case the corresponding results were obtained in a different framework.

\subsection{On inhomogeneous function spaces.}

We discuss first about the usual well known trace theorem on $\mathbb{R}^n$ with trace on $\mathbb{R}^{n-1}\times\{0\}$ in the inhomogeneous case, the result is a rewritten weaker version adapted to our context.
\begin{theorem}[{\cite[Theorem~6.6.1]{BerghLofstrom1976}} ]\label{thm:inhomTracewholespace} Let $p\in(1,+\infty)$, $q\in[1,+\infty]$, $s\in(\frac{1}{p},+\infty)$, and consider the following operator
\begin{align*}
\gamma_0\,:\,\left\{\begin{array}{rl}
\eus{S}(\mathbb{R}^n) &\longrightarrow \eus{S}(\mathbb{R}^{n-1})\\
u &\longmapsto \,\,\, u(\cdot,0)
\end{array}\right.\text{, }
\end{align*}
then following statements are true:
\begin{enumerate}[label=($\roman*$)]
    \item the trace operator extends uniquely as a bounded surjection $\gamma_{0}\,:\, \mathrm{H}^{s,p}(\mathbb{R}^n)\longrightarrow \mathrm{B}^{s-\frac{1}{p}}_{p,p}(\mathbb{R}^{n-1})$, in particular for all $u\in \mathrm{H}^{s,p}(\mathbb{R}^{n})$,
    \begin{align*}
        \lVert  \gamma_{0}u \rVert_{\mathrm{B}^{s-\frac{1}{p}}_{p,p}(\mathbb{R}^{n-1})}  \lesssim_{s,p,n} \lVert u \rVert_{\mathrm{H}^{s,p}(\mathbb{R}^{n})} ;
    \end{align*}
    \item the trace operator extends uniquely as a bounded surjection $\gamma_{0}\,:\, \mathrm{B}^{s}_{p,q}(\mathbb{R}^{n})\longrightarrow \mathrm{B}^{s-\frac{1}{p}}_{p,q}(\mathbb{R}^{n-1})$, in particular for all $u\in \mathrm{B}^{s}_{p,q}(\mathbb{R}^{n})$, 
    \begin{align*}
        \lVert \gamma_{0} u \rVert_{\mathrm{B}^{s-\frac{1}{p}}_{p,q}(\mathbb{R}^{n-1})} \lesssim_{s,p,n,q} \lVert u \rVert_{\mathrm{B}^{s}_{p,q}(\mathbb{R}^{n})}  ;
    \end{align*}
    \item the trace operator extends uniquely as a bounded surjection $\gamma_{0}\,:\, \mathrm{B}^{\frac{1}{p}}_{p,1}(\mathbb{R}^{n})\longrightarrow \mathrm{L}^p(\mathbb{R}^{n-1})$, in particular for all $u\in \mathrm{B}^{\frac{1}{p}}_{p,1}(\mathbb{R}^{n})$, 
    \begin{align*}
        \lVert \gamma_{0} u \rVert_{\mathrm{L}^p(\mathbb{R}^{n-1})} \lesssim_{p,n} \lVert u \rVert_{\mathrm{B}^{\frac{1}{p}}_{p,1}(\mathbb{R}^{n})}  ;
    \end{align*}
\end{enumerate}
Moreover the trace operator $\gamma_0$ admits a linear right bounded inverse $\mathrm{Ext}$ in cases (i) and (ii).
\end{theorem}

\begin{remark}\label{rem:TracesInhom} $\bullet$ About the boundedness for the trace operator $\gamma_{0}\,:\, \mathrm{B}^{s}_{p,\infty}(\mathbb{R}^{n})\longrightarrow \mathrm{B}^{s-\frac{1}{p}}_{p,\infty}(\mathbb{R}^{n-1})$. The construction does not follow directly by density argument since $\eus{S}(\mathbb{R}^n)$ is not strongly but only weakly* dense in $\mathrm{B}^{s}_{p,\infty}(\mathbb{R}^{n})$. However, $\eus{S}(\mathbb{R}^n)$ is a dense subspace of $\mathrm{B}^{s_0}_{p,1}(\mathbb{R}^n)\cap \mathrm{B}^{s_1}_{p,1}(\mathbb{R}^n)$, so that the trace is consistent on both spaces $\mathrm{B}^{s_0}_{p,1}(\mathbb{R}^n)$ and $\mathrm{B}^{s_1}_{p,1}(\mathbb{R}^n)$, and we can define the trace operator
\begin{align*}
    \gamma_{0}\,:\, \mathrm{B}^{s_0}_{p,1}(\mathbb{R}^{n}) + \mathrm{B}^{s_1}_{p,1}(\mathbb{R}^{n}) \longrightarrow \mathrm{B}^{s_0-\frac{1}{p}}_{p,1}(\mathbb{R}^{n-1}) + \mathrm{B}^{s_1-\frac{1}{p}}_{p,1}(\mathbb{R}^{n-1}).
\end{align*}
But by real interpolation, if $1/p<s_0<s<s_1$ we have $$\mathrm{B}^{s}_{p,\infty}(\mathbb{R}^n)=(\mathrm{B}^{s_0}_{p,1}(\mathbb{R}^n),\mathrm{B}^{s_1}_{p,1}(\mathbb{R}^n))_{\frac{s-s_0}{s_1-s_0},\infty} \hookrightarrow \mathrm{B}^{s_0}_{p,1}(\mathbb{R}^n)+\mathrm{B}^{s_1}_{p,1}(\mathbb{R}^n).$$
Thus, we have the well-defined (and bounded) trace operator
\begin{align*}
    \gamma_{0}\,:\, \mathrm{B}^{s}_{p,\infty}(\mathbb{R}^{n}) \longrightarrow \mathrm{B}^{s_0-\frac{1}{p}}_{p,1}(\mathbb{R}^{n-1}) + \mathrm{B}^{s_1-\frac{1}{p}}_{p,1}(\mathbb{R}^{n-1}).
\end{align*}
It remains then to prove the boundedness $\gamma_{0}\,:\, \mathrm{B}^{s}_{p,\infty}(\mathbb{R}^{n}) \longrightarrow \mathrm{B}^{s-\frac{1}{p}}_{p,\infty}(\mathbb{R}^{n-1})$, which is given by real interpolation.

A simpler argument is available (arguing that $\mathrm{B}^{s}_{p,\infty}(\mathbb{R}^{n})\hookrightarrow \mathrm{B}^{s-\varepsilon}_{p,1}(\mathbb{R}^{n})$) but the one presented here will also apply to the trace operator in the case of homogeneous Besov spaces, for the Theorem~\ref{thm:Tracehalfspace}.

$\bullet$ One also mentions \cite[Theorems~2.2~\&~2.10]{Schneider2010}, \cite[Sections~V-VII]{JonssonWallin1984}, which give different proofs of the trace theorem (in a more general geometric setting). Notice that in  \cite[Theorems~2.2~\&~2.10]{Schneider2010} and \cite[Theorems~4.47,~4.48]{bookSawano2018} the right bounded inverse they give is not linear but covers case \textit{(iii)}.
\end{remark}

\subsection{On homogeneous function spaces.}

\begin{theorem}\label{thm:Tracehalfspace}Let $p\in(1,+\infty)$, $q\in[1,+\infty]$, $s\in(\frac{1}{p},+\infty)$, then for $(\mathfrak{h},\mathfrak{b})\in\{ (\mathrm{H},\mathrm{B}),\, (\dot{\mathrm{H}},\dot{\mathrm{B}})\}$, we consider the trace operator
\begin{align*}
\gamma_0\,:\,u &\longmapsto \, u(\cdot,0)\text{. }
\end{align*}
The following assertions are true.
\begin{enumerate}[label=($\roman*$)]
    \item For all $u\in \mathrm{H}^{s,p}(\mathbb{R}^{n}_+)$, we have $u\in\mathrm{C}^0_{0,x_n}(\overline{\mathbb{R}_+},\mathfrak{b}^{s-\frac{1}{p}}_{p,p}(\mathbb{R}^{n-1}))$, with the estimate
    \begin{align*}
        \lVert  u \rVert_{\mathrm{L}^{\infty}_{x_n}(\mathbb{R}_+,\mathfrak{b}^{s-\frac{1}{p}}_{p,p}(\mathbb{R}^{n-1}))}  \lesssim_{s,p,n} \lVert u \rVert_{\mathfrak{h}^{s,p}(\mathbb{R}^{n}_+)} ;
    \end{align*}
    In particular, the trace operator extends uniquely to a bounded linear operator $$\gamma_0\,:\,\dot{\mathrm{H}}^{s,p}(\mathbb{R}^n_+)\rightarrow \dot{\mathrm{B}}^{s-\tfrac{1}{p}}_{p,p}(\mathbb{R}^{n-1})$$ whenever $(\mathcal{C}_{s,p})$ is satisfied, and the following continuous embedding holds
    \begin{align*}
        \dot{\mathrm{H}}^{s,p}(\mathbb{R}^n_+)\hookrightarrow \mathrm{C}^{0}_{0,x_n}(\overline{\mathbb{R}_+},\dot{\mathrm{B}}^{s-\frac{1}{p}}_{p,p}(\mathbb{R}^{n-1})) \text{. }
    \end{align*}
    \item For all $u\in \mathrm{B}^{s}_{p,q}(\mathbb{R}^{n}_+)$, we have $u\in\mathrm{C}^0_{0,x_n}(\overline{\mathbb{R}_+},\mathfrak{b}^{s-\frac{1}{p}}_{p,q}(\mathbb{R}^{n-1}))$, with the estimate
    \begin{align*}
        \lVert  u \rVert_{\mathrm{L}^{\infty}_{x_n}(\mathbb{R}_+,\mathfrak{b}^{s-\frac{1}{p}}_{p,q}(\mathbb{R}^{n-1}))}  \lesssim_{s,p,n} \lVert u \rVert_{\mathfrak{b}^{s}_{p,q}(\mathbb{R}^{n}_+)} ;
    \end{align*}
    In particular, the trace operator extends uniquely to a bounded linear operator $$\gamma_0\,:\,\dot{\mathrm{B}}^{s}_{p,q}(\mathbb{R}^n_+)\rightarrow \dot{\mathrm{B}}^{s-\tfrac{1}{p}}_{p,q}(\mathbb{R}^{n-1})$$ whenever \eqref{AssumptionCompletenessExponents} is satisfied, and the following continuous embedding holds
    \begin{align*}
        \dot{\mathrm{B}}^{s}_{p,q}(\mathbb{R}^n_+)\hookrightarrow \mathrm{C}^{0}_{0,x_n}(\overline{\mathbb{R}_+},\dot{\mathrm{B}}^{s-\frac{1}{p}}_{p,q}(\mathbb{R}^{n-1})) \text{. }
    \end{align*}
    
    If $q=+\infty$, the result still holds with uniform boundedness and weak${}^\ast$ continuity only.
    \item For all $u\in \mathrm{B}^{1/p}_{p,1}(\mathbb{R}^{n}_+)$, we have $u\in\mathrm{C}^0_{0,x_n}(\overline{\mathbb{R}_+},\mathrm{L}^{p}(\mathbb{R}^{n-1}))$, with the estimate
    \begin{align*}
        \lVert  u \rVert_{\mathrm{L}^{\infty}_{x_n}(\mathbb{R}_+,\mathrm{L}^{p}(\mathbb{R}^{n-1}))}  \lesssim_{s,p,n} \lVert u \rVert_{\mathfrak{b}^{1/p}_{p,1}(\mathbb{R}^{n}_+)} ;
    \end{align*}
    In particular, the trace operator extends uniquely to a bounded linear operator $$\gamma_0\,:\,\dot{\mathrm{B}}^{1/p}_{p,1}(\mathbb{R}^n_+)\rightarrow \mathrm{L}^{p}(\mathbb{R}^{n-1})$$ and the following continuous embedding holds
    \begin{align*}
        \dot{\mathrm{B}}^{1/p}_{p,1}(\mathbb{R}^n_+)\hookrightarrow \mathrm{C}^{0}_{0,x_n}(\overline{\mathbb{R}_+},\mathrm{L}^{p}(\mathbb{R}^{n-1})) \text{. }
    \end{align*}
\end{enumerate}
Moreover, 
\begin{enumerate}[label=($\alph*$)]
    \item If $(\mathfrak{h},\mathfrak{b}) = (\mathrm{H},\mathrm{B})$, the trace operator $\gamma_0$ admits a linear right bounded inverse $\mathrm{Ext}_{\mathbb{R}^n_+}$ in cases (i) and (ii).
    \item If $(\mathfrak{h},\mathfrak{b}) = (\dot{\mathrm{H}},\dot{\mathrm{B}})$, the trace operator $\gamma_0$ admits a linear right bounded inverse $\underline{\mathrm{Ext}}_{\mathbb{R}^n_+}$ in cases (i) and (ii). 
\end{enumerate}
\end{theorem}

\begin{proof}We cut the proof in several steps.

\textbf{Step 1: } The case $(\mathfrak{h},\mathfrak{b}) = (\mathrm{H},\mathrm{B})$.

We first check validity of the embedding
\begin{align*}
    \mathrm{H}^{s,p}(\mathbb{R}^n)\hookrightarrow \mathrm{C}^{0}_{0,x_n}(\mathbb{R},\mathrm{B}^{s-\frac{1}{p}}_{p,p}(\mathbb{R}^{n-1}))\text{. }
\end{align*}
Let $u\in \mathrm{H}^{s,p}(\mathbb{R}^n)$, for $t>0$, for almost every $x=(x',x_n)\in\mathbb{R}^n$, we introduce $u_t(x',x_n):= u(x',x_n+t)$, we have $u_t \in \mathrm{H}^{s,p}(\mathbb{R}^n)$, and by Theorem \ref{thm:inhomTracewholespace},
\begin{align*}
\lVert \gamma_0 u_t  \rVert_{\mathrm{B}^{s-\frac{1}{p}}_{p,p}(\mathbb{R}^{n-1})} &\lesssim_{p,s,n} \lVert u \rVert_{\mathrm{H}^{s,p}(\mathbb{R}^n)} \text{, }\\
    \lVert \gamma_0 (u_t -u) \rVert_{\mathrm{B}^{s-\frac{1}{p}}_{p,p}(\mathbb{R}^{n-1})} &\lesssim_{p,s,n} \lVert u_t -u \rVert_{\mathrm{H}^{s,p}(\mathbb{R}^n)}\text{. }
\end{align*}
Therefore, by strong continuity of translation in Lebesgue spaces, then in Sobolev spaces, we obtain
\begin{align*}
    \lVert \gamma_0 (u_t -u) \rVert_{\mathrm{B}^{s-\frac{1}{p}}_{p,p}(\mathbb{R}^{n-1})} &\lesssim_{p,s,n} \lVert u_t -u \rVert_{\mathrm{H}^{s,p}(\mathbb{R}^n)} \xrightarrow[t\rightarrow 0]{} 0\text{. }
\end{align*}
Hence, $u\in\mathrm{C}_{b,x_n}^0(\mathbb{R},\mathrm{B}^{s-\frac{1}{p}}_{p,p}(\mathbb{R}^{n-1}))$, with the estimate,
\begin{align*}
    \Big\lVert t\mapsto\lVert u(\cdot,t) \rVert_{\mathrm{B}^{s-\frac{1}{p}}_{p,p}(\mathbb{R}^{n-1})}\Big\rVert_{\mathrm{L}^{\infty}(\mathbb{R})} &\lesssim_{p,s,n} \lVert u \rVert_{\mathrm{H}^{s,p}(\mathbb{R}^n)} \text{. }
\end{align*}
Finally, one can approximate $u$ by Schwartz functions to deduce
\begin{align*}
    u\in\mathrm{C}_{0,x_n}^0(\mathbb{R},\mathrm{B}^{s-\frac{1}{p}}_{p,p}(\mathbb{R}^{n-1}))\text{. }
\end{align*}
One may perform a similar proof for all other cases, and one may check \cite[Proposition~1.9]{Guidetti1991Interp} for the continuity of translation in Besov spaces, one may also use a density and an interpolation argument. Concerning the case of Besov spaces with index $q=+\infty$, the argument is postponed to the end of the next \textbf{Step 2.1}.

Now, one can use the definition of function spaces by restriction.

One choose $\mathrm{Ext}_{\mathbb{R}^n_+}=[\mathrm{Ext} \cdot]_{|_{\mathbb{R}^n_+}}$ introduced in the proof of Theorem \ref{thm:inhomTracewholespace} which satisfies the desired boundedness properties. 

\textbf{Step 2.1: } The case $(\mathfrak{h},\mathfrak{b}) = (\dot{\mathrm{H}},\dot{\mathrm{B}})$. Boundedness of the trace operator.

We only treat the case \textit{(ii)} other ones can be done similarly. From \textbf{Step 1}, and for fixed $p\in(1,+\infty)$, $q\in[1,+\infty]$, $s>\frac{1}{p}$, and $u\in {\mathrm{B}}^{s}_{p,q}(\mathbb{R}^n_+)$, we have
\begin{align*}
    \lVert  u \rVert_{\mathrm{L}^{\infty}_{x_n}(\mathbb{R}_+,\dot{\mathrm{B}}^{s-\frac{1}{p}}_{p,q}(\mathbb{R}^{n-1}))} \lesssim_{p,s,n}\lVert  u \rVert_{\mathrm{L}^{\infty}_{x_n}(\mathbb{R}_+,\mathrm{B}^{s-\frac{1}{p}}_{p,q}(\mathbb{R}^{n-1}))}  \lesssim_{s,p,n} \lVert u \rVert_{\mathrm{B}^{s}_{p,q}(\mathbb{R}^{n}_+)}  \text{. }
\end{align*}
Thus, one may use the fact that ${\mathrm{B}^{s}_{p,q}(\mathbb{R}^{n}_+)} = \mathrm{L}^p(\mathbb{R}^n_+) \cap {\dot{\mathrm{B}}^{s}_{p,q}(\mathbb{R}^{n}_+)}$, which comes from Proposition~\ref{prop:IntersLpHomBesovRn+=BesovRn+}, to obtain
\begin{align*}
     \lVert  u \rVert_{\mathrm{L}^{\infty}_{x_n}(\mathbb{R}_+,\dot{\mathrm{B}}^{s-\frac{1}{p}}_{p,q}(\mathbb{R}^{n-1}))} \lesssim_{s,p,n,q} \lVert  u \rVert_{\mathrm{L}^{p}(\mathbb{R}^{n}_+)} +  \lVert u \rVert_{\dot{\mathrm{B}}^{s}_{p,q}(\mathbb{R}^{n}_+)}  \text{. }
\end{align*}
So that, by a dilation argument, replacing $u$, by $u_\lambda:= u(\lambda \cdot)$, for $\lambda \in 2^\mathbb{N}$, 
\begin{align*}
     \lambda^{s- \frac{n}{p}}\lVert  u \rVert_{\mathrm{L}^{\infty}_{x_n}(\mathbb{R}_+,\dot{\mathrm{B}}^{s-\frac{1}{p}}_{p,q}(\mathbb{R}^{n-1}))} \lesssim_{s,p,n,q} \lambda^{- \frac{n}{p}}\lVert  u \rVert_{\mathrm{L}^{p}(\mathbb{R}^{n}_+)} +  \lambda^{s- \frac{n}{p}}\lVert u \rVert_{\dot{\mathrm{B}}^{s}_{p,q}(\mathbb{R}^{n}_+)}  \text{. }
\end{align*}
Hence, we can divide by $\lambda^{s-\frac{n}{p}}$ on both sides and pass to the limit $\lambda \longrightarrow +\infty$,
\begin{align*}
    \lVert  u \rVert_{\mathrm{L}^{\infty}_{x_n}(\mathbb{R}_+,\dot{\mathrm{B}}^{s-\frac{1}{p}}_{p,q}(\mathbb{R}^{n-1}))} \lesssim_{s,p,n,q} \lVert u \rVert_{\dot{\mathrm{B}}^{s}_{p,q}(\mathbb{R}^{n}_+)}  \text{. }
\end{align*}
Therefore, if $q<+\infty$, and \eqref{AssumptionCompletenessExponents} is satisfied, the embedding
\begin{align*}
    \dot{\mathrm{B}}^{s}_{p,q}(\mathbb{R}^n_+)\hookrightarrow \mathrm{C}^{0}_{0,x_n}(\overline{\mathbb{R}_+},\dot{\mathrm{B}}^{s-\frac{1}{p}}_{p,q}(\mathbb{R}^{n-1}))
\end{align*}
holds by density. If $q=+\infty$, and \eqref{AssumptionCompletenessExponents} is satisfied, the result follows from real interpolation applied to the map $u\mapsto u(\cdot,x_n)=\gamma_0 u_{x_n}$ yielding for all $x_n\in[0,+\infty)$,
\begin{align}\label{eq:traceEstimateBspInfty}
    \lVert  u(\cdot,x_n) \rVert_{\dot{\mathrm{B}}^{s-\frac{1}{p}}_{p,\infty}(\mathbb{R}^{n-1})} \lesssim_{s,p,n} \lVert u \rVert_{\dot{\mathrm{B}}^{s}_{p,\infty}(\mathbb{R}^{n}_+)}
\end{align}
for more details about the meaning of traces in this case $q=+\infty$, see the discussion in the first part of Remark \ref{rem:TracesInhom}. Notice that the bound is uniform with respect to $x_n$.

It remains to show weak* continuity with respect to $x_n$. This follows from  Lemma~\ref{lem:weak*continuityBspinfty} in Appendix~\ref{Append:ContinTranslation}. Let $\phi\in \dot{\mathrm{B}}^{-s+1-1/p'}_{p',1}(\mathbb{R}^{n-1})$. One has by Lemma~\ref{lem:weak*continuityBspinfty}
\begin{align*}
    \langle u(\cdot,x_n), \phi\rangle_{\mathbb{R}^{n-1}} = \langle \gamma_0 u_{x_n}, \phi\rangle_{\mathbb{R}^{n-1}} = \langle u_{x_n}, \gamma_0^\ast\phi\rangle_{\mathbb{R}^{n}_+}\xrightarrow[x_n\rightarrow 0]{} \langle u , \gamma_0^\ast\phi\rangle_{\mathbb{R}^{n}_+} = \langle \gamma_0 u , \phi\rangle_{\mathbb{R}^{n-1}},
\end{align*}
where $\gamma_0^{\ast} = \mathrm{I}\otimes[\delta_{0}'\ast]\,:\,\dot{\mathrm{B}}^{-s+1-1/p'}_{p',1}(\mathbb{R}^{n-1})\longrightarrow \dot{\mathrm{B}}^{-s}_{p',1,0}(\mathbb{R}^{n}_+)$ is the (pre-)dual map of the trace operator.

This yields the weak* continuity with respect to $x_n$, from which we deduce the (weak*) measurability of $x_n\mapsto u(\cdot,x_n)$, and this allows to take the supremum in \eqref{eq:traceEstimateBspInfty}.

\textbf{Step 2.2: } The case $(\mathfrak{h},\mathfrak{b}) = (\dot{\mathrm{H}},\dot{\mathrm{B}})$. Boundedness of the extension operator.

The operator $T$ given by Proposition \ref{prop:HarmExtenOpSobolevBesov} is an appropriate extension operator which satisfies the desired boundedness properties. Thus $\underline{\mathrm{Ext}}_{\mathbb{R}^n_+}:=T$ behaves as expected.
\end{proof}

A raised question is about what happens when we want to deal with intersection of homogeneous Sobolev and Besov spaces.

\begin{proposition}\label{prop:TracehalfspaceIntersec1}Let $p\in(1,+\infty)$, $q\in[1,+\infty)$, $-1+\frac{1}{p} < s_0 < \frac{1}{p} < s_1$, and $\theta\in(0,1)$ such that
\begin{align*}
    \frac{1}{p}=(1-\theta)s_0+\theta s_1\text{. }
\end{align*}
Then the following assertions hold.
\begin{enumerate}[label=($\roman*$)]
    \item for all $u\in \dot{\mathrm{H}}^{s_0,p}(\mathbb{R}^n_+)\cap \dot{\mathrm{H}}^{s_1,p}(\mathbb{R}^n_+)$, we have $\gamma_0 u \in \mathrm{B}^{s_1-\frac{1}{p}}_{p,p}(\mathbb{R}^{n-1})$, with the estimate
    \begin{align*}
        \lVert  \gamma_{0}u \rVert_{\mathrm{B}^{s_1-\frac{1}{p}}_{p,p}(\mathbb{R}^{n-1})}  \lesssim_{s_0,s_1,p,n} \lVert u \rVert_{\dot{\mathrm{H}}^{s_0,p}(\mathbb{R}^{n}_+)}^{1-\theta}\lVert u \rVert_{\dot{\mathrm{H}}^{s_1,p}(\mathbb{R}^{n}_+)}^\theta +  \lVert u \rVert_{\dot{\mathrm{H}}^{s_1,p}(\mathbb{R}^{n}_+)} \text{. }
    \end{align*}
    We also have,
    \begin{align*}
        \lVert  \gamma_{0}u \rVert_{\dot{\mathrm{B}}^{s_1-\frac{1}{p}}_{p,p}(\mathbb{R}^{n-1})}  \lesssim_{s_0,s_1,p,n}  \lVert u \rVert_{\dot{\mathrm{H}}^{s_1,p}(\mathbb{R}^{n}_+)} \text{ ; }
    \end{align*}
    \item for all $u\in \dot{\mathrm{B}}^{s_0}_{p,q}(\mathbb{R}^n_+)\cap \dot{\mathrm{B}}^{s_1}_{p,q}(\mathbb{R}^n_+)$, we have $\gamma_0 u \in \mathrm{B}^{s_1-\frac{1}{p}}_{p,q}(\mathbb{R}^{n-1})$, with the estimate
    \begin{align*}
        \lVert  \gamma_{0}u \rVert_{\mathrm{B}^{s_1-\frac{1}{p}}_{p,q}(\mathbb{R}^{n-1})}  \lesssim_{s_0,s_1,p,n} \lVert u \rVert_{\dot{\mathrm{B}}^{s_0}_{p,q}(\mathbb{R}^{n}_+)}^{1-\theta}\lVert u \rVert_{\dot{\mathrm{B}}^{s_1}_{p,q}(\mathbb{R}^{n}_+)}^\theta +  \lVert u \rVert_{\dot{\mathrm{B}}^{s_1}_{p,q}(\mathbb{R}^{n}_+)} \text{. }
    \end{align*}
    We also have,
    \begin{align*}
        \lVert  \gamma_{0}u \rVert_{\dot{\mathrm{B}}^{s_1-\frac{1}{p}}_{p,q}(\mathbb{R}^{n-1})}  \lesssim_{s_0,s_1,p,n} \lVert u \rVert_{\dot{\mathrm{B}}^{s_1}_{p,q}(\mathbb{R}^{n}_+)} ;
    \end{align*}
    \item for all $u\in \dot{\mathrm{B}}^{s_0}_{p,\infty}(\mathbb{R}^n_+)\cap \dot{\mathrm{B}}^{s_1}_{p,\infty}(\mathbb{R}^n_+)$, we have $\gamma_0 u \in \mathrm{L}^{p}(\mathbb{R}^{n-1})$, with the estimate
    \begin{align*}
        \lVert  \gamma_{0}u \rVert_{\mathrm{L}^{p}(\mathbb{R}^{n-1})}  \lesssim_{s_0,s_1,p,n}  \lVert u \rVert_{\dot{\mathrm{B}}^{s_0}_{p,\infty}(\mathbb{R}^{n}_+)}^{1-\theta}\lVert u \rVert_{\dot{\mathrm{B}}^{s_1}_{p,\infty}(\mathbb{R}^{n}_+)}^{\theta} \text{. }
    \end{align*}
\end{enumerate}
\end{proposition}

\begin{proof}We only start proving the point \textit{(ii)}, and claim that point \textit{(i)} can be achieved similarly. We start noticing, the following continuous embedding,
\begin{align*}
    \dot{\mathrm{B}}^{s_0}_{p,q}(\mathbb{R}^n_+)\cap \dot{\mathrm{B}}^{s_1}_{p,q}(\mathbb{R}^n_+) \overset{\iota}{\hookrightarrow} (\dot{\mathrm{B}}^{s_0}_{p,q}(\mathbb{R}^n_+), \dot{\mathrm{B}}^{s_1}_{p,q}(\mathbb{R}^n_+))_{\theta,1} = \dot{\mathrm{B}}^{\frac{1}{p}}_{p,1}(\mathbb{R}^n_+) \overset{\gamma_0}{\hookrightarrow} \mathrm{L}^p (\mathbb{R}^{n-1})\text{. }
\end{align*}
Here, $\iota$ is the canonical embedding obtained via standard interpolation theory, and the last embedding via the trace operator is a direct consequence of Theorem \ref{thm:Tracehalfspace}, and everything can be turned into the following inequality,
\begin{align*}
    \lVert  \gamma_{0}u \rVert_{\mathrm{L}^{p}(\mathbb{R}^{n-1})}  \lesssim_{s_0,s_1,p,n} \lVert u \rVert_{\dot{\mathrm{B}}^{s_0}_{p,q}(\mathbb{R}^{n}_+)}^{1-\theta}\lVert u \rVert_{\dot{\mathrm{B}}^{s_1}_{p,q}(\mathbb{R}^{n}_+)}^\theta\text{, } \forall u\in\dot{\mathrm{B}}^{s_0}_{p,q}(\mathbb{R}^n_+)\cap \dot{\mathrm{B}}^{s_1}_{p,q}(\mathbb{R}^n_+) \text{. }
\end{align*}
Again, from Theorem \ref{thm:Tracehalfspace} we obtain for all $u\in\eus{S}_0(\overline{\mathbb{R}^n_+})$,
\begin{align*}
    \lVert  \gamma_{0}u \rVert_{\dot{\mathrm{B}}^{s_1-\frac{1}{p}}_{p,q}(\mathbb{R}^{n-1})}  \lesssim_{s_1,p,n}  \lVert u \rVert_{\dot{\mathrm{B}}^{s_1}_{p,q}(\mathbb{R}^{n}_+)}\text{. }
\end{align*}
Then one may sum both inequality, notice that $\mathrm{L}^{p}(\mathbb{R}^{n-1})\cap \dot{\mathrm{B}}^{s_1-\frac{1}{p}}_{p,p}(\mathbb{R}^{n-1}) = {\mathrm{B}}^{s_1-\frac{1}{p}}_{p,p}(\mathbb{R}^{n-1})$ and use the density argument provided by Proposition \ref{prop:IntersLpHomBesovRn+=BesovRn+} so that each estimate holds.
\end{proof}

\begin{remark}As in Theorem \ref{thm:Tracehalfspace}, above Proposition \ref{prop:TracehalfspaceIntersec1} could be turned into an $\mathrm{C}^{0}_{0,x_n}$-embedding in the appropriate Besov space.
\end{remark}

\begin{proposition}\label{prop:TracehalfspaceIntersec2}Let $p_j\in(1,+\infty)$, $q_j\in[1,+\infty)$, $ s_j> 1/p_j$, $j\in\{0,1\}$, such that $(\mathcal{C}_{s_0,p_0})$ (resp. $(\mathcal{C}_{s_0,p_0,q_0})$) is satisfied.
Then,
\begin{enumerate}[label=($\roman*$)]
    \item For all $u\in [\dot{\mathrm{H}}^{s_0,p_0}\cap\dot{\mathrm{H}}^{s_1,p_1}](\mathbb{R}^n_+) $, we have $\gamma_0 u \in \dot{\mathrm{B}}^{s_j-\frac{1}{p_j}}_{p_j,p_j}(\mathbb{R}^{n-1})$, $j\in\{0,1\}$, with the estimate
    \begin{align*}
        \lVert  \gamma_{0}u \rVert_{\dot{\mathrm{B}}^{s_j-\frac{1}{p_j}}_{p_j,p_j}(\mathbb{R}^{n-1})}  \lesssim_{s_j,p_j,n} \lVert u \rVert_{\dot{\mathrm{H}}^{s_j,p_j}(\mathbb{R}^{n}_+)} ;
    \end{align*}
    \item For all $u\in [\dot{\mathrm{B}}^{s_0}_{p_0,q_0}\cap \dot{\mathrm{B}}^{s_1}_{p_1,q_1}](\mathbb{R}^n_+)$, we have $\gamma_0 u \in \dot{\mathrm{B}}^{s_j-\frac{1}{p_j}}_{p_j,q_j}(\mathbb{R}^{n-1})$, $j\in\{0,1\}$, with the estimate
    \begin{align*}
        \lVert  \gamma_{0}u \rVert_{\dot{\mathrm{B}}^{s_j-\frac{1}{p_j}}_{p_j,q_j}(\mathbb{R}^{n-1})}  \lesssim_{s_j,p_j,n} \lVert u \rVert_{\dot{\mathrm{B}}^{s_j}_{p_j,q_j}(\mathbb{R}^{n}_+)};
    \end{align*}
\end{enumerate}
\end{proposition}

\begin{remark}Corollary \ref{cor:HarmExtenOpSobolevBesovIntersec} yields the ontoness of the trace operator on intersection spaces given by above Proposition \ref{prop:TracehalfspaceIntersec2}.
\end{remark}

From now on, since we justified that the meaning of traces does not depend on the underlying function spaces (they agree on their intersection, when it is well-defined), and since we have the description $\partial\mathbb{R}^n_+=\mathbb{R}^{n-1}\times\{0\}$, we will refer to the trace of $u$ as $u_{|_{\partial \mathbb{R}^n_+}}$ instead of $\gamma_{0} u$.

\begin{lemma}\label{lem:IdentHsp0and0trace} Let $p_j\in(1,+\infty)$, $s\in(1/p_j,1+1/p_j)$, $j\in\{0,1\}$ such that $(\mathcal{C}_{s_0,p_0})$ is satisfied. For all $u\in [\dot{\mathrm{H}}^{s_0,p_0}\cap\dot{\mathrm{H}}^{s_1,p_1}](\mathbb{R}^n_+,\mathbb{C})$ such that $u_{|_{\partial\mathbb{R}^n_+}} =0$, the extension $\tilde{u}$ to $\mathbb{R}^n$ by $0$, satisfies
\begin{align*}
    \tilde{u}\in [\dot{\mathrm{H}}^{s_0,p_0}_0\cap\dot{\mathrm{H}}^{s_1,p_1}_0](\mathbb{R}^n_+,\mathbb{C})
\end{align*}
with the estimate
\begin{align*}
    \lVert \tilde{u} \rVert_{\dot{\mathrm{H}}^{s_j,p_j}(\mathbb{R}^n)} \lesssim_{p_j,s_j,n} \lVert {u} \rVert_{\dot{\mathrm{H}}^{s_j,p_j}(\mathbb{R}^n_+)} \,\text{, }\, j\in\{0,1\}\text{. }
\end{align*}

Hence, we have the following canonical isomorphism of Banach spaces
\begin{align*}
    \{\, u\in[\dot{\mathrm{H}}^{s_0,p_0}\cap\dot{\mathrm{H}}^{s_1,p_1}](\mathbb{R}^n_+,\mathbb{C}) \,|\, u_{|_{\partial\mathbb{R}^n_+}} =0 \,\} \simeq [\dot{\mathrm{H}}^{s_0,p_0}_0\cap\dot{\mathrm{H}}^{s_1,p_1}_0](\mathbb{R}^n_+,\mathbb{C}) \text{. }
\end{align*}

The result still holds replacing $\dot{\mathrm{H}}^{s_j,p_j}$ by $\dot{\mathrm{B}}^{s_j}_{p_j,q_j}$, $q_j\in[1,+\infty]$, $j\in\{0,1\}$ assuming that $(\mathcal{C}_{s_0,p_0,q_0})$ is satisfied.
\end{lemma}

\begin{proof} Let $u\in [\dot{\mathrm{H}}^{s_0,p_0}\cap\dot{\mathrm{H}}^{s_1,p_1}](\mathbb{R}^n_+,\mathbb{C})$ such that $u_{|_{\partial\mathbb{R}^n_+}} =0$, then for all $\phi\in [\dot{\mathrm{H}}^{1-s_j,p_j'}\cap\eus{S}](\mathbb{R}^n,\mathbb{C}^n)$, we have
\begin{align*}
    \int_{\mathbb{R}^n_+} \nabla u \cdot \phi = - \int_{\mathbb{R}^n_+} \div (\phi) \, u \text{. }
\end{align*}
So that introducing the extensions by $0$ to $\mathbb{R}^n$,  $\tilde{u}$ and $\widetilde{\nabla u}$,
\begin{align*}
    \int_{\mathbb{R}^n} \widetilde{\nabla u} \cdot \phi = \int_{\mathbb{R}^n_+} \nabla u \cdot \phi = - \int_{\mathbb{R}^n_+} \div (\phi) \, u = - \int_{\mathbb{R}^n} \div (\phi) \, \tilde{u} = \big\langle \nabla\tilde{ u} ,\phi \big\rangle_{\mathbb{R}^n}\text{. }
\end{align*}
Therefore, for all $\phi\in [\dot{\mathrm{H}}^{1-s_j,p_j'}\cap\eus{S}](\mathbb{R}^n,\mathbb{C}^n)$,
\begin{align*}
    \int_{\mathbb{R}^n} \widetilde{\nabla u} \cdot \phi= \big\langle \nabla\tilde{ u} ,\phi \big\rangle_{\mathbb{R}^n}\text{. }
\end{align*}
Hence $\widetilde{\nabla u} = \nabla\tilde{ u}$ in $\eus{S}'(\mathbb{R}^n,\mathbb{C}^n)$. Thus, by Propositions \ref{prop:dualityRieszpotential} and \ref{prop:SobolevMultiplier}, we deduce that
\begin{align*}
    \big\lvert\big\langle\nabla\tilde{ u} ,\phi \big\rangle _{\mathbb{R}^n}\big\rvert &\leqslant \lVert \phi \rVert_{\dot{\mathrm{H}}^{1-s_j,p_j'}(\mathbb{R}^n)}\lVert \widetilde{\nabla u} \rVert_{\dot{\mathrm{H}}^{s_j-1,p_j}(\mathbb{R}^n)}\\
    &\lesssim_{p_j,n,s_j} \lVert {\phi} \rVert_{\dot{\mathrm{H}}^{1-s_j,p_j'}(\mathbb{R}^n)}\lVert {\nabla u} \rVert_{\dot{\mathrm{H}}^{s_j-1,p_j}(\mathbb{R}^n_+)}\\
    &\lesssim_{p_j,n,s_j} \lVert {\phi} \rVert_{\dot{\mathrm{H}}^{1-s_j,p_j'}(\mathbb{R}^n)}\lVert {u} \rVert_{\dot{\mathrm{H}}^{s_j,p_j}(\mathbb{R}^n_+)}\text{. }
\end{align*}
One may conclude thanks to Proposition \ref{prop:dualityRieszpotential}, and Corollary \ref{cor:EqNormNablakmHspHaq}. The case of Besov spaces follows the same lines. The isomorphism is then a direct consequence.
\end{proof}

\subsection{Several remarks}

We finish this section giving several remarks. The goal of presenting here a definitive construction of homogeneous Sobolev and Besov spaces on the half-space is certainly not reached:

\begin{itemize}
    \item A related problem is that the extension operator we use is not general enough and disallow to recover too much negative index of regularity in case of homogeneous function spaces. It would be of interest to know if one can also recover non-complete positive index independently, without using intersection or density tricks. As mentioned at the beginning of this section, to know if one can construct an operator similar to Rychkov's extension operator, from \cite{Rychkov1999}, $\mathcal{E}$ such that $\mathcal{E}(\eus{S}'_h(\overline{\mathbb{R}^n_+}))\subset \eus{S}'_h(\mathbb{R}^n)$ with homogeneous estimates would be a sufficiently powerful result to overcome such troubles.
    \item Other definitions are possible for $\eus{S}'_h(\mathbb{R}^n)$. We have chosen here the one with the strongest convergence for the sum of low frequencies to continue the work started in \cite[Chapter~2]{bookBahouriCheminDanchin} and \cite[Chapter~3]{DanchinHieberMuchaTolk2020}. The choice of possible definitions and their functional analytic consequences on Besov spaces' construction are reviewed by Cobb in \cite[Appendix]{Cobb2021} and \cite{Cobb2022}. 
\end{itemize}

Not to further burden the actual presentation, we just mention that one could also investigate spaces such as
\begin{align*}
    \dot{\mathcal{B}}^{s}_{p,\infty}(\mathbb{R}^n_+)\text{ and }\dot{\mathcal{B}}^{s}_{p,\infty,0}(\mathbb{R}^n_+)\text{. }
\end{align*}
For those spaces, the space $\eus{S}_0(\overline{\mathbb{R}^n_+})$ is (strongly) dense in the first one by construction, and we can show that the space $\mathrm{C}_c^\infty(\mathbb{R}^n_+)$ is dense in the second one, and both may be recovered from interpolation of other appropriate homogeneous Sobolev and Besov spaces. We can also prove corresponding duality and traces results with values in $\dot{\mathcal{B}}^{s-1/p}_{p,\infty}(\mathbb{R}^{n-1})$ as presented in Theorem \eqref{thm:Tracehalfspace}. Details are left to the interested reader.


\section{Applications: the Dirichlet and Neumann Laplacians on the half-space}\label{Sec:DirNeuHalfspace}

\subsubsection*{Notations and Definitions : Sectorial operators, Operators on Sobolev and Besov spaces.}
We introduce domains for an operator $A$ acting on Sobolev or Besov spaces, denoting
\begin{itemize}
    \item $\mathrm{D}^{s}_{p}(A)$ (resp. $\dot{\mathrm{D}}^{s}_{p}(A)$)  its domain on $\mathrm{H}^{s,p}$ (resp. $\dot{\mathrm{H}}^{s,p}$);
    \item $\mathrm{D}^{s}_{p,q}(A)$ (resp. $\dot{\mathrm{D}}^{s}_{p,q}(A)$)  its domain on $\mathrm{B}^{s}_{p,q}$ (resp. $\dot{\mathrm{B}}^{s}_{p,q}$);
    \item $\mathrm{D}_{p}(A) = \mathrm{D}^{0}_{p}(A) = \dot{\mathrm{D}}^{0}_{p}(A)$ its domain on $\mathrm{L}^p$.
\end{itemize}

Similarly, $\mathrm{N}^{s}_{p}(A)$, $\mathrm{N}^{s}_{p,q}(A)$ will stand for its nullspace on $\mathrm{H}^{s,p}$ and $\mathrm{B}^{s}_{p,q}$, and range spaces will be given respectively by $\mathrm{R}^{s}_{p}(A)$ and $\mathrm{R}^{s}_{p,q}(A)$. We replace $\mathrm{N}$ and $\mathrm{R}$ by $\dot{\mathrm{N}}$ and $\dot{\mathrm{R}}$ for their corresponding sets on homogeneous function spaces. 

If the operator $A$ has different realizations depending on various function spaces and on the considered open set, we may write its domain $\mathrm{D}(A,\Omega)$, and similarly for its nullspace $\mathrm{N}$ and range space $\mathrm{R}$. We omit the open set $\Omega$ if there is no possible confusion.

We introduce the following subsets of the complex plane
\begin{align*}
    \Sigma_\mu &:=\{ \,z\in\mathbb{C}^\ast\,:\,\lvert\mathrm{arg}(z)\rvert<\mu\,\}\text{, if } \mu\in(0,\pi)\text{, }
\end{align*}
we also define $\Sigma_0 := (0,+\infty)$,  and later we are going to consider $\overline{\Sigma}_\mu$ its closure.

An operator $(\mathrm{D}(A),A)$ on complex valued Banach space $X$ is said to be $\omega$-\textit{\textbf{sectorial}}, if for a fixed $\omega\in (0,\pi)$, both conditions are satisfied
\begin{enumerate}
    \item $\sigma(A)\subset \overline{\Sigma}_\omega $, where $\sigma(A)$ denotes the spectrum of $A$ ;
    \item For all $\mu\in(\omega,\pi)$, $\sup_{\lambda\in \mathbb{C}\setminus\overline{\Sigma}_\mu}\lVert \lambda(\lambda \mathrm{I}-A)^{-1}\rVert_{X\rightarrow X} < +\infty$ .
\end{enumerate}

Sectorial operators is widely reviewed in several references but we mention here Haase's book \cite{bookHaase2006}. One may also check \cite[Chapter~3]{EgertPhDThesis2015}.

\subsubsection*{The Dirichlet and Neumann Laplacians.}

Before starting the analysis of the Dirichlet and Neumann Laplacians on the half-space, we introduce two appropriate extension operators. We denote $\mathrm{E}_\mathcal{J}$, for $\mathcal{J}\in\{\mathcal{D},\mathcal{N}\}$,  the extension operator defined for any measurable function $u$ on $\mathbb{R}^n_+$, for almost every $x=(x',x_n)\in\mathbb{R}^{n-1}\times\mathbb{R}_+$:
\begin{align*}
    \mathrm{E}_\mathcal{D}u (x',x_n) &:= \begin{cases}
  u(x',x_n)\,&\text{, if } (x',x_n)\in\mathbb{R}^{n-1}\times\mathbb{R}_+\text{, }\\    
  -u(x',-x_n)\,&\text{, if } (x',x_n)\in\mathbb{R}^{n-1}\times\mathbb{R}_-^*\text{ ; }
\end{cases}\\
    \mathrm{E}_\mathcal{N}u (x',x_n) &:= \begin{cases}
  u(x',x_n)\,&\text{, if } (x',x_n)\in\mathbb{R}^{n-1}\times\mathbb{R}_+\text{, }\\    
  u(x',-x_n)\,&\text{, if } (x',x_n)\in\mathbb{R}^{n-1}\times\mathbb{R}_-^* \text{. }
\end{cases}
\end{align*}

Obviously, for $\mathcal{J}\in\{\mathcal{D},\mathcal{N}\}$, $s\in (-1+1/p,1/p)$, $p\in(1,+\infty)$, the Proposition \ref{prop:SobolevMultiplier} leads to boundedness of
\begin{align}
    \mathrm{E}_\mathcal{J}\,:\, \dot{\mathrm{H}}^{s,p}(\mathbb{R}^n_+)\longrightarrow \dot{\mathrm{H}}^{s,p}(\mathbb{R}^n)  \text{. }
\end{align}
The same result holds replacing $\dot{\mathrm{H}}^{s,p}$ by either ${\mathrm{H}}^{s,p}$, ${\mathrm{B}}^{s}_{p,q}$, or even by $\dot{\mathrm{B}}^{s}_{p,q}$, $q\in[1,+\infty]$.

We are going to use the properties of Laplacian acting on the whole space to build the resolvent estimates for both the Dirichlet and the Neumann Laplacian. Usual Dirichlet and Neumann Laplacians are the operators $(\mathrm{D}(\Delta_\mathcal{J}),-\Delta_\mathcal{J})$, for $\mathcal{J}\in\{\mathcal{D},\mathcal{N}\}$, where the subscript $\mathcal{D}$ (resp. $\mathcal{N}$) stands for the Dirichlet (resp. Neumann) Laplacian, with, for $p\in(1,+\infty)$,
\begin{align*}
    \mathrm{D}_p(\Delta_\mathcal{D}) &:= \left\{\,u\in \mathrm{H}^{1,p}(\mathbb{R}^n_+,\mathbb{C})\,\Big{|}\, \Delta u \in \mathrm{L}^{p}(\mathbb{R}_+^n,\mathbb{C})\text{ and } u_{|_{\partial\mathbb{R}^n_+}}=0 \,\right\}\text{, }\\
    \mathrm{D}_p(\Delta_\mathcal{N}) &:= \left\{\,u\in \mathrm{H}^{1,p}(\mathbb{R}^n_+,\mathbb{C})\,\Big{|}\, \Delta u \in \mathrm{L}^{p}(\mathbb{R}_+^n,\mathbb{C})\text{ and } \partial_{\nu}u_{|_{\partial\mathbb{R}^n_+}}=0 \,\right\}\text{. }
\end{align*}
For $\mathcal{J}\in\{\mathcal{D},\mathcal{N}\}$, and all $u\in \mathrm{D}_p(\Delta_{\mathcal{J}})$,
\begin{align*}
    -\Delta_\mathcal{J}u:=-\Delta u\text{. }
\end{align*}

When $p=2$, one can also realize both Dirichlet and Neumann Laplacians by the mean of densely defined, symmetric, accretive, continuous, closed,  sesquilinear forms on $\mathrm{L}^2(\mathbb{R}^n_+,\mathbb{C})$, for $\mathcal{J}\in\{\mathcal{D},\mathcal{N}\}$,
\begin{align}\label{eq:sesqlinformLaplacianDirNeu}
    \mathfrak{a}_\mathcal{J}\,:\, \mathrm{D}_2(\mathfrak{a}_\mathcal{J})^2\ni(u,v) \longmapsto \int_{\mathbb{R}^n_+}\nabla u \cdot \overline{\nabla v}
\end{align}
with $\mathrm{D}_2(\mathfrak{a}_\mathcal{D})= \mathrm{H}^{1}_0(\mathbb{R}^n_+,\mathbb{C})$, $\mathrm{D}_2(\mathfrak{a}_\mathcal{N})= \mathrm{H}^{1}(\mathbb{R}^n_+,\mathbb{C})$, so that it is easy to see, and well-known, that both, the Neumann and Dirichlet Laplacians, are closed, densely defined, non-negative self-adjoint operators on $\mathrm{L}^2(\mathbb{R}^n_+,\mathbb{C})$, see \cite[Chapter~1,~Section~1.2]{bookOuhabaz2005}. We can be even more precise.

\begin{proposition}\label{prop:L2DirNEuLapl} Provided $\mathcal{J}\in\{\mathcal{D},\mathcal{N}\}$, the operator $(\mathrm{D}_2(\Delta_\mathcal{J}),-\Delta_\mathcal{J})$ is an injective non-negative self-adjoint and $0$-sectorial operator on $\mathrm{L}^2(\mathbb{R}^n_+,\mathbb{C})$.

Moreover, the following hold
\begin{enumerate}
    \item $\mathrm{D}_2(\Delta_\mathcal{J})$ is a closed subspace of $\mathrm{H}^2(\mathbb{R}^n_+,\mathbb{C})$;
    \item Provided $\mu\in [0,\pi)$, for $\lambda\in\Sigma_\mu$, $f\in\mathrm{L}^2(\mathbb{R}^n_+,\mathbb{C})$, then $u:=(\lambda\mathrm{I}-\Delta_\mathcal{J})^{-1}f $ satisfies
    \begin{align*}
        \lvert\lambda\rvert\lVert  u\rVert_{{\mathrm{L}}^{2}(\mathbb{R}^n_+)}+\lvert\lambda\rvert^\frac{1}{2}\lVert \nabla u\rVert_{{\mathrm{L}}^{2}(\mathbb{R}^n_+)}+\lVert \nabla^2 u\rVert_{{\mathrm{L}}^{2}(\mathbb{R}^n_+)} &\lesssim_{n,\mu} \lVert f\rVert_{{\mathrm{L}}^{2}(\mathbb{R}^n_+)} \text{ ; }
    \end{align*}
    \item The following resolvent identity holds for all $\mu\in [0,\pi)$, $\lambda\in\Sigma_\mu$, $f\in\mathrm{L}^2(\mathbb{R}^n_+,\mathbb{C})$,
    \begin{align*}
        \mathrm{E}_\mathcal{J}(\lambda\mathrm{I}-\Delta_\mathcal{J})^{-1}f = (\lambda\mathrm{I}-\Delta)^{-1}\mathrm{E}_\mathcal{J} f\text{. }
    \end{align*}
\end{enumerate}
\end{proposition}

\begin{remark}\label{rem:commutPartialExtOpDirNeu} For $u\,:\,\mathbb{R}^n_+\longrightarrow \mathbb{C}$, we set
\begin{align*}
    \tilde{u}_{\mathcal{J}}:= [\mathrm{E}_{\mathcal{J}} u]_{|_{\mathbb{R}^n_{-}}}
\end{align*}
for $\mathcal{J}\in\{\mathcal{D},\mathcal{N}\}$. We notice that in $\eus{D}'(\mathbb{R}^n_-,\mathbb{C})$,
\begin{align*}
    \partial_{x_n} [\tilde{u}_{\mathcal{N}}] = \widetilde{[\partial_{x_n} u]}_{\mathcal{D}} \text{ and  } \partial_{x_n} [\tilde{u}_{\mathcal{D}}] = \widetilde{[\partial_{x_n} u]}_{\mathcal{N}} \text{. }
\end{align*}
\end{remark}

\begin{proof}One may use self-adjointness and \eqref{eq:sesqlinformLaplacianDirNeu} which gives, by standard Hilbertian theory, the following resolvent estimate
\begin{align*}
        \lvert\lambda\rvert\lVert  u\rVert_{{\mathrm{L}}^{2}(\mathbb{R}^n_+)}+\lvert\lambda\rvert^\frac{1}{2}\lVert \nabla u\rVert_{{\mathrm{L}}^{2}(\mathbb{R}^n_+)}+\lVert \Delta u\rVert_{{\mathrm{L}}^{2}(\mathbb{R}^n_+)} &\lesssim_{\mu} \lVert f\rVert_{{\mathrm{L}}^{2}(\mathbb{R}^n_+)} \text{, }
\end{align*}
where $u:=(\lambda\mathrm{I}-\Delta_\mathcal{J})^{-1}f$, $f\in\mathrm{L}^2(\mathbb{R}^n_+,\mathbb{C})$, $\lambda \in\Sigma_\mu$, $\mu\in[0,\pi)$.

Now, for fixed  $f\in\mathrm{L}^2(\mathbb{R}^n_+,\mathbb{C})$, $\lambda \in\Sigma_\mu$, $\mu\in[0,\pi)$, we consider $u:=(\lambda\mathrm{I}-\Delta_\mathcal{J})^{-1}f$. Assuming $\mathcal{J}=\mathcal{N}$, we have for $\phi\in\eus{S}(\mathbb{R}^n,\mathbb{C})$,
\begin{align*}
    \big\langle \mathrm{E}_{\mathcal{N}}u, -\Delta \phi \big\rangle_{\mathbb{R}^n} &= \big\langle u, -\Delta \phi \big\rangle_{\mathbb{R}^n_+}  +  \big\langle \tilde{u}_{\mathcal{N}}, -\Delta \phi \big\rangle_{\mathbb{R}^n_{-}}\\
    &= \big\langle \nabla u, \nabla \phi \big\rangle_{\mathbb{R}^n_+} + \big\langle u, \nabla \phi \cdot\mathfrak{e_n} \big\rangle_{\partial\mathbb{R}^n_+} - \big\langle \tilde{u}_{\mathcal{N}}, \nabla \phi \cdot\mathfrak{e_n} \big\rangle_{\partial\mathbb{R}^n_-} \\
    &\quad + \big\langle \widetilde{[\nabla' u]}_{\mathcal{N}}, \nabla' \phi  \big\rangle_{\mathbb{R}^n_-} + \big\langle \widetilde{[\partial_{x_n} u]}_{\mathcal{D}}, \partial_{x_n} \phi  \big\rangle_{\mathbb{R}^n_-}
\end{align*}
Since $\partial \mathbb{R}^n_+ = \partial \mathbb{R}^n_-=\mathbb{R}^{n-1}\times\{0\}$, with traces $\tilde{u}_{\mathcal{N}|_{\partial \mathbb{R}^n_-}}={u}_{|_{\partial \mathbb{R}^n_+}}$, we deduce $\big\langle u_{|_{\partial \mathbb{R}^n_+}}, \nabla \phi \cdot\mathfrak{e_n} \big\rangle_{\partial\mathbb{R}^n_+} - \big\langle \tilde{u}_{\mathcal{N}|_{\partial \mathbb{R}^n_-}}, \nabla \phi \cdot\mathfrak{e_n} \big\rangle_{\partial\mathbb{R}^n_-}=0$. Then, thanks to Remark \ref{rem:commutPartialExtOpDirNeu} and the boundary condition on $u$, \textit{i.e.} $\partial_{x_n}u_{|_{\partial \mathbb{R}^n_+}}=0$, we have
\begin{align*}
    \big\langle \mathrm{E}_{\mathcal{N}}u, -\Delta \phi \big\rangle_{\mathbb{R}^n} &= \big\langle \nabla u, \nabla \phi \big\rangle_{\mathbb{R}^n_+} + \big\langle \widetilde{[\nabla' u]}_{\mathcal{N}}, \nabla' \phi  \big\rangle_{\mathbb{R}^n_-} + \big\langle \widetilde{[\partial_{x_n} u]}_{\mathcal{D}}, \partial_{x_n} \phi  \big\rangle_{\mathbb{R}^n_-}\\
    &= \big\langle -\Delta u, \phi \big\rangle_{\mathbb{R}^n_+} + \big\langle \widetilde{[-\Delta' u]}_{\mathcal{N}}, \phi \big\rangle_{\mathbb{R}^n_- } + \big\langle \widetilde{[-\partial_{x_n}^2 u]}_{\mathcal{N}}, \phi \big\rangle_{\mathbb{R}^n_- } \\
    &\qquad - \big\langle \partial_{x_n} u, \phi \big\rangle_{\partial\mathbb{R}^n_+} - \big\langle \widetilde{[\partial_{x_n} u]}_{\mathcal{D}}, \phi \big\rangle_{\partial\mathbb{R}^n_{-}} \\
    &= \big\langle \mathrm{E}_{\mathcal{N}}[-\Delta u],  \phi \big\rangle_{\mathbb{R}^n} \text{. }
\end{align*}
Thus, $-\Delta\mathrm{E}_{\mathcal{N}}u = \mathrm{E}_{\mathcal{N}}[-\Delta u]$ in $\eus{S}'(\mathbb{R}^n,\mathbb{C})$. One may reproduce the above calculations for $\mathcal{J}=\mathcal{D}$. So for $\mathcal{J}\in\{\mathcal{D},\mathcal{N}\}$, $\mathrm{E}_{\mathcal{J}} u$ is a solution of
\begin{align*}
    \lambda U - \Delta U = \mathrm{E}_\mathcal{J}f\text{. }
\end{align*}
We have $\mathrm{E}_\mathcal{J}f\in\mathrm{L}^2(\mathbb{R}^n,\mathbb{C})$. By uniqueness of  the solution provided in $\mathbb{R}^n$, we necessarily have $U=\mathrm{E}_{\mathcal{J}} u$, which can be written as
\begin{align*}
        \mathrm{E}_\mathcal{J}(\lambda\mathrm{I}-\Delta_\mathcal{J})^{-1}f = (\lambda\mathrm{I}-\Delta)^{-1}\mathrm{E}_\mathcal{J} f\text{. }
\end{align*}
Thus one deduces point \textit{(iii)}, from the definition of function spaces by restriction, \textit{(ii)} follows, and finally setting $\lambda=1$ in point \textit{(ii)} yields \textit{(i)}.
\end{proof}

We want to show some sharp regularity results on the Dirichlet and Neumann resolvent problems, on the scale of inhomogeneous and homogeneous Sobolev and Besov spaces. To do so, we introduce their corresponding domains on each space. Provided $p\in(1,+\infty)$ $s\in(-1+1/p,1+1/p)$, if is satisfied $(\mathcal{C}_{s,p})$:

\begin{align*}
    \dot{\mathrm{D}}^s_p(\Delta_\mathcal{D}) &:= \left\{\,u\in [\dot{\mathrm{H}}^{s,p}_0\cap\dot{\mathrm{H}}^{s+1,p}](\mathbb{R}_+^n,\mathbb{C})\,\Big{|}\, \Delta u \in \dot{\mathrm{H}}^{s,p}_0(\mathbb{R}_+^n,\mathbb{C}) \text{ and } u_{|_{\partial\mathbb{R}^n_+}}=0 \,\right\}\subset \dot{\mathrm{H}}^{s,p}_0(\mathbb{R}_+^n,\mathbb{C})\text{, }\\
    \dot{\mathrm{D}}^s_p(\Delta_\mathcal{N}) &:= \left\{\,u\in [\dot{\mathrm{H}}^{s,p}\cap\dot{\mathrm{H}}^{s+1,p}](\mathbb{R}_+^n,\mathbb{C})\,\Big{|}\, \Delta u \in \dot{\mathrm{H}}^{s,p}(\mathbb{R}_+^n,\mathbb{C})\text{ and } \partial_{\nu}u_{|_{\partial\mathbb{R}^n_+}}=0 \,\right\}\subset \dot{\mathrm{H}}^{s,p}(\mathbb{R}_+^n,\mathbb{C})\text{. }
\end{align*}

We can also consider their domains on inhomogeneous Sobolev and Besov spaces, as well as homogeneous spaces, replacing $(\dot{\mathrm{D}}^s_p,\dot{\mathrm{H}}^{s,p})$ by either $({\mathrm{D}}^s_p,{\mathrm{H}}^{s,p})$, $({\mathrm{D}}^s_{p,q},{\mathrm{B}}^{s}_{p,q})$ and finally $(\dot{\mathrm{D}}^s_{p,q},\dot{\mathrm{B}}^{s}_{p,q})$ provided $q\in[1,+\infty]$, and \eqref{AssumptionCompletenessExponents} is satisfied. Certainly, the conditions $(\mathcal{C}_{s,p})$ and \eqref{AssumptionCompletenessExponents} are no longer necessary when one only deals with inhomogeneous function spaces.

\begin{remark} We allowed us a slight abuse of notation here: we identified $\dot{\mathrm{H}}^{s,p}_0(\mathbb{R}^n_+)$ with either
\begin{itemize}
    \item $\dot{\mathrm{H}}^{s,p}(\mathbb{R}^n_+)$ when $s\in(-1+1/p,1/p)$, thanks to Proposition~\ref{prop:SobolevMultiplier};
    \item $\dot{\mathrm{H}}^{s,p}(\mathbb{R}^n_+)$ with homogeneous Dirichlet boundary condition when $s\in(1/p,1+1/p)$, thanks to Lemma~\ref{lem:IdentHsp0and0trace}.
\end{itemize}
The same identification is made for Besov spaces, and inhomogeneous function spaces.
\end{remark}

It is then not difficult to see that the Dirichlet and Neumann Laplacians are well-defined unbounded closed linear operators, densely defined, if $q\in[1,+\infty)$ in the case of inhomogeneous and homogeneous Besov spaces. If $q=+\infty$, the domain of the Dirichlet (resp. Neumann) Laplacian is only known to be weak${}^\ast$ dense in ${\mathrm{B}}^{s}_{p,\infty,0}$ (resp. in ${\mathrm{B}}^{s}_{p,\infty}$) and $\dot{\mathrm{B}}^{s}_{p,\infty,0}$ (resp. $\dot{\mathrm{B}}^{s}_{p,\infty}$).

\begin{proposition}\label{prop:DirResolvPbHspBspqRn+} Let $p,\tilde{p}\in(1,+\infty)$, $q,\tilde{q}\in[1,+\infty]$, $s\in(-1+\frac{1}{p},1+\frac{1}{p})$, $s\neq 1/p$, $\alpha\in(-1+\frac{1}{\tilde{p}},1+\frac{1}{\tilde{p}})$, $\alpha\neq1/\tilde{p}$, and $\lambda\in\Sigma_\mu$ provided $\mu\in[0,\pi)$. We assume that $(\mathcal{C}_{s,p})$, and we let  $f\in\dot{\mathrm{H}}^{s,p}_0(\mathbb{R}^n_+,\mathbb{C})$. Let us consider the \textbf{resolvent Dirichlet} problem with homogeneous boundary condition:
\begin{align}\tag{$\mathcal{DL}_{\lambda}$}\label{ResolvDirLap}
\left\{\begin{array}{rrl}
        \lambda u -\Delta u   &=  f\text{,} \,&\text{ in } \mathbb{R}^n_+ \text{, }\\
         u_{|_{\partial\mathbb{R}^n_+}}  &= 0\text{,} \,&\text{ on } \partial\mathbb{R}^n_+ \text{. }
\end{array}
\right.
\end{align}
The problem \eqref{ResolvDirLap} admits a unique solution $u\in [\dot{\mathrm{H}}^{s,p}_0\cap\dot{\mathrm{H}}^{s+2,p}](\mathbb{R}^n_+,\mathbb{C})$ with the estimate
    \begin{align*}
        \lvert\lambda\rvert\lVert  u\rVert_{\dot{\mathrm{H}}^{s,p}(\mathbb{R}^n_+)}+\lvert\lambda\rvert^\frac{1}{2}\lVert \nabla u\rVert_{\dot{\mathrm{H}}^{s,p}(\mathbb{R}^n_+)}+\lVert \nabla^2 u\rVert_{\dot{\mathrm{H}}^{s,p}(\mathbb{R}^n_+)} &\lesssim_{p,n,s,\mu} \lVert f\rVert_{\dot{\mathrm{H}}^{s,p}(\mathbb{R}^n_+)} \text{. }
    \end{align*}
    
    If moreover $\alpha\neq 1/\tilde{p}$ and $f\in\dot{\mathrm{H}}^{\alpha,\tilde{p}}_0(\mathbb{R}^n_+,\mathbb{C})$, then we also have $u\in [\dot{\mathrm{H}}^{\alpha,\tilde{p}}\cap\dot{\mathrm{H}}^{\alpha+2,\tilde{p}}](\mathbb{R}^n_+,\mathbb{C})$ with the corresponding estimate
    \begin{align*}
        \lvert\lambda\rvert\lVert  u\rVert_{\dot{\mathrm{H}}^{\alpha,\tilde{p}}(\mathbb{R}^n_+)}+\lvert\lambda\rvert^\frac{1}{2}\lVert \nabla u\rVert_{\dot{\mathrm{H}}^{\alpha,\tilde{p}}(\mathbb{R}^n_+)}+\lVert \nabla^2 u\rVert_{\dot{\mathrm{H}}^{\alpha,\tilde{p}}(\mathbb{R}^n_+)} &\lesssim_{\tilde{p},n,\alpha,\mu} \lVert f\rVert_{\dot{\mathrm{H}}^{\alpha,\tilde{p}}(\mathbb{R}^n_+)} \text{. }
    \end{align*}
    
    The result still holds replacing $(\dot{\mathrm{H}}^{s,p},\dot{\mathrm{H}}^{s+2,p},\dot{\mathrm{H}}^{\alpha,\tilde{p}},\dot{\mathrm{H}}^{\alpha+2,\tilde{p}})$ by $(\dot{\mathrm{B}}^{s}_{p,q},\dot{\mathrm{B}}^{s+2}_{p,q},\dot{\mathrm{B}}^{\alpha}_{\tilde{p},\tilde{q}},\dot{\mathrm{B}}^{\alpha+2}_{\tilde{p},\tilde{q}})$ whenever \eqref{AssumptionCompletenessExponents} is satisfied.

    The whole result still holds for inhomogeneous function spaces, replacing $(\dot{\mathrm{H}}, \dot{\mathrm{B}})$ by $({\mathrm{H}}, {\mathrm{B}})$. However, in this case, the consideration of intersection spaces and the conditions $(\mathcal{C}_{s,p})$ and \eqref{AssumptionCompletenessExponents} are no longer necessary.
\end{proposition}

\begin{remark} $\bullet$ For this specific Proposition \ref{prop:DirResolvPbHspBspqRn+}, we have excluded the cases $s=1/p$ and $\alpha= 1/\tilde{p}$. Both require to introduce, e.g. in case of Sobolev spaces, the homogeneous counterpart of the Lions-Magenes Sobolev space $\dot{\mathrm{H}}^{1/q,q}_{00}(\mathbb{R}^n_+)$, $q\in\{p,\tilde{p}\}$. See for instance \cite[Chapter~1,~Theorem~11.7]{LionsMagenes1972} for the inhomogeneous space in the case $q=2$.

$\bullet$ We bring to the attention of the reader that $(\mathcal{C}_{\alpha,\tilde{p}})$ is \textbf{NEVER} assumed, only $(\mathcal{C}_{s,p})$ is. This is in order to echo the principle of decoupled estimates in intersection spaces when one wants to deal with higher regularities involving some non-complete spaces. All the other results presented below follow the same principle.
\end{remark}

\begin{proof} Provided $p\in(1,+\infty)$, and firstly that $s\in(-1+1/p,1/p)$, for $f\in\dot{\mathrm{H}}^{s,p}(\mathbb{R}^n_+,\mathbb{C})$, 
it follows from Proposition \ref{prop:SobolevMultiplier} that for $U:=(\lambda\mathrm{I}-\Delta)^{-1}\mathrm{E}_{\mathcal{D}}f$
\begin{align*}
        \lvert\lambda\rvert\lVert U \rVert_{\dot{\mathrm{H}}^{s,p}(\mathbb{R}^n)}+\lvert\lambda\rvert^\frac{1}{2}\lVert  \nabla U\rVert_{\dot{\mathrm{H}}^{s,p}(\mathbb{R}^n)}+\lVert \nabla^2  U\rVert_{\dot{\mathrm{H}}^{s,p}(\mathbb{R}^n)} &\lesssim_{p,n,s,\mu} \lVert f\rVert_{\dot{\mathrm{H}}^{s,p}(\mathbb{R}^n_+)} \text{. }
\end{align*}
Thus, by the definition of function spaces by restriction, we set $u:=U_{|_{\mathbb{R}^n_+}}$ which satisfies
\begin{align*}
        \lvert\lambda\rvert\lVert u \rVert_{\dot{\mathrm{H}}^{s,p}(\mathbb{R}^n_+)}+\lvert\lambda\rvert^\frac{1}{2}\lVert  \nabla u\rVert_{\dot{\mathrm{H}}^{s,p}(\mathbb{R}^n_+)}+\lVert \nabla^2  u\rVert_{\dot{\mathrm{H}}^{s,p}(\mathbb{R}^n_+)} &\lesssim_{p,n,s,\mu} \lVert f\rVert_{\dot{\mathrm{H}}^{s,p}(\mathbb{R}^n_+)} \text{, }
\end{align*}
then the map $f\mapsto [(\lambda\mathrm{I}-\Delta)^{-1}\mathrm{E}_{\mathcal{D}}f]_{|_{\mathbb{R}^n_+}}$ is a bounded map on $\dot{\mathrm{H}}^{s,p}(\mathbb{R}^n_+,\mathbb{C})$. Everything goes similarly for ${\mathrm{H}}^{s,p}(\mathbb{R}^n_+,\mathbb{C})$. One may check, as in the proof of Proposition \ref{prop:L2DirNEuLapl}, and by a limiting argument, given the density of $[\mathrm{L}^2\cap\dot{\mathrm{H}}^{s,p}](\mathbb{R}^n_+,\mathbb{C})$ in $\dot{\mathrm{H}}^{s,p}(\mathbb{R}^n_+,\mathbb{C})$, that $u_{|_{\partial \mathbb{R}^n_+}} = 0 $, and
\begin{align*}
    \lambda u - \Delta u = f \,\text{ in } \mathbb{R}^n_+ \text{. }
\end{align*}
Again, as in the proof of Proposition \ref{prop:L2DirNEuLapl}, one may check that any solution $u$ to the above resolvent Dirichlet problem necessarily satisfies $\mathrm{E}_\mathcal{D}u=(\lambda\mathrm{I}-\Delta)^{-1}\mathrm{E}_{\mathcal{D}}f$.

Now if $s\in(1/p,1+1/p)$, $f\in [\dot{\mathrm{H}}^{s-1,p}_0\cap\dot{\mathrm{H}}^{s,p}_0](\mathbb{R}^n_+,\mathbb{C})$ then we have, thanks to previous considerations, $U:=(\lambda\mathrm{I}-\Delta)^{-1}\mathrm{E}_{\mathcal{D}}f\in \dot{\mathrm{H}}^{s-1,p}(\mathbb{R}^n,\mathbb{C})$. It suffices to show that $U\in \dot{\mathrm{H}}^{s,p}(\mathbb{R}^n,\mathbb{C})$, which is true. Indeed, we have
\begin{align*}
        \lvert\lambda\rvert\lVert U \rVert_{\dot{\mathrm{H}}^{s,p}(\mathbb{R}^n)} &\lesssim_{s,p,n,\mu} \vert\lambda\rvert\lVert \nabla U \rVert_{\dot{\mathrm{H}}^{s-1,p}(\mathbb{R}^n)}\\
        &\lesssim_{s,p,n,\mu} \lVert \nabla \mathrm{E}_\mathcal{D} f \rVert_{\dot{\mathrm{H}}^{s-1,p}(\mathbb{R}^n)}\\
         &\lesssim_{s,p,n,\mu} \sum_{k=1}^n\lVert \partial_{x_k} \mathrm{E}_\mathcal{D} f \rVert_{\dot{\mathrm{H}}^{s-1,p}(\mathbb{R}^n)} \text{. }
\end{align*}
Since equalities $\partial_{x_k} \mathrm{E}_\mathcal{D} f =  \mathrm{E}_\mathcal{D} \partial_{x_k} f$, $k\in\llb 1,n-1\rrb$ and $\partial_{x_n} \mathrm{E}_\mathcal{D} f =  \mathrm{E}_\mathcal{N} \partial_{x_n} f$ occur in $\eus{S}'(\mathbb{R}^n,\mathbb{C})$, we deduce
\begin{align*}
        \lvert\lambda\rvert\lVert u \rVert_{\dot{\mathrm{H}}^{s,p}(\mathbb{R}^n_+)}\leqslant\lvert\lambda\rvert\lVert U \rVert_{\dot{\mathrm{H}}^{s,p}(\mathbb{R}^n)} \lesssim_{s,p,n,\mu} \lVert  f \rVert_{\dot{\mathrm{H}}^{s,p}(\mathbb{R}^n_+)} \text{. }
\end{align*}
One may proceed similarly as before to obtain the full estimate
\begin{align*}
        \lvert\lambda\rvert\lVert u \rVert_{\dot{\mathrm{H}}^{s,p}(\mathbb{R}^n_+)}+\lvert\lambda\rvert^\frac{1}{2}\lVert \nabla  u\rVert_{\dot{\mathrm{H}}^{s,p}(\mathbb{R}^n_+)}+\lVert \nabla^2  u\rVert_{\dot{\mathrm{H}}^{s,p}(\mathbb{R}^n_+)} &\lesssim_{p,n,s,\mu} \lVert f\rVert_{\dot{\mathrm{H}}^{s,p}(\mathbb{R}^n_+)} \text{. }
\end{align*}
Thus the estimates still hold by density for all $f\in \dot{\mathrm{H}}^{s,p}_0(\mathbb{R}^n_+)$, $s\in(-1+1/p,1+1/p)$, $s\neq 1/p$, whenever $(\mathcal{C}_{s,p})$ is satisfied.

The $\dot{\mathrm{H}}^{\alpha,\tilde{p}}$-estimate for $f\in[\dot{\mathrm{H}}^{s,p}_0\cap \dot{\mathrm{H}}^{\alpha,\tilde{p}}_0](\mathbb{R}^n_+)$ can be obtained the same way, whenever $(\mathcal{C}_{s,p})$ is satisfied.

The case of Besov spaces $\dot{\mathrm{B}}^{s}_{p,q,0}$ can be achieved via similar argument for $q<+\infty$, the case $q=+\infty$ is obtained via real interpolation. The case of the $\dot{\mathrm{B}}^{\alpha}_{\tilde{p},\tilde{q},0}$-estimate for $f\in\dot{\mathrm{B}}^{s}_{p,q,0}\cap\dot{\mathrm{B}}^{\alpha}_{\tilde{p},\tilde{q},0}$ can be done as above.
\end{proof}

The proof for the Neumann resolvent problem in the proposition below is fairly similar to the proof of Proposition \ref{prop:DirResolvPbHspBspqRn+}, a complex interpolation argument allows values $s=1/p$ and $\alpha= 1/\tilde{p}$.

\begin{proposition}\label{prop:NeuResolvPbHspBspqRn+} Let $p,\tilde{p}\in(1,+\infty)$, $q,\tilde{q}\in[1,+\infty]$, $s\in(-1+\frac{1}{p},1+\frac{1}{p})$, $\alpha\in(-1+\frac{1}{\tilde{p}},1+\frac{1}{\tilde{p}})$ and $\lambda\in\Sigma_\mu$ provided $\mu\in[0,\pi)$. We assume that $(\mathcal{C}_{s,p})$, and we let $f\in\dot{\mathrm{H}}^{s,p}(\mathbb{R}^n_+,\mathbb{C})$. Let us consider the \textbf{resolvent Neumann} problem with homogeneous boundary condition:
\begin{align}\tag{$\mathcal{NL}_{\lambda}$}\label{ResolvNeumannLap}
\left\{\begin{array}{rrl}
        \lambda u -\Delta u   &=  f\text{,} \,&\text{ in } \mathbb{R}^n_+ \text{, }\\
        \partial_\nu u_{|_{\partial\mathbb{R}^n_+}}  &= 0\text{,} \,&\text{ on } \partial\mathbb{R}^n_+ \text{. }
\end{array}
\right.
\end{align}
The problem \eqref{ResolvNeumannLap} admits a unique solution $u\in [\dot{\mathrm{H}}^{s,p}\cap\dot{\mathrm{H}}^{s+2,p}](\mathbb{R}^n_+,\mathbb{C})$ with the estimate
    \begin{align*}
        \lvert\lambda\rvert\lVert  u\rVert_{\dot{\mathrm{H}}^{s,p}(\mathbb{R}^n_+)}+\lvert\lambda\rvert^\frac{1}{2}\lVert \nabla u\rVert_{\dot{\mathrm{H}}^{s,p}(\mathbb{R}^n_+)}+\lVert \nabla^2 u\rVert_{\dot{\mathrm{H}}^{s,p}(\mathbb{R}^n_+)} &\lesssim_{p,n,s,\mu} \lVert f\rVert_{\dot{\mathrm{H}}^{s,p}(\mathbb{R}^n_+)} \text{. }
    \end{align*}
    
    If moreover $f\in\dot{\mathrm{H}}^{\alpha,\tilde{p}}(\mathbb{R}^n_+,\mathbb{C})$, then we also have $u\in [\dot{\mathrm{H}}^{\alpha,\tilde{p}}\cap\dot{\mathrm{H}}^{\alpha+2,\tilde{p}}](\mathbb{R}^n_+,\mathbb{C})$ with corresponding the estimate
    \begin{align*}
        \lvert\lambda\rvert\lVert  u\rVert_{\dot{\mathrm{H}}^{\alpha,\tilde{p}}(\mathbb{R}^n_+)}+\lvert\lambda\rvert^\frac{1}{2}\lVert \nabla u\rVert_{\dot{\mathrm{H}}^{\alpha,\tilde{p}}(\mathbb{R}^n_+)}+\lVert \nabla^2 u\rVert_{\dot{\mathrm{H}}^{\alpha,\tilde{p}}(\mathbb{R}^n_+)} &\lesssim_{\tilde{p},n,\alpha,\mu} \lVert f\rVert_{\dot{\mathrm{H}}^{\alpha,\tilde{p}}(\mathbb{R}^n_+)} \text{. }
    \end{align*}
    
    The result still holds replacing $(\dot{\mathrm{H}}^{s,p},\dot{\mathrm{H}}^{s+2,p},\dot{\mathrm{H}}^{\alpha,\tilde{p}},\dot{\mathrm{H}}^{\alpha+2,\tilde{p}})$ by $(\dot{\mathrm{B}}^{s}_{p,q},\dot{\mathrm{B}}^{s+2}_{p,q},\dot{\mathrm{B}}^{\alpha}_{\tilde{p},\tilde{q}},\dot{\mathrm{B}}^{\alpha+2}_{\tilde{p},\tilde{q}})$ whenever \eqref{AssumptionCompletenessExponents} is satisfied.

    The whole result still holds for inhomogeneous function spaces, replacing $(\dot{\mathrm{H}}, \dot{\mathrm{B}})$ by $({\mathrm{H}}, {\mathrm{B}})$. However, in this case, the consideration of intersection spaces and the conditions $(\mathcal{C}_{s,p})$ and \eqref{AssumptionCompletenessExponents} are no longer necessary.
\end{proposition}

\begin{proposition}\label{prop:DirPbRn+} Let $p,\tilde{p}\in(1,+\infty)$, $q,\tilde{q}\in[1,+\infty]$, $s\in(-1+\frac{1}{p},+\infty)$, $\alpha\in(-1+\frac{1}{\tilde{p}},+\infty)$ such that $(\mathcal{C}_{s+2,p})$ is satisfied.
For $f\in\dot{\mathrm{H}}^{s,p}(\mathbb{R}^n_+,\mathbb{C})$, $g\in \dot{\mathrm{B}}^{s+2-\frac{1}{p}}_{p,p}(\mathbb{R}^{n-1},\mathbb{C})$, let us consider the \textbf{Dirichlet} problem with inhomogeneous boundary condition:
\begin{align}\tag{$\mathcal{DL}_{0}$}\label{DirLap}
\left\{\begin{array}{rrl}
         -\Delta u   &=  f\text{,} \,&\text{ in } \mathbb{R}^n_+ \text{, }\\
         u_{|_{\partial\mathbb{R}^n_+}}  &= g\text{,} \,&\text{ on } \partial\mathbb{R}^n_+ \text{. }
\end{array}
\right.
\end{align}

The problem \eqref{DirLap} admits a unique solution $u$ such that
\begin{align*}
    u\in\dot{\mathrm{H}}^{s+2,p}(\mathbb{R}^n_+,\mathbb{C})\subset \mathrm{C}^0_{0,x_n}(\overline{\mathbb{R}_+},\dot{\mathrm{B}}^{s+2-\frac{1}{p}}_{p,p}(\mathbb{R}^{n-1},\mathbb{C}))
\end{align*}
with the estimate
    \begin{align*}
        \lVert u\rVert_{\mathrm{L}^\infty(\mathbb{R}_+,\dot{\mathrm{B}}^{s+2-\frac{1}{p}}_{p,p}(\mathbb{R}^{n-1}))} \lesssim_{s,p,n}\lVert \nabla^2 u\rVert_{\dot{\mathrm{H}}^{s,p}(\mathbb{R}^n_+)} &\lesssim_{p,n,s} \lVert f\rVert_{\dot{\mathrm{H}}^{s,p}(\mathbb{R}^n_+)} + \lVert g\rVert_{\dot{\mathrm{B}}^{s+2-\frac{1}{p}}_{p,p}(\mathbb{R}^{n-1})}  \text{. }
    \end{align*}
    
    If moreover $f\in\dot{\mathrm{H}}^{\alpha,\tilde{p}}(\mathbb{R}^n_+,\mathbb{C})$ and $g\in \dot{\mathrm{B}}^{\alpha+2-\frac{1}{\tilde{p}}}_{\tilde{p},\tilde{p}}(\mathbb{R}^{n-1},\mathbb{C})$ then the solution $u$ also satisfies $ u\in \dot{\mathrm{H}}^{\alpha+2,\tilde{p}}(\mathbb{R}^n_+,\mathbb{C})$ with the corresponding estimate
    \begin{align*}
        \lVert \nabla^2 u\rVert_{\dot{\mathrm{H}}^{\alpha,\tilde{p}}(\mathbb{R}^n_+)} &\lesssim_{\tilde{p},n,\alpha} \lVert f\rVert_{\dot{\mathrm{H}}^{\alpha,\tilde{p}}(\mathbb{R}^n_+)} + \lVert g\rVert_{\dot{\mathrm{B}}^{\alpha+2-\frac{1}{\tilde{p}}}_{\tilde{p},\tilde{p}}(\mathbb{R}^{n-1})}  \text{. }
    \end{align*}
    
    The result still holds if we replace both families of spaces $(\dot{\mathrm{H}}^{s,p},\dot{\mathrm{H}}^{s+2,p},\dot{\mathrm{B}}^{s+2-\frac{1}{p}}_{p,p})$ and $(\dot{\mathrm{H}}^{\alpha,\tilde{p}},\dot{\mathrm{H}}^{\alpha+2,\tilde{p}},\dot{\mathrm{B}}^{\alpha+2-\frac{1}{\tilde{p}}}_{\tilde{p},\tilde{p}})$ by respectively $(\dot{\mathrm{B}}^{s}_{p,q},\dot{\mathrm{B}}^{s+2}_{p,q},\dot{\mathrm{B}}^{s+2-\frac{1}{p}}_{p,q})$ and $(\dot{\mathrm{B}}^{\alpha}_{\tilde{p},\tilde{q}},\dot{\mathrm{B}}^{\alpha+2}_{\tilde{p},\tilde{q}},\dot{\mathrm{B}}^{\alpha+2-\frac{1}{\tilde{p}}}_{\tilde{p},\tilde{q}})$ whenever $(\mathcal{C}_{s+2,p,q})$ is satisfied, $q<+\infty$.
    
    If $q=+\infty$, everything still holds except $x_n\mapsto u(\cdot,x_n)$ is no more strongly continuous but only weak${}^\ast$ continuous with values in $\dot{\mathrm{B}}^{s+2-\frac{1}{p}}_{p,q}(\mathbb{R}^{n-1},\mathbb{C})$.
\end{proposition}

\begin{proof} Let $p\in(1,+\infty)$, $s>-1+1/p$, such that $(\mathcal{C}_{s+2,p})$ is satisfied. Then for $f\in\dot{\mathrm{H}}^{s,p}(\mathbb{R}^n_+,\mathbb{C})$, $g\in \dot{\mathrm{B}}^{s+2-\frac{1}{p}}_{p,p}(\mathbb{R}^{n-1},\mathbb{C})$ we can write the problem \eqref{DirLap} as an evolution problem in the $x_n$ variable,
\begin{align}\label{eq:DirichletEvol}
\left\{\begin{array}{rrl}
         -\partial_{x_n}^2 u -\Delta' u   &=  f\text{,} \,&\text{ in } \mathbb{R}^{n-1}\times (0,+\infty) \text{, }\\
         u(\cdot,0)  &= g\text{,} \,&\text{ on } \mathbb{R}^{n-1} \text{. }
\end{array}
\right.
\end{align}
Thanks to \cite[Theorem~3.8.3]{ArendtBattyHieberNeubranker2011}, considering the semigroup $(e^{-x_n(-\Delta')^{1/2}})_{x_n\geqslant 0}$ and its mapping properties given by Proposition \ref{prop:HarmExtenOpSobolevBesov} and Theorem \ref{thm:Tracehalfspace}, if $f=0$, the problem \eqref{eq:DirichletEvol} admits a unique solution  $u\in \mathrm{C}^0_{0,x_n}(\overline{\mathbb{R}_+},\dot{\mathrm{B}}^{s+2-\frac{1}{p}}_{p,p}(\mathbb{R}^{n-1},\mathbb{C}))$. Thus, by linearity, we also have uniqueness of the solution $u$ in $\mathrm{C}^0_{0,x_n}(\overline{\mathbb{R}_+},\dot{\mathrm{B}}^{s+2-\frac{1}{p}}_{p,p}(\mathbb{R}^{n-1},\mathbb{C}))$ for non-identically zero function $f$. Therefore, it suffices to construct a solution.

Since $f\in \dot{\mathrm{H}}^{s,p}(\mathbb{R}^n_+,\mathbb{C})$, by definition, there exists $F\in \dot{\mathrm{H}}^{s,p}(\mathbb{R}^n,\mathbb{C})$ such that
\begin{align*}
    F_{|_{\mathbb{R}^n_+}}=f,\qquad \text{ and }\qquad \left\lVert f \right\rVert_{\dot{\mathrm{H}}^{s,p}(\mathbb{R}^n_+)} \sim \left\lVert F \right\rVert_{\dot{\mathrm{H}}^{s,p}(\mathbb{R}^n)}.
\end{align*}
Let $v:=(-\Delta)^{-1}F\in \dot{\mathrm{H}}^{s+2,p}(\mathbb{R}^n,\mathbb{C})$, this follows from the isomorphism property,  which is itself due to the point \textit{(ii)} of Proposition \ref{prop:PropertiesHomSobolevSpacesRn}, $-\Delta\,:\,\dot{\mathrm{H}}^{s+2,p}(\mathbb{R}^n,\mathbb{C})\longrightarrow \dot{\mathrm{H}}^{s,p}(\mathbb{R}^n,\mathbb{C})$ since $(\mathcal{C}_{s+2,p})$ is assumed (one may argue by density). One obtains consequently the estimates
\begin{align*}
    \left\lVert v \right\rVert_{\dot{\mathrm{H}}^{s+2,p}(\mathbb{R}^n)} \lesssim_{s,p,n} \left\lVert F \right\rVert_{\dot{\mathrm{H}}^{s,p}(\mathbb{R}^n)} \lesssim_{s,p,n} \left\lVert f \right\rVert_{\dot{\mathrm{H}}^{s,p}(\mathbb{R}^n_+)}.
\end{align*}
So it suffices to prove the result for $w\in \dot{\mathrm{H}}^{s+2,p}(\mathbb{R}^n_+,\mathbb{C})$, such that
\begin{align*}
\left\{\begin{array}{rrl}
         -\Delta w &=  0\text{,} \,&\text{ in } \mathbb{R}^{n-1}\times (0,+\infty) \text{, }\\
         w_{|_{\partial \mathbb{R}^n_+}}  &= \Tilde{g}\text{,} \,&\text{ on } \mathbb{R}^{n-1} \text{, }
\end{array}
\right.
\end{align*}
where $\Tilde{g}\in \dot{\mathrm{B}}^{s+2-\frac{1}{p}}_{p,p}(\mathbb{R}^{n-1},\mathbb{C})$ can be seen as $g-v(\cdot,0)$. But such a $w$ exists and is unique thanks to Proposition \ref{prop:HarmExtenOpSobolevBesov} and \cite[Theorem~3.8.3]{ArendtBattyHieberNeubranker2011}, and satisfies
\begin{align*}
    \left\lVert w \right\rVert_{\dot{\mathrm{H}}^{s+2,p}(\mathbb{R}^n_+)} \lesssim_{p,n,s} \left\lVert \Tilde{g} \right\rVert_{\dot{\mathrm{B}}^{s+2-\frac{1}{p}}_{p,p}(\mathbb{R}^{n-1})}.
\end{align*}
Now, we can set $u:=v+w$ which is a solution of \eqref{DirLap}, and the triangle inequality leads to
\begin{align*}
    \left\lVert u \right\rVert_{\dot{\mathrm{H}}^{s+2,p}(\mathbb{R}^n_+)} &\leqslant \,\,\,\,\quad\left\lVert v \right\rVert_{\dot{\mathrm{H}}^{s+2,p}(\mathbb{R}^n_+)}+\left\lVert w \right\rVert_{\dot{\mathrm{H}}^{s+2,p}(\mathbb{R}^n_+)} \\
    &\lesssim_{p,n,s} \left\lVert v \right\rVert_{\dot{\mathrm{H}}^{s+2,p}(\mathbb{R}^n_+)}+\left\lVert g \right\rVert_{\dot{\mathrm{B}}^{s+2-\frac{1}{p}}_{p,p}(\mathbb{R}^{n-1})} + \left\lVert v(\cdot,0) \right\rVert_{\dot{\mathrm{B}}^{s+2-\frac{1}{p}}_{p,p}(\mathbb{R}^{n-1})}\\
    &\lesssim_{p,n,s} \left\lVert f \right\rVert_{\dot{\mathrm{H}}^{s,p}(\mathbb{R}^n_+)} + \left\lVert g \right\rVert_{\dot{\mathrm{B}}^{s+2-\frac{1}{p}}_{p,p}(\mathbb{R}^{n-1})}
\end{align*}
which was the desired bound.

The Besov spaces case for $(f,g)\in\dot{\mathrm{B}}^{s}_{p,q}(\mathbb{R}^n_+,\mathbb{C})\times \dot{\mathrm{B}}^{s+2-1/p}_{p,q}(\mathbb{R}^{n-1},\mathbb{C})$, whenever $(\mathcal{C}_{s+2,p,q})$ is satisfied, follows the same lines as before, except when $q=+\infty$ where the uniqueness argument can only be checked in a weak sense since $(e^{-x_n(-\Delta')^{1/2}})_{x_n\geqslant 0}$ is only weak${}^\ast$ continuous in $\dot{\mathrm{B}}^{s+2-1/p}_{p,\infty}(\mathbb{R}^{n-1},\mathbb{C})$.

Now, if we assume that $f\in [\dot{\mathrm{H}}^{s,p}\cap\dot{\mathrm{H}}^{\alpha,\tilde{p}}](\mathbb{R}^n_+,\mathbb{C})$ and $g\in [\dot{\mathrm{B}}^{s+2-1/p}_{p,p}\cap\dot{\mathrm{B}}^{\alpha+2-1/\tilde{p}}_{\tilde{p},\tilde{p}}](\mathbb{R}^{n-1},\mathbb{C})$, then with the same notations as above, by Proposition \ref{prop:TracehalfspaceIntersec2}, we have
\begin{align*}
    v=(-\Delta)^{-1}F\in [\dot{\mathrm{H}}^{s+2,p}\cap\dot{\mathrm{H}}^{\alpha+2,\tilde{p}}](\mathbb{R}^n,\mathbb{C})\,\text{ and }\,v(\cdot,0)\in \dot{\mathrm{B}}^{s+2-1/p}_{p,p}\cap\dot{\mathrm{B}}^{\alpha+2-1/\tilde{p}}_{\tilde{p},\tilde{p}}](\mathbb{R}^{n-1},\mathbb{C})\text{. }
\end{align*}
As before, the fact that $v\in [\dot{\mathrm{H}}^{s+2,p}\cap\dot{\mathrm{H}}^{\alpha+2,\tilde{p}}](\mathbb{R}^n,\mathbb{C})$ can be obtained arguing by density, by mean of point \textit{(ii)} of Proposition \ref{prop:PropertiesHomSobolevSpacesRn} and Lemma \ref{lem:IntersecHomHsp}.
From this, one may reproduce the estimates as above to obtain
\begin{align*}
    \lVert \nabla^2 u\rVert_{\dot{\mathrm{H}}^{\alpha,\tilde{p}}(\mathbb{R}^n_+)} &\lesssim_{\tilde{p},n,\alpha} \lVert f\rVert_{\dot{\mathrm{H}}^{\alpha,\tilde{p}}(\mathbb{R}^n_+)} + \lVert g\rVert_{\dot{\mathrm{B}}^{\alpha+2-\frac{1}{\tilde{p}}}_{\tilde{p},\tilde{p}}(\mathbb{R}^{n-1})}\text{. }
\end{align*}
The case of intersection of Besov spaces follows the same lines.
\end{proof}

We state the same result for the corresponding Neumann problem, for which the proof is very close.

\begin{proposition}\label{prop:NeuPbRn+} Let $p,\tilde{p}\in(1,+\infty)$, $q,\tilde{q}\in[1,+\infty]$, $s\in(-1+\frac{1}{p},+\infty)$, $\alpha\in(-1+\frac{1}{\tilde{p}},+\infty)$, such that $(\mathcal{C}_{s+2,p})$ is satisfied.
For $f\in\dot{\mathrm{H}}^{s,p}(\mathbb{R}^n_+,\mathbb{C})$, $g\in \dot{\mathrm{B}}^{s+1-\frac{1}{p}}_{p,p}(\mathbb{R}^{n-1},\mathbb{C})$, let us consider the \textbf{Neumann} problem with inhomogeneous boundary condition:
\begin{align}\tag{$\mathcal{NL}_{0}$}\label{NeumannLap}
\left\{\begin{array}{rrl}
         -\Delta u   &=  f\text{,} \,&\text{ in } \mathbb{R}^n_+ \text{, }\\
        \partial_\nu u_{|_{\partial\mathbb{R}^n_+}}  &= g\text{,} \,&\text{ on } \partial\mathbb{R}^n_+ \text{. }
\end{array}
\right.
\end{align}

The problem \eqref{NeumannLap} admits a unique solution $u$ such that
\begin{align*}
    u \in\dot{\mathrm{H}}^{s+2,p}(\mathbb{R}^n_+,\mathbb{C})\subset \mathrm{C}^0_{0,x_n}(\overline{\mathbb{R}_+},\dot{\mathrm{B}}^{s+2-\frac{1}{p}}_{p,p}(\mathbb{R}^{n-1},\mathbb{C}))
\end{align*}
with the estimate
    \begin{align*}
        \lVert u\rVert_{\mathrm{L}^\infty(\mathbb{R}_+,\dot{\mathrm{B}}^{s+2-\frac{1}{p}}_{p,p}(\mathbb{R}^{n-1}))} \lesssim_{s,p,n}\lVert \nabla^2 u\rVert_{\dot{\mathrm{H}}^{s,p}(\mathbb{R}^n_+)} &\lesssim_{p,n,s} \lVert f\rVert_{\dot{\mathrm{H}}^{s,p}(\mathbb{R}^n_+)} + \lVert g\rVert_{\dot{\mathrm{B}}^{s+1-\frac{1}{p}}_{p,p}(\mathbb{R}^{n-1})}  \text{. }
    \end{align*}
    
    If moreover $f\in\dot{\mathrm{H}}^{\alpha,\tilde{p}}(\mathbb{R}^n_+,\mathbb{C})$ and $g\in \dot{\mathrm{B}}^{\alpha+1-\frac{1}{\tilde{p}}}_{\tilde{p},\tilde{p}}(\mathbb{R}^{n-1},\mathbb{C})$ then the solution $u$ also satisfies $u\in \dot{\mathrm{H}}^{\alpha+2,\tilde{p}}(\mathbb{R}^n_+,\mathbb{C})$ with the corresponding estimate
    \begin{align*}
        \lVert \nabla^2 u\rVert_{\dot{\mathrm{H}}^{\alpha,\tilde{p}}(\mathbb{R}^n_+)} &\lesssim_{\tilde{p},n,\alpha} \lVert f\rVert_{\dot{\mathrm{H}}^{\alpha,\tilde{p}}(\mathbb{R}^n_+)} + \lVert g\rVert_{\dot{\mathrm{B}}^{\alpha+1-\frac{1}{\tilde{p}}}_{\tilde{p},\tilde{p}}(\mathbb{R}^{n-1})}  \text{. }
    \end{align*}
    
    The result still holds, replacing $(\dot{\mathrm{H}}^{s,p},\dot{\mathrm{H}}^{s+2,p},\dot{\mathrm{B}}^{s+1-\frac{1}{p}}_{p,p},\dot{\mathrm{B}}^{s+2-\frac{1}{p}}_{p,p})$ by $(\dot{\mathrm{B}}^{s}_{p,q},\dot{\mathrm{B}}^{s+2}_{p,q},\dot{\mathrm{B}}^{s+1-\frac{1}{p}}_{p,q},\dot{\mathrm{B}}^{s+2-\frac{1}{p}}_{p,q})$ and $(\dot{\mathrm{H}}^{\alpha,\tilde{p}},\dot{\mathrm{H}}^{\alpha+2,\tilde{p}},\dot{\mathrm{B}}^{\alpha+1-\frac{1}{\tilde{p}}}_{\tilde{p},\tilde{p}})$ by  $(\dot{\mathrm{B}}^{\alpha}_{\tilde{p},\tilde{q}},\dot{\mathrm{B}}^{\alpha+2}_{\tilde{p},\tilde{q}},\dot{\mathrm{B}}^{\alpha+1-\frac{1}{\tilde{p}}}_{\tilde{p},\tilde{q}})$ whenever $(\mathcal{C}_{s+2,p,q})$ is satisfied and $q<+\infty$.
    
    If $q=+\infty$, everything still holds except $x_n\mapsto u(\cdot,x_n)$ is no more strongly continuous but only weak${}^\ast$ continuous with values in $\dot{\mathrm{B}}^{s+2-\frac{1}{p}}_{p,\infty}(\mathbb{R}^{n-1},\mathbb{C})$.
\end{proposition}

We notice that similar, but a little bit different, results of well-posedness and regularity are also available in \cite[Chapter~3]{DanchinMucha2015} with arguments of a different nature, moreover the case of Sobolev spaces and the resolvent problems does not seem to be treated.


\appendix

\section{Complex interpolation for intersection of homogeneous Besov spaces}\label{Append:CompInterpLem}

The next result is direct. Thanks to the fact that for all $a,b>0$, $\theta\in[0,1]$,
\begin{align*}
    a + a^{1-\theta}b^{\theta} \leqslant 2 (a+b)^{\theta}a^{1-\theta} \leqslant 2 (a + a^{1-\theta}b^{\theta})\text{ , }
\end{align*}
and since for $q\in[1,+\infty)$, $s_0,s_1\in \mathbb{R}$, and for $\theta \in (0,1)$, if $s= (1-\theta)s_0 + \theta s_1$, we have with equivalence of norms
\begin{align*}
    \ell^q_{s_0}(\mathbb{Z})\cap \ell^q_{s}(\mathbb{Z}) = \ell^q(\mathbb{Z}, (2^{ks_0 q} + 2^{ksq})\mathrm{d}k)=\ell^q(\mathbb{Z}, (2^{ks_0 q} + 2^{ks_1q})^{\theta} 2^{ks_0 q(1-\theta)}\mathrm{d}k)\text{ , }
\end{align*}
therefore, by complex interpolation of weighted $\mathrm{\ell}^q$ spaces, see \cite[Section~1.18.5]{bookTriebel1978}, we obtain
\begin{proposition}\label{prop:CompInterpolIntersectlpX} Let $q\in [1,+\infty)$, $s_0,s_1\in \mathbb{R}$, consider a complex Banach space $X$, and for $\theta \in (0,1)$ let's introduce $s:= (1-\theta)s_0 + \theta s_1$. The following equality holds with equivalence of norms
\begin{align*}
    [\ell^q_{s_0}(\mathbb{Z},X),\ell^q_{s_0}(\mathbb{Z},X)\cap \ell^q_{s_1}(\mathbb{Z},X)]_\theta = \ell^q_{s_0}(\mathbb{Z},X)\cap \ell^q_{s}(\mathbb{Z},X)\text{.}
\end{align*}
The result still holds with $\mathbb{N}$ instead of $\mathbb{Z}$.
\end{proposition}

The next corollary will have its importance in the proof of the next Proposition \ref{prop:HarmExtenOpSobolevBesov}.

\begin{corollary}\label{cor:CompInterpolIntersecBesov} Let $p\in[1,+\infty]$, $q\in[1,+\infty)$, $s_j\in\mathbb{R}$, $j\in\{0,1\}$ such that $(\mathcal{C}_{s_0,p,q})$ is satisfied. Then for $\theta\in (0,1)$, let's introduce $s:= (1-\theta)s_0 + \theta s_1$. Then the following equality holds with equivalence for norms
\begin{align*}
    [\dot{\mathrm{B}}^{s_0}_{p,q}(\mathbb{R}^{n}),\dot{\mathrm{B}}^{s_0}_{p,q}(\mathbb{R}^{n})\cap \dot{\mathrm{B}}^{s_1}_{p,q}(\mathbb{R}^{n})]_\theta = \dot{\mathrm{B}}^{s_0}_{p,q}(\mathbb{R}^{n})\cap \dot{\mathrm{B}}^{s}_{p,q}(\mathbb{R}^{n}) \text{ . }
\end{align*}
\end{corollary}

\begin{proof}Both function spaces $\dot{\mathrm{B}}^{s_0}_{p,q}(\mathbb{R}^{n})$, and $\dot{\mathrm{B}}^{s_0}_{p,q}(\mathbb{R}^{n})\cap \dot{\mathrm{B}}^{s_1}_{p,q}(\mathbb{R}^{n})$ are complete normed vector spaces, see \cite[Theorem~2.25]{bookBahouriCheminDanchin}.

Now, we apply \cite[Theorem~6.4.2]{BerghLofstrom1976} and Proposition \ref{prop:CompInterpolIntersectlpX}, claiming that, for all $s\in\mathbb{R}$, $\dot{\mathrm{B}}^{s}_{p,q}(\mathbb{R}^{n})$ is a retraction of $\ell^q_{s}(\mathbb{Z},\mathrm{L}^p(\mathbb{R}^n))$ through the homogeneous Littlewood-Paley decomposition $(\dot{\Delta}_j)_{j\in\mathbb{Z}}$.
\end{proof}

\section{Estimates for the Poisson semigroup}\label{Append:PoissonSemigrp}

\begin{lemma}\label{lem:DefBesovPoissonSemigrp} Let $s>0$, $\alpha \geqslant 0$ and $p,q\in[1,+\infty]^2$. For all $u\in\eus{S}'_h(\mathbb{R}^n)$,
\begin{align*}
    \lVert u \rVert_{\dot{\mathrm{B}}^{\alpha-s}_{p,q}(\mathbb{R}^n)} \sim_{p,s,\alpha,n,q} \big\lVert t\mapsto\lVert t^{s}(-\Delta)^{\frac{\alpha}{2}}e^{-t(-\Delta)^{\frac{1}{2}}}u \rVert_{\mathrm{L}^p(\mathbb{R}^n)}\big\rVert_{\mathrm{L}^q_{\ast}(\mathbb{R}_+)}\text{ .}
\end{align*}
\end{lemma}

\begin{proof} It suffices to show the result for $\alpha=0$. But in this case, the proof is straightforward the same as the one of \cite[Theorem~2.34]{bookBahouriCheminDanchin} for the heat semigroup.
\end{proof}

The following result was already proven in the case of homogeneous Besov spaces only, see \cite[Lemma~2]{DanchinMucha2009}. It is extended here to the case of homogeneous Sobolev spaces, with a new proof that also cover the case of Besov spaces.

\begin{proposition}\label{prop:HarmExtenOpSobolevBesov} Let $p\in(1,+\infty)$, $q\in[1,+\infty]$. The map
\begin{align*}
 T\,:\,f   &   \longmapsto    \left[    (x',x_n)\mapsto e^{-x_n(-\Delta')^\frac{1}{2}}f(x') \right]
\end{align*}
is such that 
\begin{enumerate}[label=($\roman*$)]
    \item Given $s\geqslant0$, for all $f\in\dot{\mathrm{B}}^{-\frac{1}{p}}_{p,p}(\mathbb{R}^{n-1})\cap \dot{\mathrm{B}}^{s-\frac{1}{p}}_{p,p}(\mathbb{R}^{n-1})$, we have
    \begin{align*}
        \lVert Tf\rVert_{\dot{\mathrm{H}}^{s,p}(\mathbb{R}^n_+)} \lesssim_{s,p,n} \lVert   f \rVert_{\dot{\mathrm{B}}^{s-\frac{1}{p}}_{p,p}(\mathbb{R}^{n-1})}\text{ . }
    \end{align*}
    In particular, $T$ extends uniquely as a bounded linear operator $T\,:\,\dot{\mathrm{B}}^{s-\frac{1}{p}}_{p,p}(\mathbb{R}^{n-1})\longrightarrow \dot{\mathrm{H}}^{s,p}(\mathbb{R}^{n}_+)$ whenever $(\mathcal{C}_{s,p})$ is satisfied.
    \item Given $s>0$, for all $f\in\dot{\mathrm{B}}^{-\frac{1}{p}}_{p,p}(\mathbb{R}^{n-1})\cap \dot{\mathrm{B}}^{s-\frac{1}{p}}_{p,q}(\mathbb{R}^{n-1})$, we have
    \begin{align*}
        \lVert Tf\rVert_{\dot{\mathrm{B}}^{s}_{p,q}(\mathbb{R}^{n}_+)} \lesssim_{s,p,n} \lVert   f \rVert_{\dot{\mathrm{B}}^{s-\frac{1}{p}}_{p,q}(\mathbb{R}^{n-1})}\text{ . }
    \end{align*}
    In particular, $T$ extends uniquely as a bounded linear operator $T\,:\,\dot{\mathrm{B}}^{s-\frac{1}{p}}_{p,q}(\mathbb{R}^{n-1})\longrightarrow \dot{\mathrm{B}}^{s}_{p,q}(\mathbb{R}^{n}_+)$ whenever \eqref{AssumptionCompletenessExponents} is satisfied.
\end{enumerate}
\end{proposition}

\begin{proof} \textbf{Point \textit{(i)}:} For $p\in(1,+\infty)$, let's consider $f\in \dot{\mathrm{B}}^{-\frac{1}{p}}_{p,p}(\mathbb{R}^{n-1})$. We apply Lemma \ref{lem:DefBesovPoissonSemigrp} to obtain,
\begin{align*}
    \left\lVert Tf\right\rVert_{\mathrm{L}^p(\mathbb{R}^n_+)} &= \left(\int_{0}^{+\infty}\lVert e^{-x_n(-\Delta')^\frac{1}{2}}f\rVert_{\mathrm{L}^p(\mathbb{R}^{n-1})}^p\mathrm{~d}x_n\right)^\frac{1}{p}\\
    &= \left(\int_{0}^{+\infty}\left(t^{\frac{1}{p}}\lVert e^{-t(-\Delta')^\frac{1}{2}}f\rVert_{\mathrm{L}^p(\mathbb{R}^{n-1})}\right)^p\frac{\mathrm{d}t}{t}\right)^\frac{1}{p}\\
    &\lesssim_{p,n} \left\lVert f \right\rVert_{\dot{\mathrm{B}}^{-\frac{1}{p}}_{p,p}(\mathbb{R}^{n-1})}\text{ . }
\end{align*}
We continue noticing that for all $f\in\eus{S}'_h(\mathbb{R}^{n-1})$, $m\in\mathbb{N}$, $\partial_{x_n}^m Tf = (-\Delta')^\frac{m}{2} Tf= T(-\Delta')^\frac{m}{2} f$ and $Tf\in \eus{S}'_h(\mathbb{R}^{n-1})$, thus if $f\in \dot{\mathrm{B}}^{-\frac{1}{p}}_{p,p}(\mathbb{R}^{n-1})\cap \dot{\mathrm{B}}^{m-\frac{1}{p}}_{p,p}(\mathbb{R}^{n-1})$ we may apply previous inequality to obtain,
\begin{align*}
    \lVert Tf\rVert_{\dot{\mathrm{H}}^{m,p}(\mathbb{R}^n_+)} &\sim_{p,n,m}  \lVert \partial_{x_n}^m Tf\rVert_{\mathrm{L}^{p}(\mathbb{R}^n_+)} + \lVert (-\Delta')^\frac{m}{2} Tf\rVert_{\mathrm{L}^{p}(\mathbb{R}^n_+)} \\
    &\sim_{p,n,m}  \lVert T (-\Delta')^\frac{m}{2} f\rVert_{\mathrm{L}^{p}(\mathbb{R}^n_+)} \\
    &\lesssim_{p,n,m} \lVert f \rVert_{\dot{\mathrm{B}}^{m-\frac{1}{p}}_{p,p}(\mathbb{R}^{n-1})}\text{ . }
\end{align*}
So that for all $m\in \mathbb{N}$, all $f\in\dot{\mathrm{B}}^{-\frac{1}{p}}_{p,p}(\mathbb{R}^{n-1})\cap\dot{\mathrm{B}}^{m-\frac{1}{p}}_{p,p}(\mathbb{R}^{n-1})$,
\begin{align*}
    \lVert Tf\rVert_{{\mathrm{H}}^{m,p}(\mathbb{R}^n_+)} &\lesssim_{p,n,m} \lVert f \rVert_{\dot{\mathrm{B}}^{-\frac{1}{p}}_{p,p}(\mathbb{R}^{n-1})} + \lVert f \rVert_{\dot{\mathrm{B}}^{m-\frac{1}{p}}_{p,p}(\mathbb{R}^{n-1})}\text{ . }
\end{align*}
Thus, by complex interpolation and Corollary \ref{cor:CompInterpolIntersecBesov}, for all $s\geqslant 0$, all $f\in\dot{\mathrm{B}}^{-\frac{1}{p}}_{p,p}(\mathbb{R}^{n-1})\cap\dot{\mathrm{B}}^{s-\frac{1}{p}}_{p,p}(\mathbb{R}^{n-1})$,
\begin{align*}
    \lVert Tf\rVert_{{\mathrm{H}}^{s,p}(\mathbb{R}^n_+)} &\lesssim_{p,n,s} \lVert f \rVert_{\dot{\mathrm{B}}^{-\frac{1}{p}}_{p,p}(\mathbb{R}^{n-1})} + \lVert f \rVert_{\dot{\mathrm{B}}^{s-\frac{1}{p}}_{p,p}(\mathbb{R}^{n-1})}\text{ . }
\end{align*}
Hence, thanks to Proposition \ref{prop:IntersecHomHspRn+}, ${{\mathrm{H}}^{s,p}(\mathbb{R}^n_+)} = {\mathrm{L}}^{p}(\mathbb{R}^n_+)\cap\dot{\mathrm{H}}^{s,p}(\mathbb{R}^n_+)$,
\begin{align*}
    \lVert Tf\rVert_{{\mathrm{L}}^{p}(\mathbb{R}^n_+)} + \lVert Tf\rVert_{\dot{\mathrm{H}}^{s,p}(\mathbb{R}^n_+)} &\lesssim_{p,n,s} \lVert f \rVert_{\dot{\mathrm{B}}^{-\frac{1}{p}}_{p,p}(\mathbb{R}^{n-1})} + \lVert f \rVert_{\dot{\mathrm{B}}^{s-\frac{1}{p}}_{p,p}(\mathbb{R}^{n-1})}\text{ . }
\end{align*}
Therefore, if $\lambda\in 2^{\mathbb{N}}$, we can consider $f_\lambda$ the dilation by factor $\lambda$ of $f$, so that plugging $f_\lambda$ instead of $f$ in above inequality, and checking the fact that $Tf_\lambda = (Tf)_\lambda$, we obtain
\begin{align*}
    \lambda^{-\frac{n}{p}}\lVert Tf\rVert_{{\mathrm{L}}^{p}(\mathbb{R}^n_+)} + \lambda^{s-\frac{n}{p}}\lVert Tf\rVert_{\dot{\mathrm{H}}^{s,p}(\mathbb{R}^n_+)} &\lesssim_{p,n,s} \lambda^{-\frac{n}{p}}\lVert f \rVert_{\dot{\mathrm{B}}^{-\frac{1}{p}}_{p,p}(\mathbb{R}^{n-1})} + \lambda^{s-\frac{n}{p}}\lVert f \rVert_{\dot{\mathrm{B}}^{s-\frac{1}{p}}_{p,p}(\mathbb{R}^{n-1})}\text{ . }
\end{align*}
One may divide above inequality by $\lambda^{s-\frac{n}{p}}$, so as $\lambda$ tends to infinity, it yields
\begin{align*}
     \lVert Tf\rVert_{\dot{\mathrm{H}}^{s,p}(\mathbb{R}^n_+)} &\lesssim_{p,n,s} \lVert f \rVert_{\dot{\mathrm{B}}^{s-\frac{1}{p}}_{p,p}(\mathbb{R}^{n-1})}\text{ . }
\end{align*}
So that the result holds by density whenever $(\mathcal{C}_{s,p})$ is satisfied.

\textbf{Point \textit{(ii)}:} Now let $q\in[1,+\infty]$, since for all $s\geqslant 0$, all $f\in\dot{\mathrm{B}}^{-\frac{1}{p}}_{p,p}(\mathbb{R}^{n-1})\cap\dot{\mathrm{B}}^{s-\frac{1}{p}}_{p,p}(\mathbb{R}^{n-1})$,
\begin{align*}
    \lVert Tf\rVert_{{\mathrm{H}}^{s,p}(\mathbb{R}^n_+)} &\lesssim_{p,n,s} \lVert f \rVert_{\dot{\mathrm{B}}^{-\frac{1}{p}}_{p,p}(\mathbb{R}^{n-1})} + \lVert f \rVert_{\dot{\mathrm{B}}^{s-\frac{1}{p}}_{p,p}(\mathbb{R}^{n-1})}\text{ . }
\end{align*}
Hence, by real interpolation, using \cite[Proposition~B.2.7]{bookHaase2006} instead of Corollary \ref{cor:CompInterpolIntersecBesov}, we obtain that for all $s>0$
all $f\in\dot{\mathrm{B}}^{-\frac{1}{p}}_{p,p}(\mathbb{R}^{n-1})\cap\dot{\mathrm{B}}^{s-\frac{1}{p}}_{p,q}(\mathbb{R}^{n-1})$,
\begin{align*}
    \lVert Tf\rVert_{{\mathrm{B}}^{s}_{p,q}(\mathbb{R}^n_+)} &\lesssim_{p,n,s} \lVert f \rVert_{\dot{\mathrm{B}}^{-\frac{1}{p}}_{p,p}(\mathbb{R}^{n-1})} + \lVert f \rVert_{\dot{\mathrm{B}}^{s-\frac{1}{p}}_{p,q}(\mathbb{R}^{n-1})}\text{ . }
\end{align*}
Then the same dilation procedure as before, yields 
\begin{align*}
    \lVert Tf\rVert_{\dot{\mathrm{B}}^{s}_{p,q}(\mathbb{R}^n_+)} &\lesssim_{p,n,s}\lVert f \rVert_{\dot{\mathrm{B}}^{s-\frac{1}{p}}_{p,q}(\mathbb{R}^{n-1})}\text{ , }
\end{align*}
which again allows concluding via a density argument if $q<+\infty$ and \eqref{AssumptionCompletenessExponents} is satisfied. The case $q=+\infty$, when \eqref{AssumptionCompletenessExponents} is satisfied,  follows from real interpolation with the last estimate.
\end{proof}

Proposition \ref{prop:HarmExtenOpSobolevBesov} can be self-improved as

\begin{corollary}\label{cor:HarmExtenOpSobolevBesovIntersec} Let $p_j\in(1,+\infty)$, $q_j\in[1,+\infty]$, $j\in\{ 0,1\}$. The map
\begin{align*}
 T\,:\,f   &   \longmapsto    \left[    (x',x_n)\mapsto e^{-x_n(-\Delta')^\frac{1}{2}}f(x') \right]
\end{align*}
is such that 
\begin{enumerate}[label=($\roman*$)]
    \item If $s_j\geqslant0$, $j\in\{0,1\}$, such that $(\mathcal{C}_{s_0,p_0})$ is satisfied. For all $f\in[\dot{\mathrm{B}}^{s_0-\frac{1}{p_0}}_{p_0,p_0}\cap \dot{\mathrm{B}}^{s_1-\frac{1}{p_1}}_{p_1,p_1}](\mathbb{R}^{n-1})$, we have 
    \begin{align*}
        \lVert Tf\rVert_{\dot{\mathrm{H}}^{s_j,p_j}(\mathbb{R}^n_+)} \lesssim_{s_j,p_j,n} \lVert   f \rVert_{\dot{\mathrm{B}}^{s_j-\frac{1}{p_j}}_{p_j,p_j}(\mathbb{R}^{n-1})}\quad\text{ , } j\in\{0,1\}\text{ . }
    \end{align*}
    \item If $s_j>0$, $j\in\{0,1\}$, such that $(\mathcal{C}_{s_0,p_0,q_0})$ is satisfied. For all $f\in[\dot{\mathrm{B}}^{s_0-\frac{1}{p_0}}_{p_0,q_0}\cap \dot{\mathrm{B}}^{s_1-\frac{1}{p_1}}_{p_1,q_1}](\mathbb{R}^{n-1})$, we have
    \begin{align*}
        \lVert Tf\rVert_{\dot{\mathrm{B}}^{s_j}_{p_j,q_j}(\mathbb{R}^{n}_+)} \lesssim_{s_j,p_j,n} \lVert   f \rVert_{\dot{\mathrm{B}}^{s_j-\frac{1}{p_j}}_{p_j,q_j}(\mathbb{R}^{n-1})}\quad\text{ , } j\in\{0,1\}\text{ . }
    \end{align*}
\end{enumerate}
\end{corollary}

\section{Weak* continuity of translations for endpoint Besov spaces }\label{Append:ContinTranslation}

\begin{lemma}\label{lem:weak*continuityBspinfty}Let $p\in(1,+\infty)$, $s\in\mathbb{R}$. Then for all $t\in\mathbb{R}$, the map
\begin{align*}
    \tau_t \,:\, u \longmapsto \left[ (x',x_n)\longmapsto u(x',x_n+t)\right]
\end{align*}
is uniformly bounded on ${\mathrm{B}}^{s}_{p,\infty}(\mathbb{R}^n)$, that is for all $u\in{\mathrm{B}}^{s}_{p,\infty}(\mathbb{R}^n)$
    \begin{align*}
        \lVert \tau_tu\rVert_{{\mathrm{B}}^{s}_{p,\infty}(\mathbb{R}^n)} \leqslant \lVert   u \rVert_{{\mathrm{B}}^{s}_{p,\infty}(\mathbb{R}^n)}\quad\text{, } \forall t\in\mathbb{R}\text{,}
    \end{align*}
    and has a limit
    \begin{align*}
        \tau_t u \xrightarrow[t\rightarrow 0]{} u \text{, weakly* in }{\mathrm{B}}^{s}_{p,\infty}(\mathbb{R}^n).
    \end{align*}

$\bullet$ Assuming $s<n/p$, the results still hold with $\dot{\mathrm{B}}$ instead of $\mathrm{B}$.

$\bullet$ The whole result still holds with $\mathbb{R}^n_+$ instead of $\mathbb{R}^n$, with $t\geqslant 0$, assuming $s\in(-1+1/p,n/p)$ in the case of homogeneous Besov spaces.
\end{lemma}

\begin{proof} We only treat the case of homogeneous function spaces, the case of inhomogeneous function spaces can be treated with the same arguments.

\textbf{Step 1:} The case of the whole space $\mathbb{R}^n$.  The proof relies on the fact that, by Proposition \ref{prop:dualityHomBesov}, $\dot{\mathrm{B}}^{s}_{p,\infty}(\mathbb{R}^n)=(\dot{\mathrm{B}}^{-s}_{p',1}(\mathbb{R}^n))'= (\dot{\mathcal{B}}^{s}_{p,\infty}(\mathbb{R}^n))''$.

Let $t\in\mathbb{R}$, $\phi\in\dot{\mathrm{B}}^{-s}_{p',1}(\mathbb{R}^n)$, for $u\in \dot{\mathrm{B}}^{s}_{p,\infty}(\mathbb{R}^n)$, one obtains by strong continuity of translation on $\dot{\mathrm{B}}^{-s}_{p',1}(\mathbb{R}^n)$, see \cite[Proposition~1.9]{Guidetti1991Interp} (one may use instead an interpolation and density argument),
\begin{align*}
    \langle \tau_tu, \phi\rangle_{\mathbb{R}^n} = \langle u, \tau_{-t}\phi\rangle_{\mathbb{R}^n} \xrightarrow[t\rightarrow 0]{} \langle u, \phi\rangle_{\mathbb{R}^n}.
\end{align*}
The norm estimate follows from the fact that for all $t\in\mathbb{R}$, $j\in\mathbb{Z}$, $\dot{\Delta}_j\tau_t u=\tau_t\dot{\Delta}_j u$ in $\eus{S}'(\mathbb{R}^n)$, so that by translation invariance of Lebesgue norms,
\begin{align*}
    \lVert \tau_tu\rVert_{\dot{\mathrm{B}}^{s}_{p,\infty}(\mathbb{R}^n)} = \sup_{j\in\mathbb{Z}}\, 2^{js}\lVert \tau_t\dot{\Delta}_j u \rVert_{\mathrm{L}^p(\mathbb{R}^n)} = \sup_{j\in\mathbb{Z}}\, 2^{js}\lVert \dot{\Delta}_j u \rVert_{\mathrm{L}^p(\mathbb{R}^n)} =\lVert u\rVert_{\dot{\mathrm{B}}^{s}_{p,\infty}(\mathbb{R}^n)}.
\end{align*}

\textbf{Step 2:} Now, the case of the half-space $\mathbb{R}^n_+$. We show the estimate first. Let $t\geqslant 0$ and $u\in \dot{\mathrm{B}}^{s}_{p,\infty}(\mathbb{R}^n_+)$, we choose an arbitrary extension $U\in \dot{\mathrm{B}}^{s}_{p,\infty}(\mathbb{R}^n)$ of $u$, \textit{i.e.} such that $U_{|_{\mathbb{R}^n_+}}=u$, then $\tau_tU$ is an extension of $\tau_t u$. From this point, we deduce the following estimate by the definition of function spaces by restriction
\begin{align*}
    \lVert \tau_tu\rVert_{\dot{\mathrm{B}}^{s}_{p,\infty}(\mathbb{R}^n_+)}\leqslant \lVert \tau_tU\rVert_{\dot{\mathrm{B}}^{s}_{p,\infty}(\mathbb{R}^n)} = \lVert U\rVert_{\dot{\mathrm{B}}^{s}_{p,\infty}(\mathbb{R}^n)}.
\end{align*}
One takes the infimum over all such $U$ to obtain
\begin{align*}
    \lVert \tau_tu\rVert_{\dot{\mathrm{B}}^{s}_{p,\infty}(\mathbb{R}^n_+)}\leqslant \lVert u\rVert_{\dot{\mathrm{B}}^{s}_{p,\infty}(\mathbb{R}^n_+)}.
\end{align*}
Now, we proceed as in \textbf{Step 1} using the fact that $\dot{\mathrm{B}}^{s}_{p,\infty}(\mathbb{R}^n_+) = (\dot{\mathrm{B}}^{-s}_{p',1,0}(\mathbb{R}^n_+))'$, see Proposition \ref{prop:dualityBesovRn+}. Let $\phi\in \dot{\mathrm{B}}^{-s}_{p',1,0}(\mathbb{R}^n_+)$, in particular, we have $\phi\in \dot{\mathrm{B}}^{-s}_{p',1}(\mathbb{R}^n)$ with $\supp \phi \subset \overline{\mathbb{R}^n_+}$. Then for an arbitrary extension $U\in \dot{\mathrm{B}}^{s}_{p,\infty}(\mathbb{R}^n)$ of $u$, we have $(\tau_tU)_{|_{\mathbb{R}^n_+}}=\tau_{t} u$ for all $t\geqslant 0$. Therefore, by the previous \textbf{Step 1}, using the same arguments one can find in the proof of Proposition \ref{prop:dualityBesovRn+}, it follows that
\begin{align*}
    \langle \tau_tu, \phi\rangle_{\mathbb{R}^n_+} = \langle \tau_t U, \phi\rangle_{\mathbb{R}^n} \xrightarrow[t\rightarrow 0]{} \langle U, \phi\rangle_{\mathbb{R}^n} = \langle u, \phi\rangle_{\mathbb{R}^n_+},
\end{align*}
which is the desired result.
\end{proof}

\typeout{}                                
\bibliographystyle{alpha}
{\footnotesize
\bibliography{Biblio}}

\begin{thebibliography}{DHMT21}

\bibitem[AA18]{AmentaAuscher2018}
A.~Amenta and P.~Auscher.
\newblock {\em {Elliptic Boundary Value Problems with Fractional Regularity
  Data: The First Order Approach}}, volume~37 of {\em CRM Monograph Series}.
\newblock Amer. Math. Soc., 2018.

\bibitem[ABHN11]{ArendtBattyHieberNeubranker2011}
W.~Arendt, C.~J.~K. Batty, M.~Hieber, and F.~Neubrander.
\newblock {\em Vector-valued {L}aplace transforms and {C}auchy problems},
  volume~96 of {\em Monographs in Mathematics}.
\newblock Birkh\"{a}user/Springer Basel AG, Basel, 2${}^{\text{nd}}$ edition,
  2011.

\bibitem[BCD11]{bookBahouriCheminDanchin}
H.~Bahouri, J.-Y. Chemin, and R.~Danchin.
\newblock {\em Fourier analysis and nonlinear partial differential equations},
  volume 343 of {\em Grundlehren der Mathematischen Wissenschaften [Fundamental
  Principles of Mathematical Sciences]}.
\newblock Springer, Heidelberg, 2011.

\bibitem[BL76]{BerghLofstrom1976}
J.~Bergh and J.~L\"{o}fstr\"{o}m.
\newblock {\em Interpolation spaces. {A}n introduction}.
\newblock Grundlehren der Mathematischen Wissenschaften, No. 223.
  Springer-Verlag, Berlin-New York, 1976.

\bibitem[Bou88]{Bourdaud1988}
G.~Bourdaud.
\newblock R{\'e}alisations des espaces de {Besov} homog{\`e}nes. ({Realization}
  of homogeneous {Besov} spaces).
\newblock {\em Ark. Mat.}, 26(1):41--54, 1988.

\bibitem[Bou13]{Bourdaud2013}
G.~Bourdaud.
\newblock Realizations of homogeneous {Besov} and {Lizorkin}-{Triebel} spaces.
\newblock {\em Math. Nachr.}, 286(5-6):476--491, 2013.

\bibitem[Cob21]{Cobb2021}
D.~Cobb.
\newblock {Bounded Solutions in Incompressible Hydrodynamics}.
\newblock {\em arXiv e-prints}, page arXiv:2105.03257v1, may 2021.

\bibitem[Cob23]{Cobb2022}
D.~Cobb.
\newblock {R}emarks on {C}hemin's space of homogeneous distributions.
\newblock {\em Math. Nachr.}, 2023.
\newblock DOI: 10.1002/mana.202200293, Open Access.

\bibitem[DHMT21]{DanchinHieberMuchaTolk2020}
R.~{Danchin}, M.~{Hieber}, P.B. {Mucha}, and P.~{Tolksdorf}.
\newblock {Free Boundary Problems via Da Prato-Grisvard Theory}.
\newblock {\em arXiv e-prints}, November 2021.
\newblock arXiv:2011.07918v2.

\bibitem[DM09]{DanchinMucha2009}
R.~Danchin and P.~B. Mucha.
\newblock A critical functional framework for the inhomogeneous
  {N}avier-{S}tokes equations in the half-space.
\newblock {\em J. Funct. Anal.}, 256(3):881--927, 2009.

\bibitem[DM15]{DanchinMucha2015}
R.~Danchin and P.~B. Mucha.
\newblock Critical functional framework and maximal regularity in action on
  systems of incompressible flows.
\newblock {\em Mémoires de la {S}ociété {M}athématique de France}, 143,
  2015.

\bibitem[Ege15]{EgertPhDThesis2015}
M.~Egert.
\newblock {\em On {K}ato's conjecture and mixed boundary conditions}.
\newblock Sierke {V}erlag, G\"{o}ttingen, 2015.
\newblock {P}h{D} {T}hesis, {T}echnischen Universit\"{a}t {D}armstadt.

\bibitem[Gau23]{GaudinThesis2023}
A.~Gaudin.
\newblock {\em {Homogeneous Sobolev and Besov spaces on half-spaces}}.
\newblock {PhD Thesis}, {Aix-Marseille Université}, July 2023.

\bibitem[Gra14a]{bookGrafakos2014Classical}
L.~Grafakos.
\newblock {\em Classical {F}ourier analysis}, volume 249 of {\em Graduate Texts
  in Mathematics}.
\newblock Springer, New York, 3${}^{\text{rd}}$ edition, 2014.

\bibitem[Gra14b]{bookGrafakos2014Modern}
L.~Grafakos.
\newblock {\em Modern {F}ourier analysis}, volume 250 of {\em Graduate Texts in
  Mathematics}.
\newblock Springer, New York, 3${}^{\text{rd}}$ edition, 2014.

\bibitem[Gui91]{Guidetti1991Interp}
D.~Guidetti.
\newblock On interpolation with boundary conditions.
\newblock {\em Math. Z.}, 207(3):439--460, 1991.

\bibitem[Haa06]{bookHaase2006}
M.~Haase.
\newblock {\em The functional calculus for sectorial operators}, volume 169 of
  {\em Operator Theory: Advances and Applications}.
\newblock Birkh\"{a}user Verlag, Basel, 2006.

\bibitem[IMT19]{IwabuchiMatsuyamaTaniguchi2019}
T.~Iwabuchi, T.~Matsuyama, and K.~Taniguchi.
\newblock Besov spaces on open sets.
\newblock {\em Bull. Sci. Math.}, 152:93--149, 2019.

\bibitem[Jaw78]{Jawerth78}
Bj\"{o}rn Jawerth.
\newblock The trace of {S}obolev and {B}esov spaces if {$0<p<1$}.
\newblock {\em Studia Math.}, 62(1):65--71, 1978.

\bibitem[JK95]{JerisonKenig1995}
D.~Jerison and C.~E. Kenig.
\newblock The inhomogeneous {D}irichlet problem in {L}ipschitz domains.
\newblock {\em J. Funct. Anal.}, 130(1):161--219, 1995.

\bibitem[JW84]{JonssonWallin1984}
A.~Jonsson and H.~Wallin.
\newblock {\em {Function {S}paces on {S}ubsets of
  \texorpdfstring{${\mathbb{R}}^n$}{Rn}}}, volume~2.
\newblock Harwood Academic Publishers, 1984.

\bibitem[KMM07]{KaltonMayborodaMitrea2007}
N.~Kalton, S.~Mayboroda, and M.~Mitrea.
\newblock {I}nterpolation of {H}ardy-{S}obolev-{B}esov-{T}riebel {L}izorkin
  spaces and applications to problems in partial differential equations.
\newblock In {\em Interpolation Theory and Applications (Contemporary
  Mathematics 445)}, pages 121--177, 2007.

\bibitem[LM72]{LionsMagenes1972}
J.~L. Lions and E.~Magenes.
\newblock {\em Non-homogeneous boundary value problems and applications. {Vol}.
  {I}.}, volume 181 of {\em Grundlehren Math. Wiss.}
\newblock Springer, Cham, 1972.
\newblock {Translated} from the {French} by {P}. {Kenneth}.

\bibitem[Lun18]{bookLunardiInterpTheory}
A.~Lunardi.
\newblock {\em Interpolation theory}.
\newblock Appunti. Scuola Normale Superiore di Pisa (Nuova Serie). [Lecture
  Notes. Scuola Normale Superiore di Pisa (New Series)]. Edizioni della
  Normale, Pisa, 3${}^{\text{rd}}$ edition, 2018.

\bibitem[Ouh05]{bookOuhabaz2005}
E.~M. Ouhabaz.
\newblock {\em Analysis of heat equations on domains}, volume~31 of {\em London
  Mathematical Society Monographs Series}.
\newblock Princeton University Press, Princeton, NJ, 2005.

\bibitem[Ryc99]{Rychkov1999}
V.S. Rychkov.
\newblock On {R}estrictions and {E}xtensions of the {B}esov and
  {T}riebel–{L}izorkin spaces with {R}espect to {L}ipschitz {D}omains.
\newblock {\em Journal of The London Mathematical Society}, 60:237--257, 08
  1999.

\bibitem[Saw18]{bookSawano2018}
Y.~Sawano.
\newblock {\em Theory of {B}esov spaces}, volume~56 of {\em Developments in
  Mathematics}.
\newblock Springer, Singapore, 2018.

\bibitem[Sch10]{Schneider2010}
C.~Schneider.
\newblock Trace {O}perators in {B}esov and {T}riebel–{L}izorkin {S}paces.
\newblock {\em Z. Anal. Anwend.}, 29(1):275--302, 2010.

\bibitem[Ste70]{Stein1970}
E.M. Stein.
\newblock {\em Singular Integrals and Differentiability Properties of
  Functions}.
\newblock Monographs in Harmonic Analysis. Princeton University Press, 1970.

\bibitem[Str67]{Strichartz1967}
R.S. Strichartz.
\newblock Multipliers on {F}ractional {S}obolev {S}paces.
\newblock {\em Journal of Mathematics and Mechanics}, 16(9):1031--1060, 1967.

\bibitem[Tri78]{bookTriebel1978}
H.~Triebel.
\newblock {\em Interpolation Theory, Function Spaces, Differential Operators}.
\newblock North-Holland Mathematical Library. Elsevier, 1978.

\bibitem[Tri83]{bookTriebel1983}
H.~Triebel.
\newblock {\em Theory of {Function Spaces}}.
\newblock Modern Birkhäuser Classics. Birkh\"{a}user Basel, 1983.

\bibitem[Tri92]{bookTriebel1992}
H.~Triebel.
\newblock {\em Theory of {Function Spaces II}}.
\newblock Modern Birkhäuser Classics. Birkh\"{a}user Basel, 1992.

\bibitem[Tri15]{Triebel2015}
H.~Triebel.
\newblock {\em Tempered homogeneous function spaces}.
\newblock EMS Ser. Lect. Math. Z{\"u}rich: European Mathematical Society (EMS),
  2015.

\end{thebibliography}

\end{document}